\documentclass[a4paper, leqno]{article}
\usepackage[english]{babel}
\usepackage{amsmath,amscd,amssymb,amsthm}
\usepackage{physics}
\usepackage{pb-diagram}
\usepackage{lineno}
\numberwithin{equation}{section}
\usepackage{graphics} 
\usepackage{bm}
\usepackage{tikz} 
\usepackage{verbatim}
\usepackage{centernot,cancel}
\newtheorem{de}{Definition}[section]

\newtheorem{thm}{Theorem}[section]

\newtheorem{lem}{Lemma}[section]

\newtheorem{prop}{Proposition}[section]

\newtheorem{q}{Question}[section]
\newtheorem{remark}{Remark}[section]
\newcommand{\be}{\begin{equation}}
\newcommand{\ee}{\end{equation}}
\newcommand{\bpr}{\begin{proof}}
\newcommand{\epr}{\end{proof}}
\newcommand{\baln}{\begin{align*}}
\newcommand{\ealn}{\end{align*}}
\begin{document}
\title{Manton's Exotic Vortex Equation}
\author{Takashi Ono\thanks{Research Institute for Mathematical Sciences, Kyoto University, Kyoto, Japan, takashio@kurims.kyoto-u.ac.jp}}
\date{}
\maketitle
\begin{abstract}
In this paper, we study the equation which we call \textit{Manton's exotic vortex equation}. The equation arose in Manton's study of a deformation of the classical vortex equation \cite{Ma}. We first investigate the existence problem for Manton's exotic vortex equation and establish existence results.
 We then construct the moduli space of solutions.
 Finally, we establish a dimensional reduction for Manton's exotic vortex equation.
 More precisely, we establish a correspondence between solutions of Manton's exotic vortex equation and pseudo-Hermitian-Yang--Mills connections with pseudo-Hermitian metrics of signature $(1,1)$.
\end{abstract}
\noindent
Keywords: Vortex equation, pseudo-Hermitian-Yang-Mills connection\\
2020 MSC: 53C07, 14D21, 58D27
\section*{Introduction}
\subsection*{Manton's Exotic Vortex Equation}
In this paper, we study the equation which we call \textit{Manton's exotic vortex equation}.
The equation arose in Manton's study of a deformation of the classical vortex equation \cite{Ma}.
Before we introduce Manton's exotic vortex equation, we first recall the vortex equation.
\par
Let $X$ be a compact Riemann surface with a K\"ahler form $\omega_X$.
Let $(L,\overline\partial_L)$ be a holomorphic line bundle over $X$, let $\phi$ be a holomorphic section of $(L,\overline\partial_L)$, and let $h$ be a Hermitian metric of $L$.
We denote as $\nabla_h$ the Chern connection with respect to $\overline\partial_L,h$, and as $F_h$ the curvature of $\nabla_h$.
Let $\tau>0$ be a positive constant.
We say that the triple $((L,\overline\partial_L),\phi,h)$ satisfies the $\tau$-vortex equation if 
\[
\sqrt{-1}\Lambda_{\omega_X}F_{h}+\frac12|\phi|^2_h-\frac{\tau}{2}=0.
\]
Here, $\Lambda_{\omega_X}$ is the contraction operator with respect to $\omega_X$.
The vortex equation was introduced by Ginzburg and Landau \cite{GL}, and has since been extensively studied in both mathematics and physics.
 Since it is impossible to give an exhaustive list of the literature, we mention only a few references \cite{JT, Br1, GP1, GP3}.
 \par
In \cite{Ma}, Manton deformed the vortex equation by changing the sign of the coefficients.
More precisely, he introduced the family of equations
\[
\sqrt{-1}\Lambda_{\omega_X}F_{h}+C_0\frac12|\phi|^2_h-C_1\frac{\tau}{2}=0,\quad C_0,C_1\in\{-1,0,1\}.
\]
Note that we still assume $\tau>0$.
Although there are nine possible pairs $(C_0,C_1)$, only the following five pairs
\[
(C_0,C_1)=(1,1),(-1,1),(-1,0),(-1,-1),(0,1)
\]
can occur.
We will explain why the other four pairs cannot occur in Section \ref{sec 1.1}.
Among the five possible pairs $(C_0,C_1)$, the $(C_0,C_1)=(1,1)$ case is the classical vortex equation and the $(C_0,C_1)=(0,1)$ case is the Hermitian-Yang-Mills equation.
While both cases are extensively studied, the remaining three cases  $(C_0,C_1)=(-1,-1),(-1,0),(-1,1)$ are relatively unexplored in the mathematical literature, and are what we are interested in.
\begin{remark}
Although the three cases  $(C_0,C_1)=(-1,-1),(-1,0),(-1,1)$ are relatively unexplored in the mathematical literature, they have been studied in the physics literature.
The $(C_0,C_1)=(-1,-1)$ case was studied by Popov \cite{Po1, Po2}, the $(C_0,C_1)=(-1,0)$ case was studied by Jackiw-Pi \cite{JaPi1, JaPi2}, and the $(C_0,C_1)=(-1,1)$ case was studied by Ambj{\o}rn-Olesen \cite{AmOl1, AmOl2}.
\end{remark}

Since we assume throughout that $\tau>0,$ and since $C_1$ takes the values $-1, 0,$ and $1$ for the three cases of interest, these three equations can be unified by allowing $\tau$ to vary over all real numbers.
This leads to the following definition of Manton's exotic $\tau$-vortex equation.
\begin{de}
Let $(X,\omega_X)$ be a compact Riemann surface with a K\"ahler form $\omega_X$.
Let $(L,\overline\partial_L)$ be a holomorphic line bundle over $X$, let $\phi$ be a holomorphic section of $(L,\overline\partial_L)$, and let $h$ be a Hermitian metric on $L$.
Let $\tau\in\mathbb{R}$ be a constant.\par
We say that the triple $((L,\overline\partial_L),\phi,h)$ satisfies Manton's exotic $\tau$-vortex equation if 
\[
\sqrt{-1}\Lambda_{\omega_X}F_{h}-\frac12|\phi|^2_h-\frac{\tau}{2}=0.
\]
\end{de}
In Section \ref{sec 1}, we study the existence problem for solutions of Manton's exotic vortex equation.
More precisely, given a holomorphic line bundle $(L,\overline{\partial}_L)$ and a holomorphic section $\phi,$ we ask whether there exists a Hermitian metric $h$ on $L$ such that the triple $((L,\overline{\partial}_L),\phi,h)$ satisfies Manton's exotic vortex equation.
To study this existence problem, we follow Bradlow's approach \cite{Br1} and apply a theorem of Kazdan and Warner \cite{KW}.
In particular, we show the following.
\begin{thm}[Theorem \ref{thm 1.2}]\label{thm 0.1}
Let $(X,\omega_X)$ be a compact Riemann surface with the K\"ahler form $\omega_X$.
We normalize $\omega_X$ so that $\mathrm{Vol}_X:=\int_X\omega_X=1.$
\begin{itemize}
\item[(1)] Let $\tau\in\mathbb{R}$ be a constant.
Let  $(L,\overline\partial_L)$ be a holomorphic line bundle over $X$, $\phi$ be a holomorphic section of $(L,\overline\partial_L)$, and $h$ be a Hermitian metric of $L$.
Let $v_{h}$ be the function defined in Section \ref{sec 1.3}.\par
If $\tau$ satisfies 
\[
4\pi\mathrm{deg}L-c(|\phi|^2_h e^{2v_h})<\tau<4\pi\mathrm{deg}L,
\]
there exists a Hermitian metric $k$ such that the triple $((L,\overline{\partial}_L),\phi,k)$ satisfies Manton's exotic $\tau$-vortex equation.
See Section \ref{sec 1.3} for details on the constant  $c(|\phi|^2_h e^{2v_h})$.
\item[(2)] For every holomorphic line bundle $(L,\overline\partial_L)$ over $X$ and for every holomorphic section $\phi$ of $(L,\overline\partial_L)$, if we take $\tau$ appropriately, there exists a Hermitian metric $h$ such that the triple $((L,\overline{\partial}_L),\phi,h)$ satisfies Manton's exotic $\tau$-vortex equation.
\end{itemize}
\end{thm}
In Section \ref{sec 2}, we construct the moduli space of solutions of Manton's exotic $\tau$-vortex equation.
 By Theorem \ref{thm 0.1}, the moduli space is non-empty for suitable values of $\tau$.
We construct the moduli space in the Sobolev setting.
We denote by $\mathcal{M}_\tau$ the moduli space of solutions of Manton's exotic $\tau$-vortex equation.
See Section \ref{sec 2.1} and \ref{sec 2.2} for the definition of $\mathcal{M}_\tau$.
We show the following.
\begin{thm}[Theorem \ref{thm 2.1}]
Let $\mathcal{M}^\ast_\tau$ be a subset of $\mathcal{M}_\tau$ that parametrizes irreducible and regular solutions of Manton's exotic $\tau$-vortex equation.
See Section \ref{sec 2.2} for the definition of irreducible and regular solutions.
Then, $\mathcal{M}^\ast_\tau$ is a smooth manifold of real dimension $2\mathrm{deg}L$.
\end{thm}
\subsection*{Dimensional Reduction}
In general, dimensional reduction is a procedure for constructing a solution of a gauge-theoretic equation from a symmetric solution of a gauge-theoretic equation on a higher-dimensional base space.
It therefore provides a natural way to relate different gauge-theoretic equations.
\par
In \cite{GP1}, Garc\'ia-Prada established a dimensional reduction for the vortex equation on a compact K\"ahler manifold $M$.
He showed that solutions of the vortex equation on $M$ can be obtained from $SU(2)$-equivariant Hermitian-Einstein metrics on $M\times \mathbb{P}^1$.
Here $\mathbb{P}^1$ is the projective line.
\par
In the case of Manton's exotic vortex equation, Contatto and Dunajski also obtained a dimensional reduction in \cite{CD}.
We would like to explain their result since it was the main motivation of this work.\par
Let $X$ be a compact Riemann surface.
Contatto and Dunajski showed that solutions of Manton's exotic $\tau$-vortex equation can be obtained from anti-self-dual connections with symmetry.
Recall that the parameter $\tau$ is allowed to take any real value. 
Accordingly, the corresponding dimensional reduction depends on whether $\tau>0$, $\tau=0$, or $\tau<0$.
 For $\tau>0$, it is obtained from $SU(2)$-equivariant anti-self-dual connections on $X\times\mathbb{P}^1$; for $\tau<0$, from $SU(1,1)$-equivariant anti-self-dual connections on $X\times\Delta$; and for $\tau=0$, from $SE(2)$-equivariant anti-self-dual connections on $X\times\mathbb{C}$.
 Here $\Delta$ is the unit disk in the complex plane and $SE(2)$ is the semi-direct product $\mathbb{C}\rtimes U(1)$.
It is worth noting that the structure group of the connections is $SU(1,1)$ in all three cases. 
Therefore, Manton's exotic vortex equation is naturally related to gauge theory with the non-compact gauge group $SU(1,1)$.
We also note that the dimensional reduction for the case $\tau<0$ and $X=S^2$ was established earlier by Popov \cite{Po2}.
\par
While the dimensional reduction established by Contatto and Dunajski \cite{CD} is described in terms of local anti-self-dual connections, inspired by the work of Garc\'ia-Prada \cite{GP1}, in Section \ref{sec 3} we establish a global dimensional reduction in terms of holomorphic vector bundles and connections.
More precisely, we derive solutions of Manton's exotic vortex equation from pseudo-Hermitian-Yang-Mills connections (see Definition \ref{def 3.2}) with pseudo-Hermitian metrics of signature (1,1) (see Definition \ref{def 3.1}), rather than from anti-self-dual connections.
We also construct a pseudo-Hermitian-Yang-Mills connection together with a pseudo-Hermitian metric of signature $(1,1)$ from solutions of Manton's exotic vortex equation.
\par
We now describe our dimensional reduction result more precisely.
For brevity, we only state the $\tau>0$ case, which we treated in Section \ref{sec 3.2}.
The $\tau<0$ case is treated in Section \ref{sec 3.3} and the $\tau=0$ case is treated in Section \ref{sec 3.4}.
Before we state the result, we prepare some notations.
We denote by 
\begin{align*}
&p:X\times \mathbb{P}^1\to X\\
&q:X\times \mathbb{P}^1\to \mathbb{P}^1
\end{align*}
the natural projections, and denote by $\underline{\mathbb{C}}_{X\times \mathbb{P}^1}$ the trivial line bundle over $X\times\mathbb{P}^1$.
We first state the result that constructs a pseudo-Hermitian-Yang-Mills connection together with a pseudo-Hermitian metric of signature $(1,1)$ from solutions of Manton's exotic vortex equation.
\begin{thm}[Theorem \ref{thm 3.1}]
Let $X$ be a compact connected Riemann surface and let  $\omega_X$ be a K\"ahler form of $X$ satisfying $\int_X\omega_X=1$.
Let $(L,\overline\partial_L)$ be a holomorphic line bundle over $X$, $\phi$ a holomorphic section  of $(L,\overline\partial_L)$, and $h_L$ a Hermitian metric on $L$.
Assume that the triple $((L,\overline\partial_L),\phi,h_L)$ satisfies Manton's exotic $\tau$-vortex equation with $\tau>0$.
Then 
\begin{align*}
E:=&p^\ast L\otimes q^\ast \mathcal{O}(-2)\oplus \underline{\mathbb{C}}_{X\times\mathbb{P}^1},\\
\overline\partial_{E}:=&
\begin{pmatrix}
p^\ast \overline\partial_L\otimes \mathrm{Id}+\mathrm{Id}\otimes q^\ast \overline\partial_{\mathcal{O}(-2)}& \sqrt{\frac{1}{\tau}}\cdot p^\ast \phi\otimes q^\ast \alpha\\
0&\overline\partial_{X\times\mathbb{P}^1}
\end{pmatrix}.
\end{align*}
is an $SU(2)$-equivariant holomorphic bundle over $X\times \mathbb{P}^1$.
See Section \ref{sec 3.2.1} for the definition of $\alpha$.
We fix a K\"ahler form on $X\times\mathbb{P}^1$ given by
\be\label{eq 0.1}
\Omega_\tau=\bigg(\frac{\tau}{8\pi}p^\ast\omega_X\bigg)\oplus q^\ast \omega_{\mathrm{FS}}.
\ee
Here $\omega_{\mathrm{FS}}$ is the K\"ahler form of $\mathbb{P}^1$ associated to the Fubini-Study metric.
Then there exists a Hermitian metric $h_X$ on $\underline{\mathbb{C}}_X$ such that the pseudo-Hermitian metric of signature (1,1)
\begin{align*}
h_E=
\begin{pmatrix}
p^\ast (h_Xh_L)\otimes q^\ast h^{(-2)}&0\\
0& -p^\ast h_X
\end{pmatrix}
\end{align*}
is a pseudo-Hermitian-Einstein metric of $(E,\overline\partial_E)$ with respect to the K\"ahler form $\Omega_\tau$.
Equivalently, the unique connection $\nabla_{h_E}$ which is compatible with $h_E$ and $\nabla^{0,1}_{h_E}=\bar\partial_E$ is a pseudo-Hermitian-Yang-Mills connection.
\end{thm}
We now state our dimensional reduction result.
\begin{thm}[Theorem \ref{thm 3.2}]
Let $X$ be a compact connected Riemann surface and let  $\omega_X$ be a K\"ahler form of $X$ satisfying $\int_X\omega_X=1$.
Let $(L,\overline\partial_L)$ be a holomorphic line bundle over $X$.
Let $E$ be a holomorphic bundle over $X\times \mathbb{P}^1$ which is an extension of $\underline{\mathbb{C}}_{X\times \mathbb{P}^1}$ by $p^\ast L\otimes q^\ast\mathcal{O}(-2)$.
Let $\overline\partial_E$ be the Dolbeault operator of $E$.
On $X\times\mathbb{P}^1$, we consider the K\"ahler form $\Omega_\tau$ in \eqref{eq 0.1}.
\par
If $(E,\bar\partial_E)$ is an $SU(2)$-equivariant holomorphic bundle, then there exists a holomorphic section $\phi$ of $(L,\bar\partial_L)$ such that 
\[
\overline\partial_E=
\begin{pmatrix}
p^\ast \overline\partial_L\otimes\mathrm{Id}+\mathrm{Id}\otimes q^\ast \overline\partial_{\mathcal{O}(-2)}& \sqrt{\frac{1}{\tau}}\cdot p^\ast\phi\otimes q^\ast\alpha\\
0&\overline\partial_{X\times\mathbb{P}^1}
\end{pmatrix}.
\]
Let $h_E$ be a pseudo-Hermitian metric of signature $(1,1)$ on $E.$
If $h_E$ is $SU(2)$-invariant, then there exist Hermitian metrics $h_L$ on $L$ and $h_X$ on $\underline{\mathbb{C}}_X$ such that 
\[
h_E
=
\begin{pmatrix}
C_0\cdot p^\ast (h_Xh_L)\otimes q^\ast h^{(-2)}&0\\
0& C_1\cdot p^\ast h_X
\end{pmatrix}
\]
where $(C_0,C_1)=(1,-1)$ or $(-1,1)$.\par
Let $h_E$ and $(E,\overline\partial_E)$ be as above, and $\nabla_{h_E}$ be the connection compatible with $h_E$ and $\nabla^{0,1}_{h_E}=\overline\partial_E$.
If $\nabla_{h_E}$ is a pseudo-Hermitian-Yang-Mills connection, then the triple 
 $((L,\overline\partial_L),\phi,h_L)$ satisfies Manton's exotic $\tau$-vortex equation with $\tau>0$.
\end{thm}
\begin{remark}
The Manton equation for $\tau< 0$ and its generalization were studied in the unpublished paper by Torres de Lizaur \cite{Tdl} - produced as part of his ENS Master program visit at ICMAT (Madrid) under the supervision of Garc\'ia-Prada-, where the moment map interpretation and the dimensional reduction were established.
In the present paper, we also construct pseudo-Hermitian-Yang-Mills connections from solutions of Manton's exotic vortex equation.
\end{remark}
\begin{remark}
There is also a relation between Cartan geometry and vortex equations; see \cite{R}.
\end{remark}
 \subsubsection*{Acknowledgment}
 The author would like to thank Oscar Garc\'ia-Prada, Nicholas Manton, Calum Ross, and Alexander Popov for comments on this paper.
 The author also thanks Oscar Garc\'ia-Prada for kindly providing a copy of the paper \cite{Tdl}.
  \par
 The author presented some parts of this paper at the "Algebraic Geometry Seminar" at Kyoto University, "Seminar on Geometric Complex Analysis" at Tokyo University, and "New developments in Kobayashi-Hitchin correspondence and Higgs bundles 2" at Osaka Metropolitan University.
The author appreciates the organizers for the opportunity.
\par
The author was supported by JSPS KAKENHI Grant Number JP26KJ0189.
\subsubsection*{Notations}
Let $X$ be a complex manifold and $E$ be a complex vector bundle over $X$. 
We denote the space of smooth sections as $A(E)$. 
We denote the dual bundle of $E$ as $E^\vee$.
 We denote the space of smooth $i$-forms which take values in $E$ as $A^i(E)$ and smooth $(p,q)$-forms as $A^{p,q}(E)$.
\par
We denote the trivial holomorphic line bundle $X\times \mathbb{C}$ over $X$ as $\underline{\mathbb{C}}_X$.
\section{Vortex Equations}\label{sec 1}
In Section \ref{sec 1.1}, we introduce Manton's exotic vortex equation.
In Section \ref{sec 1.2}, we introduce a functional and show that solutions of the equation are critical points of the functional.
In Section \ref{sec 1.3}, we study the existence of solutions to the equation.
Using a result of Kazdan and Warner \cite{KW}, we prove the existence theorem under suitable assumptions.
\subsection{Vortex Equation and Manton's Exotic Vortex Equation}\label{sec 1.1}
In this section, we introduce Manton's exotic $\tau$-vortex equation.
\par
Let $(X,\omega_X)$ be a compact Riemann surface with a K\"ahler form $\omega_X$.
Let $L$ be a complex line bundle with a holomorphic structure $\overline\partial_L$.
Let $\phi$ be a holomorphic section of $(L,\overline\partial_L)$ and $h$ be a Hermitian metric on $L$.
We denote by $\nabla_h$ the Chern connection of $L$ with respect to $\overline\partial_L$ and $h$, i.e., the unique unitary connection of $h$ such that the $(0,1)$-part of $\nabla_h$ is $\overline\partial_L$.
We denote by $F_{h}$ the curvature of $\nabla_h$. 
\par
Let $\tau>0$ be a positive constant.
In his study of the deformation of the vortex equation \cite{Ma}, Manton introduced the family of equations
\be\label{eq 1.1}
\sqrt{-1}\Lambda_{\omega_X}F_{h}+C_0\frac12|\phi|^2_h-C_1\frac{\tau}{2}=0,\quad (C_0,C_1\in\{-1,0,1\}).
\ee
Although there are nine possible pairs $(C_0,C_1)$, only the following five 
\[
(C_0,C_1)=(1,1),(-1,1),(-1,0),(-1,-1),(0,1)
\]
pairs can occur.
We explain why the pair $(C_0,C_1)=(1,-1)$ cannot occur. 
The remaining cases can be treated similarly.
Suppose 
\[
\sqrt{-1}\Lambda_{\omega_X}F_{h}+\frac12|\phi|^2_h+\frac{\tau}{2}=0.
\]
Since we are interested in the case where $\phi$ is a non-trivial holomorphic section, the degree of $L$ is positive.
Therefore, after integrating both sides over $X$, the left-hand side is strictly positive, whereas the right-hand side is zero, yielding a contradiction.
\par
Since the $(C_0,C_1)=(1,1)$ and $(C_0,C_1)=(0,1)$ cases have already been extensively studied, in this paper we focus on the remaining three cases,
\[
(C_0,C_1)=(-1,-1),(-1,0),(-1,1).
\]
Since the parameter $\tau$ is assumed to be positive in \eqref{eq 1.1}, and since $C_1$ takes the values $-1, 0,$ and $1$ for the three cases of interest, these three equations can be unified by allowing $\tau$ to vary over all real numbers.
 This leads to the following definition of Manton's exotic $\tau$-vortex equation.
\begin{de}
Let $(X,\omega_X)$ be a compact Riemann surface with a K\"ahler form $\omega_X$.
Let $L$ be a complex line bundle with a holomorphic structure $\overline\partial_L$, $\phi$ be a holomorphic section of $(L,\overline\partial_L)$, and $h$ be a Hermitian metric on $L$.
Let $\tau\in\mathbb{R}$ be a constant.\par
We say that the triple $((L,\overline\partial_L),\phi,h)$ satisfies Manton's exotic $\tau$-vortex equation if 
\be\label{eq 1.2}
\sqrt{-1}\Lambda_{\omega_X}F_{h}-\frac12|\phi|^2_h-\frac{\tau}{2}=0.
\ee
\end{de}
 \subsection{Functional}\label{sec 1.2}
Let $(X,\omega_X)$ be a compact Riemann surface with the K\"ahler form $\omega_X$.
We normalize $\omega_X$ so that 
\[
\mathrm{Vol}_X:=\int_X\omega_X=1.
\]
Let $L$ be a complex line bundle with a fixed Hermitian metric $h$.
We denote by $\mathcal{A}_h$ the space of unitary connections of $L$ with respect to $h$.
Let $\nabla_h\in\mathcal{A}_h$.
We denote by $F_{h}$ the curvature of $\nabla_h$ and by $\nabla^{0,1}_h$ the $(0,1)$-part of $\nabla_h$.
\par
Let $\phi\in A(L)$ be a holomorphic section of $(L,\nabla^{0,1}_h)$.
We now define a functional and show, at least formally, that if $((L,\nabla^{0,1}_h),\phi,h)$ satisfies Manton's exotic vortex equation, then $(\nabla_h,\phi)$ is a critical point of the functional.
\begin{de}
Let $\tau\in\mathbb{R}$ be a constant.
We define the Yang-Mills-Higgs functional
\[
\mathrm{YMH}_\tau:\mathcal{A}_h\times A(L)\to \mathbb{R} 
\]
as 
\[
\mathrm{YMH}_\tau(\nabla_h,\phi):=\| F_{h}\|^2-\|\nabla_h\phi\|^2+\frac{1}{4}\||\phi|^2_h+\tau\|^2.
\]
Here $\|\cdot\|$ denotes the $L^2$-norm defined by the metric on $X$ and $h$.
\end{de}
Using the result of \cite{Br1}, we have
\begin{prop}\label{prop 1.1}
\[
\mathrm{YMH}_\tau(\nabla_h,\phi)=\bigg\|\sqrt{-1}\Lambda_\omega F_{h}-\frac{1}{2}|\phi|^2_h-\frac{\tau}{2}\bigg\|^2-2\|\nabla^{0,1}_h\phi\|^2+2\pi\tau\mathrm{deg}L.
\]
\end{prop} 
\bpr
We expand 
\be\label{eq 1.3}
\bigg\|\sqrt{-1}\Lambda_\omega F_{h}-\frac{1}{2}|\phi|^2_h-\frac{\tau}{2}\bigg\|^2=\|\Lambda_\omega F_{h}\|^2+\frac{1}{4}\||\phi|^2_h+\tau\|^2-
\langle\sqrt{-1}\Lambda_\omega F_{h},|\phi|^2_h\rangle
-\langle\sqrt{-1}\Lambda_\omega F_{h},\tau\rangle.
\ee
We have 
\begin{align*}
\langle\sqrt{-1}\Lambda_\omega F_{h},\tau\rangle&=2\pi\tau\mathrm{deg}L,\\
\|\Lambda_\omega F_{h}\|^2&=\| F_{h}\|^2
\end{align*}
and by \cite[page 5]{Br1}, we also have
\begin{align*}
\langle\sqrt{-1}\Lambda_\omega F_{h},|\phi|^2_h\rangle=-\|\nabla^{0,1}_h\phi\|^2+\|\nabla^{1,0}_h\phi\|^2.
\end{align*} 
Hence equation (\ref{eq 1.3}) becomes 
\begin{align*}
(\ref{eq 1.3})=\| F_{h}\|^2+\frac{1}{4}\||\phi|^2_h+\tau\|^2+\|\nabla^{0,1}_h\phi\|^2-\|\nabla^{1,0}_h\phi\|^2-2\pi\tau\mathrm{deg}L.
\end{align*}
Therefore 
\begin{align*}
&\bigg\|\sqrt{-1}\Lambda_\omega F_{h}-\frac{1}{2}|\phi|^2_h-\frac{\tau}{2}\bigg\|^2-2\|\nabla^{0,1}_h\phi\|^2+2\pi\tau\mathrm{deg}L\\
=&\| F_{h}\|^2+\frac{1}{4}\||\phi|^2_h+\tau\|^2+\|\nabla^{0,1}_h\phi\|^2-\|\nabla^{1,0}_h\phi\|^2-2\pi\tau\mathrm{deg}L-2\|\nabla^{0,1}_h\phi\|^2+2\pi\tau\mathrm{deg}L\\
=&\| F_{h}\|^2+\frac{1}{4}\||\phi|^2_h+\tau\|^2-\|\nabla_h\phi\|^2\\
=&\mathrm{YMH}_\tau(\nabla_h,\phi).
\end{align*}
\epr
Now, suppose that the triple $(\nabla_h,\phi,h)$ satisfies the $\tau$-vortex equation, and let $(\nabla_h(t),\phi(t)):(-\epsilon,\epsilon)\to \mathcal{A}_h\times A(L)$ be a curve such that $(\nabla_h(0),\phi(0))=(\nabla_h,\phi).$
We set 
\begin{align*}
A(t)&:=\sqrt{-1}\Lambda_\omega F_h(t)-\frac{1}{2}|\phi(t)|^2_h-\frac{\tau}{2},\\
B(t)&:=\nabla^{0,1}_h(t)\phi(t).
\end{align*}
Here $ F_h(t)$ is the curvature of $\nabla_h(t)$.
We note that $A(0)=B(0)=0$ holds by the assumption of $(\nabla_h,\phi,h)$.
Then by Proposition \ref{prop 1.1}, we have 
\begin{align*}
\frac{d}{dt}\mathrm{YMH}_\tau(\nabla_h(t),\phi(t))\bigg|_{t=0}&=\frac{d}{dt}\bigg( \langle A(t),A(t) \rangle-2\langle B(t),B(t)\rangle+2\pi\tau\mathrm{deg}L\bigg)\bigg|_{t=0}\\
&=2\bigg\langle A(0),\frac{d}{dt}A(0)\bigg\rangle-4\bigg\langle B(0),\frac{d}{dt}B(0)\bigg\rangle\\
&=0.
\end{align*}
Hence $(\nabla_h,\phi)$ is a critical point of $\mathrm{YMH}_\tau.$
\begin{remark}
The functional $\mathrm{YMH}_\tau:\mathcal{A}_h\times A(L)\to \mathbb{R} $, introduced in \cite{Br1} and also considered in \cite{GP1}, is 
\[
\mathrm{YMH}_\tau(\nabla_h,\phi):=\| F_{h}\|^2+\|\nabla_h\phi\|^2+\frac{1}{4}\||\phi|^2_h-\tau\|^2.
\]
The functional in \cite{Br1, GP1} is bounded below, whereas ours is not.
\end{remark}
\subsection{Existence of Solutions}\label{sec 1.3}
Let $(X,\omega_X)$ be a compact Riemann surface with the K\"ahler form $\omega_X$.
We normalize $\omega_X$ so that 
\[
\mathrm{Vol}_X:=\int_X\omega_X=1.
\]
In this section, we study the existence of solutions to Manton's exotic vortex equation.
In particular, we would like to answer the following question.
\begin{q}\label{q 1}
Let $\tau\in\mathbb{R}$ be a constant.
Let $(L,\overline\partial_L)$ be a holomorphic line bundle over $X$, and let $\phi\in A(L)$ be a holomorphic section of $(L,\overline\partial_L)$.
Then, when does there exist a Hermitian metric $h$ on $L$ such that the triple $(\overline\partial_L,\phi,h)$ satisfies Manton's exotic $\tau$-vortex equation \eqref{eq 1.2}?
\end{q}
We answer this question using the result of Kazdan and Warner \cite{KW}, which was also used in \cite{Br1}.\par
Let $(L,\overline\partial_L)$ be a holomorphic line bundle over $X$, and let $h$ be a Hermitian metric on $L$. 
For every Hermitian metric $k$ of $L$, we have a smooth function $u:X\to \mathbb{R}$ such that 
\[
k=he^{2u}.
\]
\begin{prop}\label{prop 1.2}
The triple $((L,\overline\partial_L),\phi,k)$ satisfies Manton's exotic $\tau$-vortex equation if and only if 
\be\label{eq 1.4}
\triangle u-\frac{|\phi|^2_h}{2}e^{2u}+\sqrt{-1}\Lambda_\omega F_{h}-\frac{\tau}{2}=0.
\ee
Here $\triangle$ is the Hodge Laplacian of $X$.
\end{prop}
\bpr
Since $F_{k}$ and $F_{h}$ are the curvatures of the Chern connections, and $L$ is a line bundle, we have
\[
F_{k}=\overline\partial\partial \mathrm{log}k=\overline\partial\partial \mathrm{log}h+2\overline\partial\partial u=F_{h}+2\overline\partial\partial u.
\]
Then by the K\"ahler identities
\begin{align*}
\sqrt{-1}\Lambda_\omega F_{k}&=\sqrt{-1}\Lambda_\omega F_{h}+2\sqrt{-1}\Lambda_\omega\overline\partial\partial u \\
&=\sqrt{-1}\Lambda_\omega F_{h}+\triangle u.
\end{align*}
See \cite{Wells} for details.
Since $k=he^{2u}$, we have
\[
|\phi|^2_k=k(\phi,\phi)=e^{2u}h(\phi,\phi)=e^{2u}|\phi|^2_h.
\]
Hence we have 
\begin{align*}
\sqrt{-1}\Lambda_\omega F_{k}-\frac{1}{2}|\phi|^2_k-\frac{\tau}{2}=\sqrt{-1}\Lambda_\omega F_{h}+\triangle u-\frac{|\phi|^2_h}{2}e^{2u}-\frac{\tau}{2}.
\end{align*}
The claim follows from the equation.
\epr
Proposition \ref{prop 1.2} provides a way to answer Question \ref{q 1}:
Let $h$ be a Hermitian metric on $L$, which always exists.
If we can find an  $u\in C^{\infty}(X,\mathbb{R})$ which satisfies the equation \eqref{eq 1.4}, then the triple $((L,\overline\partial_L),\phi, k=he^{2u})$ satisfies Manton's exotic $\tau$-vortex equation.
Thus, it suffices to solve equation \eqref{eq 1.4}.
For this purpose, we introduce an equation equivalent to \eqref{eq 1.4}.
This allows us to apply the result of Kazdan and Warner \cite{KW}.\par

We first define a constant
\[
c_{L,\tau}:=2\int_X\bigg(\sqrt{-1}\Lambda_\omega F_{h}-\frac{\tau}{2}\bigg)=4\pi\mathrm{deg}L-\tau.
\]
The second equation follows since we normalized the volume of $X$ to 1.
Since 
\[
\int_X \bigg(\sqrt{-1}\Lambda_\omega F_{h}-\frac{\tau}{2}-\frac{c_{L,\tau}}{2}\bigg)=0,
\]
there exists a smooth function $v_h:X\to \mathbb{R}$ that satisfies
\[
-\triangle v_h=\sqrt{-1}\Lambda_\omega F_{h}-\frac{\tau}{2}-\frac{c_{L,\tau}}{2}.
\] 
This follows from the standard Hodge theory (See \cite[p. 223]{Wa}).

\begin{prop}\label{prop 1.3}
Let $v_h$ and $c_{L,\tau}$ be as above.
Let $u\in C^{\infty}(X,\mathbb{R})$ and define a function
\[
w:=2(u-v_h).
\]
Then $u$ satisfies the equation \eqref{eq 1.4} if and only if $w$ satisfies the equation
\be\label{eq 1.5}
-\triangle w+|\phi|^2_h e^{2v_h}e^w-c_{L,\tau}=0.
\ee
\end{prop}
\bpr
Assume that $u$ satisfies \eqref{eq 1.4}.
Then we have 
\begin{align*}
-\triangle w&=-2\triangle u+2\triangle v_h\\
&=-|\phi|^2_he^{2u}+2\sqrt{-1}\Lambda_\omega F_{h}-\tau-2\sqrt{-1}\Lambda_\omega F_{h}+\tau+c_{L,\tau}\\
&=-|\phi|^2_he^{2u}+c_{L,\tau}.
\end{align*}
Hence $w$ satisfies \eqref{eq 1.5}.
Conversely, assume that $w$ satisfies \eqref{eq 1.5}.
Then we have 
\begin{align*}
0&=-\triangle w+|\phi|^2_h e^{2v_h}e^w-c_{L,\tau}\\
&=-2\triangle u+2\triangle v_h+|\phi|^2_h e^{2u}-c_{L,\tau}\\
&=-2\triangle u-2\sqrt{-1}\Lambda_\omega F_{h}+\tau+c_{L,\tau}+|\phi|^2_h e^{2u}-c_{L,\tau}\\
&=-2\triangle u-2\sqrt{-1}\Lambda_\omega F_{h}+\tau+|\phi|^2_h e^{2u}.
\end{align*}
It is clear that $u$ satisfies (\ref{eq 1.4}).
\epr
As a summary of Proposition \ref{prop 1.2} and \ref{prop 1.3}, we have
\begin{prop}\label{prop 1.4}
Let $w\in C^\infty(X,\mathbb{R})$.
Then the triple $((L,\overline\partial_L),\phi, he^{v_h+2w})$ satisfies Manton's exotic $\tau$-vortex equation if and only if $w$ satisfies \eqref{eq 1.5}.
\end{prop}
Hence, by Proposition \ref{prop 1.2} and \ref{prop 1.3}, the existence problem (Question \ref{q 1}) is equivalent to solving (\ref{eq 1.5}).
Since \eqref{eq 1.5} is exactly the equation considered by Kazdan and Warner \cite{KW}, we are able to apply their result.
We recall their results.
Note that we only state the parts that will be used in this paper.
\begin{thm}[{\cite[Theorem 5.3, Theorem 7.2, Theorem 10.1]{KW}}]\label{thm 1.1}
Let $(X,\omega)$ be a compact Riemann surface with $\mathrm{Vol}_X=\int_X\omega=1$.
Let $\alpha\in \mathbb{R}$ be a constant and $f\in C^{\infty}(X,\mathbb{R})$.
We consider the equation
\be\label{eq 1.6}
-\triangle w+fe^w-\alpha=0.
\ee
The existence of solutions to equation \eqref{eq 1.6} depends on the value of $\alpha$.
We therefore distinguish several cases according to the value of $\alpha$.
\begin{itemize}
\item[(1)] $\alpha=0.$ 
In this case, $f$ changes sign and $\int_Xf<0$ are necessary and sufficient conditions for the existence of a solution $w\in C^{\infty}(X,\mathbb{R})$ to equation \eqref{eq 1.6}.
\item[(2)] $\alpha>0.$
\begin{itemize}
\item[(a)] A necessary condition for the existence of a solution to \eqref{eq 1.6} is that $f$ be positive somewhere on $X$.
\item[(b)] If $f$ is positive somewhere, there is a positive constant $c(f)>0$ depending on $f$ such that \eqref{eq 1.6} has a solution $w\in C^{\infty}(X,\mathbb{R})$ if $0<\alpha<c(h)$. 
Moreover there exists a constant $\beta$ which does not depend on $f$ but possibly depends on $X$ such that 
\[
\beta\leq c(f).
\]
\end{itemize}
\item[(3)] $\alpha<0$. In this case, $\int_Xf<0$ is a necessary condition for the existence of a solution $w\in C^{\infty}(X,\mathbb{R})$ to equation \eqref{eq 1.6}.
\end{itemize}
\end{thm}
See \cite[p. 21-23]{KW} for the details of $\beta$.
\begin{remark}
The Laplacian we used has the opposite sign from the one used in \cite{KW}.
\end{remark}
\begin{remark}
The result of \cite{KW} was also used by Bradlow \cite{Br1}.
We note that the case $c<0$ (\cite[Theorem 10.1]{KW}) was used in \cite{Br1} to construct solutions of the vortex equation, while we will use the case $c>0$ (\cite[Theorem 7.2]{KW}).
\end{remark}
We state two existence results.
The first concerns a fixed $\tau\in\mathbb{R}$. 
Under suitable conditions on the triple $((L,\overline{\partial}_L),\phi,h)$, there exists a Hermitian metric $k$ such that the triple $((L,\overline{\partial}_L),\phi,k)$ satisfies Manton's exotic $\tau$-vortex equation.
The second states that for every holomorphic line bundle $(L,\overline{\partial}_L)$ and every holomorphic section $\phi$ of $(L,\overline{\partial}_L)$, if $\tau$ is chosen appropriately, then there exists a Hermitian metric $h$ such that the triple $((L,\overline{\partial}_L),\phi,h)$ satisfies Manton's exotic $\tau$-vortex equation.
\begin{thm}\label{thm 1.2}
Let $(X,\omega_X)$ be a compact Riemann surface with the K\"ahler form $\omega_X$.
We normalize $\omega_X$ so that $\mathrm{Vol}_X:=\int_X\omega_X=1.$
\begin{itemize}
\item[(1)] Let $\tau\in\mathbb{R}$ be a constant.
Let  $(L,\overline\partial_L)$ be a holomorphic line bundle over $X$, $\phi$ be a holomorphic section of $(L,\overline\partial_L)$, and $h$ be a Hermitian metric of $L$.
Let $v_{h}$ be the function defined above.\par
If $\tau$ satisfies 
\[
4\pi\mathrm{deg}L-c(|\phi|^2_h e^{2v_h})<\tau<4\pi\mathrm{deg}L,
\]
there exists a Hermitian metric $k$ such that the triple $((L,\overline{\partial}_L),\phi,k)$ satisfies Manton's exotic $\tau$-vortex equation.
Here $c(|\phi|^2_h e^{2v_h})$ is the constant which appeared in Theorem \ref{thm 1.1}.
\item[(2)] For every holomorphic line bundle $(L,\overline\partial_L)$ over $X$ and a holomorphic section $\phi$ of $(L,\overline\partial_L)$, if we take $\tau$ appropriately, there exists a Hermitian metric $h$ such that the triple $((L,\overline{\partial}_L),\phi,h)$ satisfies Manton's exotic $\tau$-vortex equation.
\end{itemize}
\end{thm}
\begin{proof}
We first show $(1).$
By Theorem \ref{thm 1.1}, if 
\[
0<c_{L,\tau}<c(|\phi|^2_h e^{2v_h})
\]
holds, which is equivalent to 
\[
4\pi\mathrm{deg}L-c(|\phi|^2_h e^{2v_h})<\tau<4\pi\mathrm{deg}L
\]
by the definition of $c_{L,\tau}$, there exists a $w\in C^{\infty}(X,\mathbb{R})$ such that $w$ is a solution of \eqref{eq 1.5}.
Then, by Proposition \ref{prop 1.4}, the triple $((L,\overline\partial_L),\phi, he^{v_h+2w})$ satisfies Manton's exotic $\tau$-vortex equation.\par
We next prove $(2)$.
Let $h$ be a Hermitian metric of $L$.
If we take $\tau$ appropriately
\[
4\pi\mathrm{deg}L-c(|\phi|^2_h e^{2v_h})<\tau<4\pi\mathrm{deg}L
\]
holds.
This is equivariant to 
\[
0<c_{L,\tau}<c(|\phi|^2_h e^{2v_h}).
\]
Then, by Theorem \ref{thm 1.1}, there exsits $w\in C^{\infty}(X,\mathbb{R})$ such that $w$ is a solution of \eqref{eq 1.5}.
 The triple $((L,\overline\partial_L),\phi, he^{v_h+2w})$ satisfies Manton's exotic $\tau$-vortex equation by Proposition \ref{prop 1.4}.
\end{proof}
\section{Moduli Space}\label{sec 2}
We use the notation we used in the previous section.\par
In Section \ref{sec 2.1}, we define two spaces which can be regarded as moduli spaces of solutions of Manton's exotic vortex equation.
We show that there is a bijective map between them.
\par
In Section \ref{sec 2.2}, we endow the moduli space with a topology induced by Sobolev norms.
We also show that a certain subset of the moduli space is a smooth manifold, and its real dimension is twice the degree of the line bundle.
\par
In Section \ref{sec 2.3}, we show that Manton's vortex equation can be interpreted as a moment map equation.
The argument in Section \ref{sec 2.3} is purely formal.
We also note that the moment map interpretation of Manton's exotic vortex equation and its generalization was already obtained in the unpublished paper \cite{Tdl}.
Since this provides a natural geometric interpretation of the equation, we included the discussion for completeness.
\subsection{Definition of the Moduli Space}\label{sec 2.1}
In this section, we define two sets that can naturally be regarded as the moduli space of the vortex equation, and construct a bijection between them.\par
We prepare some notations.
Let $X$ be a compact Riemann surface.
Let $L$ be a complex line bundle, $\mathcal{A}_{hol,L}$ be the space of holomorphic structures on $L$, and $\mathrm{Herm}(L)$ be the space of Hermitian metrics of $L$.
Let $h\in \mathrm{Herm}(L)$.
We denote the space of $h$-unitary connections as $\mathcal{A}_h$, the spaces of smooth $\mathbb{C}^\ast$- and $S^1$-valued functions as $C^{\infty}(X,\mathbb{C}^\ast)$ and $C^{\infty}(X,S^1)$, and the spaces of smooth $i$-forms ($i\geq 1$) as $A^i(X, \mathbb{C}^\ast)$ and $A^i(X, S^1)$.
We note that $C^{\infty}(X,\mathbb{C}^\ast)$ and $C^{\infty}(X,S^1)$ are groups.\par
Let $\overline\partial_L\in \mathcal{A}_{hol,L}, h\in \mathrm{Herm}(L).$
In this section, we denote by $\partial_{\overline\partial_L,h}$, the $(1,0)$-part of the Chern connection with respect to $\overline\partial_L$ and $h$, and by $\nabla_{\overline\partial_L,h}=\partial_{\overline\partial_L,h}+\overline\partial_L$,the corresponding Chern connection.
We also denote by $F_{\nabla_{h}}$ the curvature of $\nabla_h\in \mathcal{A}_h$. 
These notations differ from those used previously and are introduced to emphasize the dependence of $\nabla_{\overline\partial_L,h}, \partial_{\overline\partial_L,h}$ on $\overline\partial_L,h$, and the dependence of $F_{\nabla_h}$on $\nabla_h$.
\par
The group $C^{\infty}(X,\mathbb{C}^\ast)$ acts on $\mathcal{A}_{hol,L}\times A(L)\times \mathrm{Herm}(L)$ as
\[
(\overline\partial_L,\phi,h)\cdot g \mapsto (g^{-1}\overline\partial_Lg, g^{-1}\phi, |g|^2h),
\]
where $g\in A(X,\mathbb{C}^\ast), (\overline\partial_L,\phi,h)\in \mathcal{A}_{hol,L}\times A(L)\times \mathrm{Herm}(L)$, and the group $C^{\infty}(X,S^1)$ acts on $\mathcal{A}_h\times A(L)$ as  
\be\label{eq 2.1}
(\nabla_h,\phi)\cdot g\mapsto (g^{-1}\nabla_hg,g^{-1}\phi),
\ee
where $g\in C^{\infty}(X,S^1), (\nabla_h,\phi)\in \mathcal{A}_h\times A(L).$
\par

Let $\tau\in\mathbb{R}$ be a constant. 
We define a subset of $ \mathcal{A}_{hol,L}\times A(L)\times \mathrm{Herm}(L)$ as
\begin{align*}
\mathrm{Tri}_\tau:
=\left\{
(\overline\partial_L,\phi,h)\in \mathcal{A}_{hol,L}\times A(L)\times \mathrm{Herm}(L)
\;\middle|\;
\begin{array}{l}
\centerdot\,\, \overline\partial_L\phi=0,\\
\centerdot\,\, (L,(\overline\partial_L),\phi,h))\, \text{satisfies Manton's exotic $\tau$-vortex equation}
\end{array}
 \right\}.
\end{align*}
See the beginning of Section \ref{sec 1.2} for the definition of $(\overline\partial_L,\phi,h)$ satisfying the $\tau$-vortex equation.\par
Let $h$ be a Hermitian metric on $L$.
We define a subset of $\mathcal{A}_h\times A(L)$ as
\begin{align*}
\mathrm{Vor}_\tau(h):=\left\{ (\nabla_h,\phi)\in \mathcal{A}_h\times A(L) \;\middle|\;
\begin{array}{l}
\centerdot \,\,\nabla^{0,1}_h\phi=0,\\
\centerdot\,\,((L,\nabla^{0,1}_h),\phi,h) \,\text{satisfies Manton's exotic $\tau$-vortex equation}
\end{array}
 \right\}.
\end{align*}
\begin{lem}\label{lem 2.1}
$\mathrm{Vor}_\tau(h)$ is a $C^{\infty}(X,S^1)$-invariant subset.
\end{lem}
\bpr
Let $(\nabla_h,\phi)\in\mathrm{Vor}_\tau(h) $ and $g\in C^{\infty}(X,S^1)$.
We show that $(\nabla_h,\phi)\cdot g\in \mathrm{Vor}_\tau(h).$ 
Since the $(0,1)$-part of $g^{-1}\nabla_h g$ is $g^{-1}\nabla^{0,1}_h g$, we have 
\[
g^{-1}\nabla^{0,1}_h g(g^{-1}\phi)=g^{-1}\nabla^{0,1}_h \phi=0.
\]
Hence $g^{-1}\phi$ is a holomorphic section of $(L,g^{-1}\nabla^{0,1}_h g)$.
\par
In this section, we denote the curvature of $\nabla_h\in\mathcal{A}_h$ by $F_{\nabla_h}$ to emphasize its dependence on $\nabla_h$.
The curvatures of $g^{-1}\nabla_h g$ and $\nabla_h$ are related as
\[
F_{g^{-1}\nabla_h g}=g^{-1}F_{\nabla_h}g=F_{\nabla_{h}}.
\]
The last equation holds since $L$ is a line bundle and hence $F_{g^{-1}\nabla_h g},F_{\nabla_h}$ are smooth 2-forms.
Since $g\in C^{\infty}(X,S^1)$, we have
\[
|g^{-1}\phi|^2_h=h(g^{-1}\phi,g^{-1}\phi)=g^{-1}\overline{g^{-1}}h(\phi,\phi)=|\phi|^2_h.
\]
Therefore 
\begin{align*}
&\sqrt{-1}\Lambda_\omega F_{g^{-1}\nabla_h g}-\frac{1}{2}|g^{-1}\phi|^2_h-\frac{\tau}{2}\\
=&\sqrt{-1}\Lambda_\omega F_{\nabla_h}-\frac{1}{2}|\phi|^2_h-\frac{\tau}{2}\\
=&0.
\end{align*}
Hence, the claim is proved.
\epr
We next prove that $\mathrm{Tri}_\tau$ is a $C^{\infty}(X,\mathbb{C}^\ast)$-invariant subset.
Before proceeding, we prepare a Lemma.
\begin{lem}\label{lem 2.2}
Let $(L,\overline\partial_L)$ be a holomorphic line bundle, $h$ a Hermitian metric on $L$, and $g\in C^{\infty}(X,\mathbb{C}^\ast)$.
Since $g\in C^{\infty}(X,\mathbb{C}^\ast)$, $|g|^2h$ is also a Hermitian metric on $L$.
Then we have 
\begin{itemize}
\item[1.] The $(1,0)$-part of the Chern connection with respect to $g^{-1}\overline\partial_L g$ and $|g|^2h$ has the form
\[
\partial_{g^{-1}\overline\partial_L g, |g|^2h}=\partial_{\overline\partial_L,h}+g^{-1}\partial g.
\]
\item[2.] The Chern connection $\nabla_{g^{-1}\overline\partial_L g, |g|^2h}$ with respect to $\overline\partial_L$ and $h$ has the form
\[
\nabla_{g^{-1}\overline\partial_L g, |g|^2h}=g^{-1}\nabla_{\overline\partial_L,h}g.
\]
\end{itemize}
\end{lem}
\bpr
We first prove $2.$, assuming $1$.
Suppose $1.$ holds.
Then 
\begin{align*}
\nabla_{g^{-1}\overline\partial_L g, |g|^2h}&=\partial_{g^{-1}\overline\partial_L g, |g|^2h}+g^{-1}\overline\partial_Lg\\
&=\partial_{\overline\partial_L,h}+g^{-1}\partial g+\overline\partial_L+g^{-1}\overline\partial g\\
&=\nabla_{\overline\partial_L,h}+g^{-1}dg\\
&=g^{-1}\nabla_{\overline\partial_L,h}g.
\end{align*}
We now prove $1.$
Since the Chern connection is unique, it suffices to show that 
\[
\partial_{\overline\partial_L,h}+g^{-1}\partial g + g^{-1}\overline\partial_L g
\]
is a metric connection with respect to $|g|^2h$.
Let $u,v\in A(L)$.
\begin{align*}
&|g|^2h((\partial_{\overline\partial_L,h}+g^{-1}\partial g)u,v)+|g|^2h(u,g^{-1}\overline\partial_L g v)\\
=&|g|^2h(\partial_{\overline\partial_L,h}u,v)+g^{-1}\partial g|g|^2h(u,v)+|g|^2h(u, \overline\partial_L v+g^{-1}\overline\partial g v)\\
=&|g|^2h(\partial_{\overline\partial_L,h}u,v)+|g|^2h(u,\overline\partial_L v)+\overline{g}\partial g h(u,v)+\overline{g}^{-1}\partial \overline{g}|g|^2h(u,v)\\
=&|g|^2\partial h(u,v)+\overline{g}\partial g h(u,v)+g\partial \overline{g}h(u,v)\\
=&|g|^2\partial h(u,v)+\partial |g|^2 h(u,v)\\
=&\partial(|g|^2 h(u,v)).
\end{align*}
By a similar argument, we can show that 
\[
\overline\partial(|g|^2 h(u,v))=|g|^2h(g^{-1}\overline\partial_L g u,v)+|g|^2h(u,(\partial_{\bar\partial_L,h}+g^{-1}\partial g)v)
\]
holds.
Thus 
\[
\partial_{\bar\partial_L,h}+g^{-1}\partial g + g^{-1}\overline\partial_L g
\]
is a metric connection with respect to $|g|^2h$, and therefore 
\[
\partial_{g^{-1}\bar\partial_L g, |g|^2h}=\partial_{\bar\partial_L,h}+g^{-1}\partial g.
\]
\epr
\begin{lem}\label{lem 2.3}
$\mathrm{Tri}_\tau$ is a $C^{\infty}(X,\mathbb{C}^\ast)$-invariant subset.
\end{lem}
\bpr
Let $(\overline\partial_L,\phi,h)\in\mathrm{Tri}_\tau$ and $g\in C^{\infty}(X,\mathbb{C}^\ast)$.
We show $(\overline\partial_L,\phi,h)\cdot g=(g^{-1}\overline\partial_Lg, g^{-1}\phi, |g|^2h) \in \mathrm{Tri}_\tau.$
It is clear that $g^{-1}\phi$ is a holomorphic section of $(L,g^{-1}\overline\partial_Lg)$.\par
We use the same notation as in Lemma \ref{lem 2.2}.
By 2. of Lemma \ref{lem 2.2} , the curvatures of $\nabla_{g^{-1}\overline\partial_L g, |g|^2h}$ and $\nabla_{\overline\partial_L,h}$ are related as
\[
F_{\nabla_{g^{-1}\bar\partial_L g, |g|^2h}}=g^{-1}F_{\nabla_{\bar\partial_L,h}}g=F_{\nabla_{\bar\partial_L,h}}.
\]
We also have
\[
|g^{-1}\phi|_{|g|^2h}=|g|^2h(g^{-1}\phi,g^{-1}\phi)=|g|^2|g^{-1}|^2h(\phi,\phi)=|\phi|^2_h.
\]
Hence 
\begin{align*}
&\sqrt{-1}\Lambda_\omega F_{\nabla_{g^{-1}\bar\partial_L g, |g|^2h}}-\frac{|g^{-1}\phi|_{|g|^2h}}{2}-\frac{\tau}{2}\\
=&\sqrt{-1}\Lambda_\omega F_{\nabla_{\bar\partial_L,h}}-\frac{|\phi|^2_h}{2}-\frac{\tau}{2}\\
=&0.
\end{align*}
The claim is proved.
\epr
By Lemma \ref{lem 2.1} and \ref{lem 2.3}, the quotient spaces 
\[
\mathrm{Vor}_\tau(h) /C^{\infty}(X,S^1),\, \mathrm{Tri}_\tau /C^{\infty}(X,\mathbb{C}^\ast)
\]
are well-defined.
We regard the above two quotient spaces as moduli spaces of the $\tau$-vortex equation.
We now construct a bijection between them.
We define a map
\[
A:\mathrm{Vor}_\tau(h)\to \mathrm{Tri}_\tau /C^{\infty}(X,\mathbb{C}^\ast)
\]
as
\[
A((\nabla_h,\phi)):=[(\nabla^{0,1}_h,\phi,h)].
\]
\begin{lem}\label{lem 2.4}
$A$ is a $C^{\infty}(X, S^1)$-equivalent map.
That is, for $(\nabla_h,\phi)\in \mathrm{Vor}_\tau(h)$ and $g\in C^{\infty}(X, S^1)$
\[
A((\nabla_h,\phi)\cdot g)=A((\nabla_h,\phi))
\]
holds.
As a consequence, $A$ induces a well-defined map 
\[
\overline A:\mathrm{Vor}_\tau(h) /C^{\infty}(X,S^1)\to \mathrm{Tri}_\tau /C^{\infty}(X,\mathbb{C}^\ast)
\]
such that 
\[
\overline A([(\nabla_h,\phi)]):=[(\nabla^{0,1}_h,\phi,h)].
\]
\end{lem}
\bpr
Let $(\nabla_h,\phi)\in \mathrm{Vor}_\tau(h)$ and $g\in C^{\infty}(X, S^1)$.
We note that $g\in C^{\infty}(X, \mathbb{C}^\ast)$ and $|g|^2=1$.
We have
\begin{align*}
A((\nabla_h,\phi))=&[(\nabla^{0,1}_h,\phi,h)]\\
=&[(\nabla^{0,1}_h,\phi,h)\cdot g]\\
=&[(g^{-1}\nabla^{0,1}_hg,g^{-1}\phi,|g|^2h)]\\
=&[(g^{-1}\nabla^{0,1}_hg,g^{-1}\phi,h)]\\
=&A((\nabla_h,\phi)\cdot g).
\end{align*}
\epr
Let $(\overline\partial_L,\phi,k)\in\mathrm{Tri}_\tau$.
Then, by the general theory of vector bundles, $k^{-1}h$ is smooth positive function
\[
k^{-1}h:X\to \mathbb{R}_{>0}.
\]
We define $g\in C^{\infty}(X,\mathbb{C}^\ast)$ as
\[
g:=\mathrm{exp}\bigg(\frac{1}{2}\mathrm{log}(k^{-1}h)\bigg).
\]
Once we act $g$ on $(\overline\partial_L,\phi,k)$, we have 
\begin{align*}
(\overline\partial_L,\phi,k)\cdot g&=(g^{-1}\overline\partial_L g, g^{-1}\phi,|g|^2k)\\
&=(g^{-1}\overline\partial_L g, g^{-1}\phi,h).
\end{align*}
We note that $(g^{-1}\overline\partial_L g, g^{-1}\phi,h)\in\mathrm{Tri}_\tau$ by Lemma \ref{lem 2.3}.
Then by 2. of Lemma \ref{lem 2.2}, we have 
\[
\nabla_{g^{-1}\bar\partial_L g,h(=|g|^2k)}=g^{-1}\nabla_{\bar\partial_L,k}g.
\]
By the definition of $\mathrm{Tri}_\tau$, $(\nabla_{g^{-1}\bar\partial_L g,h},g^{-1}\phi)\in\mathrm{Vor}_\tau(h)$.  
We define a map
\[B: \mathrm{Tri}_\tau\to \mathrm{Vor}_\tau(h) /C^{\infty}(X,S^1).
\]
as
\[
B((\overline\partial_L,\phi,k))=[(\nabla_{g^{-1}\bar\partial_L g,h},g^{-1}\phi)].
\]
\begin{lem}\label{lem 2.5}
B is a $C^{\infty}(X,\mathbb{C}^\ast)$-equivalent map.
That is, for $(\overline\partial_L,\phi,h)\in\mathrm{Tri}_\tau$ and $g\in C^{\infty}(X,\mathbb{C}^\ast)$
\[
B((\overline\partial_L,\phi,h)\cdot g)=B((\overline\partial_L,\phi,h))
\]
holds.
As a consequence, $B$ defines a well-defined map 
\[
\overline B:\mathrm{Tri}_\tau /C^{\infty}(X,\mathbb{C}^\ast)\to\mathrm{Vor}_\tau(h) /C^{\infty}(X,S^1)
\]
such that 
\[
\overline B([(\overline\partial_L,\phi,k)])=[(\nabla_{g^{-1}\overline\partial_L g,h},g^{-1}\phi)].
\]
\end{lem}
\bpr
Let $(\overline\partial_{L_1},\phi_1,k_1),(\overline\partial_{L_2},\phi_2,k_2)\in\mathrm{Tri}_\tau$.
We assume that there exists a $g\in C^{\infty}(X,\mathbb{C}^\ast)$ such that 
\[
(\overline\partial_{L_1},\phi_1,k_1)\cdot g=(g^{-1}\overline\partial_{L_1}g,g^{-1}\phi_1,|g|^2k_1)=(\overline\partial_{L_2},\phi_2,k_2).
\]
We show that 
\[
B((\overline\partial_{L_1},\phi_1,k_1))=B((\overline\partial_{L_2},\phi_2,k_2))
\]
which proves the claim.
Let 
\[
g_i:=\mathrm{exp}\bigg(\frac{1}{2}\mathrm{log}(k_i^{-1}h)\bigg), \,(i=1,2).
\]
Then, by the definition of $B$, we have 
\[
B((\overline\partial_{L_i},\phi_i,k_i))=[(g^{-1}_i\nabla_{\bar\partial_{L_i,k_i}}g_i,g^{-1}_i\phi_i)],\, (i=1,2).
\]
Since we have 
\begin{align*}
k_2=k_1|g|^2, k_1|g_1|^2=h=k_2|g_2|^2,
\end{align*}
we obtain
\[
k_1|g_1|^2=k_1|g|^2|g_2|^2.
\]
Since $k_1$ is everywhere non-zero on $X$, we have 
\[
|g_1|^{-2}|g|^2|g_2|^2=1
\]
which implies
\[
g_1^{-1}\cdot g\cdot g_2\in C^\infty(X,S^1).
\]
By 2. of Lemma \ref{lem 2.2}, we have 
\[
\nabla_{\bar\partial_{L_2},k_2}=\nabla_{g^{-1}\bar\partial_{L_1}g, |g|^2k_1}=g^{-1}\nabla_{\bar\partial_{L_1},k_1}g.
\]
We now finish the proof of the claim.
The gauge group element $g_1^{-1}\cdot g\cdot g_2\in C^\infty(X,S^1)$ acts on $(g^{-1}_1\nabla_{\bar\partial_{L_1,k_1}}g_1,g^{-1}_1\phi_1)$ and maps to 
\begin{align*}
&(g^{-1}_1\nabla_{\bar\partial_{L_1,k_1}}g_1,g^{-1}_1\phi_1)\cdot (g_1^{-1}\cdot g\cdot g_2)\\
=&(g_2^{-1}g^{-1}\nabla_{\bar\partial_{L_1,k_1}}gg_2, g^{-1}_2g^{-1}\phi_1)\\
=&(g_2^{-1}\nabla_{\bar\partial_{L_2},k_2}g_2, g^{-1}_2\phi_2).
\end{align*}
This implies 
\[
[(g^{-1}_1\nabla_{\bar\partial_{L_1,k_1}}g_1,g^{-1}_1\phi_1)]=[(g^{-1}_2\nabla_{\bar\partial_{L_2,k_2}}g_2,g^{-1}_2\phi_2)]
\]
which also implies
\[
B((\overline\partial_{L_1},\phi_1,k_1))=B((\overline\partial_{L_2},\phi_2,k_2)).
\]
\epr
The following is clear from the definition of $A,B$.
\begin{prop}
The maps 
\begin{align*}
&\overline A:\mathrm{Vor}_\tau(h) /C^{\infty}(X,S^1)\to \mathrm{Tri}_\tau /C^{\infty}(X,\mathbb{C}^\ast),\\
&\overline B:\mathrm{Tri}_\tau /C^{\infty}(X,\mathbb{C}^\ast)\to\mathrm{Vor}_\tau(h) /C^{\infty}(X,S^1)
\end{align*}
are bijective maps.
\end{prop}
Thus, the two descriptions of the moduli space are equivalent.
In the following sections, we study the set
\[
\mathrm{Vor}_\tau(h) /C^{\infty}(X,S^1)
\]
after introducing Sobolev norms.
\subsection{Construction of the Moduli Space}\label{sec 2.2}
We use the same notation as in the previous section.
\par
In this section, we construct the moduli space of the solutions of the vortex equation.
Our construction is modelled on \cite{Ko, Nico, Par}.
To construct the moduli space, we first introduce the complex obtained by the linearization of the action of $C^{\infty}(X,S^1)$ and the vortex equation.
Later, we introduce $L^2_k$-Sobolev norms and construct the $L^2_k$-solution of the vortex equation for sufficiently large $k$ and argue that the moduli space does not depend on $k$.
We also study the smoothness of the moduli space by using the deformation complex. \par
Let $X$ be a compact Riemann surface with a K\"ahler form $\omega_X$.
Let $L$ be a smooth line bundle on $X$ with a Hermitian metric $h$.
From now on, we fix $\tau\in\mathbb{R}$.
\par
Let $(\nabla_h,\phi)\in \mathcal{A}_h\times A(L)$.
The linearization of the action of of $C^{\infty}(X,S^1)$ defines a linear map
\[
d_1:A(X,\sqrt{-1}\mathbb{R})\to A^1(X,\sqrt{-1}\mathbb{R})\oplus A(L)
\]
such that 
\[
d_1(X)=(dX, -X\phi).
\]
Let $(\nabla_h,\phi)\in\mathrm{Vor}_\tau$.
The linearization of $(\nabla_h,\phi)$ in $\mathrm{Vor}_\tau(h)$ defines a linear map
\[
d_2: A^1(X,\sqrt{-1}\mathbb{R})\oplus A(L) \to A^2(X,\sqrt{-1}\mathbb{R})\oplus A^{(0,1)}(L)
\]
such that 
\[
d_2(A, \dot{\phi})=(dA+\sqrt{-1}\mathrm{Re}h(\dot{\phi},\phi)\omega_X,\nabla^{0,1}_{h}\dot{\phi}+A^{0,1}\phi).
\]
Here $\mathrm{Re}h(\cdot,\cdot)$ is the real part of $h(\cdot,\cdot)$.
We denote by
\begin{align*}
C^0&:= A(X,\sqrt{-1}\mathbb{R}),\\
C^1&:= A^1(X,\sqrt{-1}\mathbb{R})\oplus A(L),\\
C^2&:= A^2(X,\sqrt{-1}\mathbb{R})\oplus A^{0,1}(L).
\end{align*}
Let $X\in L^0$.
Since 
\begin{align*}
d_2d_1(X)&=d_2(dX,-X\phi)\\
&=(ddX+\sqrt{-1}\mathrm{Re}h(-X\phi,\phi)\omega_X,\nabla^{0,1}_h(-X\phi)+\bar\partial X\phi)\\
&=(0,0).
\end{align*}
The last equation follows since $X\in C^0=A(X,\sqrt{-1}\mathbb{R})$ and $\phi$ is a holomorphic section with respect to $(L,\nabla^{0,1}_h)$.
Hence, we obtain a complex
\be\label{eq 2.2}
0\longrightarrow C^0\overset{d_1}{\longrightarrow} C^1\overset{d_2}{\longrightarrow}C^2\longrightarrow 0.
\ee
By the definition of $d_1$ and $d_2$, this is an elliptic complex.
We define a $L^2$-metric on $A^i(X,\sqrt{-1}\mathbb{R})$ as
\[
(\alpha,\beta)_{L^2}:=-\int_X\alpha\wedge\ast\beta
\]
and on $A^i(L)$ as 
\[
(u,v)_{L^2}:=\int_X\mathrm{Re}(u,v)_{h,g}.
\]
Here $(\cdot,\cdot)_{h,g}$ is the Hermitian metric on $A^i(L)$ defined by $h$ and a Riemannian metric $g$ on $X$.
We then define a $L^2$-metric on $C^i$ by using the $L^2$-metric on $A^i(X,\sqrt{-1}\mathbb{R})$ and $A^i(L)$, and denote by $d_1^\ast, d_2^\ast$ the formal adjoint of $d_1$ and $d_2$ with respect to the defined $L^2$-metric.
\par
We now fix $k>>1$.
We denote by $\mathcal{A}_{h,k}$ the space of $L^2_k$-unitary connections and by $A(L)_k$ the space of completion of $A(L)$ by $L^2_k$-Sobolev norm.
We define 
\begin{align*}
\mathcal{A}_k&:=\mathcal{A}_{h,k}\times A(L)_k,\\
\mathcal{G}_k&:=\bigg\{ s\in L^2_k(X,\mathbb{C})\,\,\bigg|\,\, |s(x)|=1,\quad \forall x\in X\bigg\}.
\end{align*}
Since we took $k$ large enough, $s\in L^2_k(X,\mathbb{C})$ is a continious map and hence $|s(x)|$ is well-defined.
By the Sobolev multiplication theorem, $\mathcal{G}_k$ is a group, and by \cite{Nico}, it is also a Hilbert Lie group.
The Hilbert Lie group $\mathcal{G}_{k+1}$ acts smoothly on $\mathcal{A}_k$ and the quotient space
\[
\mathcal{B}_k:=\mathcal{A}_k/\mathcal{G}_{k+1}
\]
is Hausdorff in the quotient topology (See \cite{Nico}).
\par
Let $C^i_k$ be the completion of $C^i$ by the $L^2_k$-Sobolev norm.
We denote the $L^2_k$-Sobolev norm by $\|\cdot\|_{L^2_k}$.
We define 
\[
\mathcal{A}^{\mathrm{irr}}_k:=\bigg\{(\nabla_h,\phi)\in \mathcal{A}_k\,\bigg|\, \mathrm{Ker}\big(d_1:C^0_{k+1}\to C^1_{k}\big)=0\bigg\}.
\]
Let $\epsilon>0$.
 For  $(\nabla_h,\phi)\in \mathcal{A}_k$, we define 
\[
S_{(\nabla_h,\phi),\epsilon}:=\bigg\{\alpha\in C^1_k\,\bigg|\, d_1^\ast\alpha=0, \|\alpha\|_{L^2_k}< \epsilon\bigg\}.
\]
Then by the slice theorem \cite{Nico, Par}, if $(\nabla_h,\phi)\in \mathcal{A}^{\mathrm{irr}}_k$ and $\epsilon$ is small enough, $S_{(\nabla_h,\phi),\epsilon}$ defines a local coordinate of 
\[
\mathcal{B}^{\mathrm{irr}}_k:=\mathcal{A}^{\mathrm{irr}}_k/\mathcal{G}_{k+1}
\]
and hence defines a Banach manifold structure on $\mathcal{B}^{\mathrm{irr}}_k$.
We define 
\begin{align*}
\mathrm{Vor}_\tau(h)_k:=\left\{ (\nabla_h,\phi)\in \mathcal{A}_k \;\middle|\;
\begin{array}{l}
\centerdot \,\nabla^{0,1}_h\phi=0,\\
\centerdot\,(\nabla_h,\phi) \,\text{satisfies Manton's exotic $\tau$-vortex equation}
\end{array}
 \right\}.
\end{align*}
We note that $\mathrm{Vor}_\tau(h)_k\subset \mathcal{A}_k $ and it is preserved under the action of $\mathcal{G}_{k+1}$.
We now define the moduli space of solutions of Manton's exotic vortex equation as
\[
\mathcal{M}_{\tau,k}:=\mathrm{Vor}_\tau(h)_k/\mathcal{G}_{k+1}
\]
and the moduli space of \textit{irreducible} solutions of Manton's exotic vortex equation as
\[
\mathcal{M}^{\mathrm{irr}}_{\tau,k}:=\mathrm{Vor}_\tau(h)_k\cap\mathcal{A}^{\mathrm{irr}}_k/\mathcal{G}_{k+1}.
\]
Let $[(\nabla_h,\phi)]\in \mathcal{M}_{\tau,k}$.
Then, by the bootstrap argument as in \cite{Nico, Par}, there exists a $g\in \mathcal{G}_{k+1}$ such that $(\nabla_h,\phi)\cdot g$ is smooth.
From now on, we assume that the point in $\mathcal{M}^{\mathrm{irr}}_{\tau,k}$ and $\mathcal{M}_{\tau,k}$ are represented by a smooth pair.\par
We study the local structure of $\mathcal{M}^{\mathrm{irr}}_{\tau,k}.$
Before we proceed, we prepare some notations.
Let $(\nabla_h,\phi)\in\mathrm{Vor}_\tau(h)$.
We denote by $H^i(C^\bullet)$ the $i$-th cohomology of \eqref{eq 2.2}, and by $\triangle_i$ the $i$-th Laplacian, i.e.
$\triangle_0=d_1^\ast d_1,
\triangle_1=d_1d_1^\ast+d_2^\ast d_2,
\triangle_2=d_2^\ast d_2$.
Then, since \eqref{eq 2.2} is an elliptic complex, $H^i(C^\bullet)\simeq\mathbb{H}^i:= \mathrm{Ker}\triangle_i$ by the classical Hodge theory \cite{Wells}.
We denote by $H^i$ the Harmonic projection and $G^i$ the Green operator of $\triangle_i$.
We note that $\triangle_i, H^i$ and $G^i$ can be extended to $C^i_k$.\par
Let $[(\nabla_h,\phi)]\in\mathcal{M}^{\mathrm{irr}}_{\tau,k}$.
Since $S_{(\nabla_h,\phi),\epsilon},\,\,0<\epsilon<<1$ is a local coordinate of $\mathcal{B}^{\mathrm{irr}}_k$, $\mathcal{M}^{\mathrm{irr}}_{\tau,k}$ is locally modeled on $S_{(\nabla_h,\phi),\epsilon}\cap \mathrm{Vor}_\tau(h)$.
Let $\alpha=(A,\dot{\phi})\in A^{1}(X,\sqrt{-1}\mathbb{R})\times A(L)$.
Then, $(\nabla_h,\phi)+\alpha\in \mathrm{Vor}_\tau(h)$ if and only if 
\begin{align*}
&dA+\sqrt{-1}\mathrm{Re}h(\dot{\phi},\phi)\omega_X+\frac{\sqrt{-1}}{2}\mathrm{Re}h(\dot{\phi},\dot{\phi})\omega_X=0,\\
&\nabla_h^{0,1}\dot{\phi}+A^{0,1}\phi+A^{0,1}\dot{\phi}.
\end{align*}
We define 
\[
\widetilde{\alpha\wedge\alpha}:=
\begin{pmatrix}
\frac{\sqrt{-1}}{2}\mathrm{Re}h(\dot{\phi},\dot{\phi})\omega_X\\
A^{0,1}\dot{\phi}
\end{pmatrix}.
\]
Therefore, we have 
\[
U_{(\nabla_h,\phi),\epsilon}:=S_{(\nabla_h,\phi),\epsilon}\cap \mathrm{Vor}_\tau(h)_k=\bigg\{\alpha\in S_{(\nabla_h,\phi),\epsilon}\,\bigg|\, d_2\alpha+\widetilde{\alpha\wedge\alpha}=0\bigg\}.
\]
We show that if $H^2(C^\bullet)=0$, $U_{(\nabla_h,\phi),\epsilon}$ is homeomorphic to the neighborhood of $0\in\mathbb{H}^1$ by using the \textit{kuranishi map}.
We define the Kuranishi map
\[
k:C^1_k\to C^1_k
\]
as
\[
k(\alpha):=\alpha+d_2^\ast G_2(\widetilde{\alpha\wedge\alpha}).
\]
Since the derivative of $k$ at the origin is the identity, we can apply the inverse function theorem and show that there exist open neighborhoods $V_1,V_2$ of 0 in $C^1_k$ such that $k:V_1\to V_2$ is a smooth isomorphism.
We assume that $\epsilon$ is small enough so that $S_{(\nabla_h,\phi),\epsilon}\subset V_1$.
Let $\alpha\in C^1_k$.
By the Hodge decomposition, 
\[
d_2\alpha+\widetilde{\alpha\wedge\alpha}=0
\]
is equivalent to 
\begin{align*}
d_2\alpha+d_2d_2^\ast G_2\widetilde{\alpha\wedge\alpha}&=0,\\
H^2(\widetilde{\alpha\wedge\alpha})&=0.
\end{align*}
Therefore, if $\alpha\in U_{(\nabla_h,\phi),\epsilon}$,  we have
\[
d_2 k(\alpha)=d_2\alpha+d_2d_2^\ast G_2\widetilde{\alpha\wedge\alpha}=0.
\]
Moreover, by the definition of $U_{(\nabla_h,\phi),\epsilon}$, $d_1^\ast k(\alpha)=0.$
Hence
\[
k(U_{(\nabla_h,\phi),\epsilon})\subset \mathbb{H}^1.
\]
Let $\beta\in \mathbb{H}^1$ be small enough so that $\alpha=k^{-1}(\beta)$ exist and $\|\alpha\|_{L^2_k}<\epsilon$.
From now on, we assume $H^2(C^\bullet)=0$.
Under this assumption, we show $\alpha\in U_{(\nabla_h,\phi),\epsilon}$.
Since $k(\alpha)=\beta$ and $\beta\in  \mathbb{H}^1$ we have 
\begin{align}
0&=d_1^\ast\beta=d_1^\ast k(\alpha)=d_1^\ast\alpha,\label{eq 2.3}\\
0&=d_2\beta=d_2 k(\alpha)=d_2\alpha+d_2d_2^\ast G_2\widetilde{\alpha\wedge\alpha}\label{eq 2.4}.
\end{align}
Moreover $\alpha$ is smooth since 
\begin{align*}
0=\triangle_1\beta&=\triangle_1(k(\alpha))\\
&=\triangle_1\alpha+d_2^\ast \widetilde{\alpha\wedge\alpha}
\end{align*}
and the elliptic regularity.
Then, by the equation \eqref{eq 2.4} and the Hodge decomposition, we obatin
\be\label{eq 2.5}
d_2\alpha+\widetilde{\alpha\wedge\alpha}=d_2\alpha+d_2d_2^\ast G_2\widetilde{\alpha\wedge\alpha}+ H^2(\widetilde{\alpha\wedge\alpha})=0.
\ee
The last equation follows from the assumption.
Hence by \eqref{eq 2.3} and \eqref{eq 2.5}, $\alpha\in U_{(\nabla_h,\phi),\epsilon}$.
We then conclude that if $[(\nabla_h,\phi)]\in\mathcal{M}^{\mathrm{irr}}_k$ is $H^2(C^\bullet)=0$, then $[(\nabla_h,\phi)]$ has a neighborhood homeomorphic to an Euclidiean space.
\par
So far, we fixed $k$ and did not discuss the dependence of the topology of the moduli space on it.
Let $\mathcal{M}_{k+1}\to\mathcal{M}_k$ be the natrual map.
Although we do not discuss it here, an argument similar to the one above describes the local structure of $\mathcal{M}_k$, which is generally singular. See \cite{DK, Nico}.
Then, we have 
\begin{prop}\label{prop 2.2}
If $k$ is sufficiently large, the natural map $\mathcal{M}_{\tau,k+1}\to\mathcal{M}_{\tau,k}$ is a homemoprhism.
\end{prop}
Hereafter, we omit the subscript $k$ and simply write $\mathcal{M}_\tau$ (resp. $\mathcal{M}^{\mathrm{irr}}_\tau$) for $\mathcal{M}_{\tau,k}$ (resp. $\mathcal{M}^{\mathrm{irr}}_{\tau,k}$).
Let $[(\nabla_h,\phi)]\in \mathcal{M}_\tau$.
Following \cite{Nico}, we say that $(\nabla_h,\phi)$ is \textit{regular} if $H^2(C^\bullet)=0.$
We define 
\[
\mathcal{M}^\ast_\tau:=\bigg\{[(\nabla_h,\phi)]\in \mathcal{M}^{\mathrm{irr}}_\tau\,\bigg|\,(\nabla_h,\phi)\,\, \textrm{is regular}\bigg\}.
\]
Then, by the discussion above, $\mathcal{M}^\ast_\tau$ is a smooth manifold.
Let $[(\nabla_h,\phi)]\in \mathcal{M}^\ast_\tau$.
The dimension of $\mathcal{M}^\ast_\tau$ around $[(\nabla_h,\phi)]$ is $\mathrm{dim}\mathbb{H}^1=\mathrm{dim}H^1(C^\bullet)$.
By the definition of $ \mathcal{M}^\ast_\tau$, $\mathrm{dim}H^1(C^\bullet)$ is the Fredholm index of $d_1^\ast+d_2$.
Since the Fredholm index only depends on the principal part of the differential operator (\cite{Nico 1}), the Fredholm index of $d_1^\ast+d_2$ is the same as the real Fredholm index of 
\[
D:A^1(X,\sqrt{-1}\mathbb{R})\oplus A(L)\to A(X,\sqrt{-1}\mathbb{R})\oplus A^2(X,\sqrt{-1}\mathbb{R})\oplus A^{0,1}(L)\
\]
defined by
\[
D(A,\phi):=(d^\ast A,dA,\nabla^{0,1}_{h}\phi).
\]
The real Fredholm index of $D$ is 
\begin{align*}
\mathrm{ind}_{\mathbb{R}}(D)&=\mathrm{ind}_{\mathbb{R}}(d^\ast\oplus d)+\mathrm{ind}_{\mathbb{R}}(\nabla^{0,1}_h)\\
&=2\mathrm{genus}(X)-2+2(\mathrm{deg}L-\mathrm{genus}(X)+1)\\
&=2\mathrm{deg}L.
\end{align*}
Here, genus($X$) is the genus of $X$.
In the second equation, we used $\mathrm{ind}_{\mathbb{R}}(\nabla^{0,1}_h)=2\cdot \mathrm{ind}_{\mathbb{C}}(\nabla^{0,1}_h)$ and the Riemann-Roch.
Therefore, if $[(\nabla_h,\phi)]\in \mathcal{M}^\ast_\tau$, the dimension of the real vector space $H^1(C^\bullet)$ is $2\mathrm{deg}L$.
As a consequence, we obtain
\begin{thm}\label{thm 2.1}
$\mathcal{M}^\ast_\tau$ is a smooth manifold of real dimension $2\mathrm{deg}L$.
\end{thm}
\subsection{Hamiltonian Structure and Moment Map}\label{sec 2.3}
In this section, we interpret the exotic vortex equation as a moment map equation and describe the moduli space as the quotient of the zero level set of the moment map by the gauge group. 
The discussion in this section is purely formal.
We note that the moment map interpretation of Manton's exotic vortex equation and its generalization was already obtained in the unpublished paper \cite{Tdl}.
Since this is a natural viewpoint on the equation, we include the discussion here for completeness.
\par
Let $X$ be a compact Riemann surface with a K\"ahler form $\omega_X$ and $L$ be a line bundle over $X$ with a Hermitian metric $h$.
We set
\begin{align*}
\mathcal{A}&:=\mathcal{A}_h\times A(L),\\
\mathcal{G}&:=A(X,S^1).
\end{align*}
Recall that we defined a right $\mathcal{G}$ action on $\mathcal{A}$ in \eqref{eq 2.1}.
Let $(\nabla_h,\phi)\in \mathcal{A}$. 
Then, $T_{(\nabla_h,\phi)}\mathcal{A}=A^1(X,\sqrt{-1}\mathbb{R})\times A(L)$.
Let $\alpha=(A,\dot{\phi}), \beta:=(B,\dot{\psi})\in T_{(\nabla_h,\phi)}\mathcal{A}$.
We define a bilinear form $\omega$ on $T_{(\nabla_h,\phi)}\mathcal{A}$ as
\[
\omega(\alpha,\beta):=-\int_XA\wedge B+\int_X\mathrm{Im}h(\dot{\phi},\dot{\psi}).
\]
We note that the signature of the term $\int_X\mathrm{Im}h(\dot{\phi},\dot{\psi})$ differs from that in the symplectic structure considered in \cite{GP3}.
 This difference reflects the difference in sign between Manton's exotic vortex equation and the classical vortex equation.
 $\omega$ is a non-degenerate skew-symmetric form.
 Since $\omega$ does not depend on the point $(\nabla_h,\phi)$, it is closed, and hence it is a symplectic form on $\mathcal{A}$.
 The action on $\mathcal{G}$ on $(\mathcal{A},\omega)$ is symplectic.
 This follows since $g\in \mathcal{G}$ maps $\alpha=(A,\dot{\phi})\in T_{(\nabla_h,\phi)}\mathcal{A}$ to $\alpha'=(g^{-1}Ag,g^{-1}\phi)\in T_{(\nabla_h,\phi)\cdot g}\mathcal{A}.$\par
 The Lie algebra $\mathfrak{g}$ of $\mathcal{G}$ is naturally identified with $A(X,\mathfrak{u}(1))=A(X,\sqrt{-1}\mathbb{R})$.
 We identify the dual of $\mathfrak{g}$ with $A^{2}(X,\sqrt{-1}\mathbb{R})$ by the pairing 
 \[
 <A,\alpha>:=-\int_X A\alpha,\qquad A\in A(X,\sqrt{-1}\mathbb{R}), \alpha\in A^2(X,\sqrt{-1}\mathbb{R}).
 \]
 We show that the action of $\mathcal{G}$ on $(\mathcal{A},\omega)$ is Hamiltonian.
  Let $\tau\in\mathbb{R}$.
 We define a map
 \[
 \mu:\mathcal{A}\to A^2(X,\sqrt{-1}\mathbb{R})
 \]
 as
 \[
 \mu(\nabla_h,\phi):=F_{\nabla_h}+\frac{\sqrt{-1}}{2}|\phi|^2_h+\frac{\sqrt{-1}}{2}\tau.
 \]
 We show that $\mu$ is a moment map.
We first show that $\mu$ is $\mathcal{G}$-equivariant.
 Let $g\in\mathcal{G}$.
 Then, we have  
 \begin{align*}
  \mu ((\nabla_h,\phi)\cdot g)&=\mu(g^{-1}\nabla_h g,g^{-1}\phi)\\
  &=g^{-1}F_{\nabla_h}g+\frac{\sqrt{-1}}{2}|g^{-1}\phi|^2_h+\frac{\sqrt{-1}}{2}\tau\\
  &=F_{\nabla_h}+\frac{\sqrt{-1}}{2}|\phi|^2_h+\frac{\sqrt{-1}}{2}\tau\\
  &=\mu(\nabla_h,\phi).
       \end{align*}
 Let $X\in \mathfrak{g}=A(X,\sqrt{-1}\mathbb{R})$ and $X^\#$ be the fundamental vector field on $\mathcal{A}$ generated by $X$.
 We show that
 \[
 \mu^X:\mathcal{A}\to \mathbb{R},\qquad \mu^X((\nabla_h,\phi))=<\mu((\nabla_h,\phi)),X>
 \] 
 is Hamiltonian, i.e.
 \[
 d\mu^X=\iota_{X^\#}\omega
 \]
Here $\iota_{X^\#}\omega$ is contraction of $\omega$ with respect to $X^\#$.
By the definition of the action of $\mathcal{G}$ on $\mathcal{A}$, we have 
\begin{align*}
X^\#_{(\nabla_h,\phi)}&=\dfrac{d}{dt}(\exp(-tX)\nabla_h\exp(tX),\exp(-tX)\phi)\bigg|_{t=0}\\
&=(dX,-X\phi).
\end{align*}
Let $(A,\dot{\phi})\in T_{(\nabla_h,\phi)}\mathcal{A}=A^1(X,\sqrt{-1}\mathbb{R})\times A(L)$.
We have 
\begin{align*}
\iota_{X^\#}\omega((A,\dot{\phi}))&=\omega(X^\#_{(\nabla_h,\phi)}, (A,\dot{\phi}))\\
&=-\int_XdX\wedge\alpha +\int_X \mathrm{Im}h(-X\phi,\dot{\phi})\\
&=\int_X XdA+\sqrt{-1}\int_X X\mathrm{Re}h(\phi,\dot{\phi}).
\end{align*}
The last equation follows from $X\in A(X,\sqrt{-1}\mathbb{R})$.
We also have
\begin{align*}
d\mu^X((A,\dot{\phi}))&=\frac{d}{dt}<\mu(\nabla_h+tA,\phi+t\dot{\phi}),X>\bigg|_{t=0}\\
&=\frac{d}{dt}\int_X\bigg\{F_{\nabla_h+tA}+\frac{\sqrt{-1}}{2}|\phi+t\dot{\phi}|^2_h\omega_X+\frac{\sqrt{-1}\tau}{2}\omega_X\bigg\}X\bigg|_{t=0}\\
&=\int_X \bigg\{ dA+\sqrt{-1}\mathrm{Re}h(\phi,\dot{\phi})\bigg\}X.
\end{align*}
Hence 
\[
 d\mu^X=\iota_{X^\#}\omega
 \]
holds and therefore $\mu$ is a moment map.
We define a subset of $\mathcal{A}$ as 
\[
\mathcal{N}:=\bigg\{ (\nabla_h,\phi)\in\mathcal{A}\,\,\bigg|\,\, \nabla^{0,1}_{h}\phi=0\bigg\}.
\]
The moduli space of Manton's exotic vortex equation can also be described as
\[
\mathcal{M}_\tau=\mathcal{N}\cap\mu^{-1}(0)/\mathcal{G}.
\]
\section{Dimensional Reduction}\label{sec 3}
In this section, we establish the dimensional reduction.
We derive solutions of Manton's exotic vortex equation from pseudo-Hermitian-Yang-Mills connections (see Definition \ref{def 3.2}) with pseudo-Hermitian metrics of signature (1,1) (see Definition \ref{def 3.1}).
We also construct a pseudo-Hermitian-Yang-Mills connection together with a pseudo-Hermitian metric of signature $(1,1)$ from solutions of Manton's exotic vortex equation.
\par
In Section \ref{sec 3.1}, we collect several preliminary results.
We establish the dimensional reduction for $\tau>0$ in Section \ref{sec 3.2}, for $\tau<0$ in Section \ref{sec 3.3}, for $\tau=0$ in Section \ref{sec 3.4}.
\subsection{Preliminary}\label{sec 3.1}
We prepare some notations and collect results for later use.
\subsubsection{Pseudo-Hermitian Metric of Signature $(p,q)$}\label{sec 3.1.1}
Let $M$ be a complex manifold with the K\"ahler form $\omega_M$, $E\to M$ be a complex vector bundle of rank $r$, and $GL(E)$ be the principal $GL(r)$-bundle associated to $E$.
We denote the complex conjugate of $E$ as $\overline E.$
\begin{de}\label{def 3.1}
A Hermitian form $h$ on $E$ is a section $h\in A(E^\vee\otimes \overline{E}^\vee)$ such that for every $x\in M$, $h|_{E_x}:E_x\times E_x\to \mathbb{C}$ is a Hermitian form.\par
We say $h$ is a pseudo-Hermitian metric if it is a non-degenerate Hermitian form on $E$.
We say that $h$ is a pseudo-Hermitian metric of signature $(p,q)$ (with $(p+q=r)$) if for each $x\in M$, $h|_{E_x}$ has signature $(p,q)$. 
Equivalently, $h$ defines a $U(p,q)$-reduction of $GL(E)$.
Here $U(p,q)$ is the indefinite unitary group.
\end{de}
We note that a Hermitian metric on $E$ is a pseudo-Hermitian metric of signature $(r,0)$.\par
Let $h$ be a pseudo-Hermitian metric of signature $(p,q)$.
We say that a connection $D$ of $E$ is compatible with $h$ if 
\[
dh(\cdot,\cdot)=h(D\cdot,\cdot)+h(\cdot,D\cdot).
\]
Let $\overline\partial_E$ be a holomorphic structure of $E$ and $h$ be a Hermitian metric on $E$.
Recall that a Chern connection $\nabla_h$ is the unique $h$-unitary connection such that $\nabla^{0,1}_h=\overline\partial_E$.
Since the construction of Chern connections (see \cite[Proposition 1.4.9]{Ko} for example) only uses the non-degeneracy of $h$, we have 
\begin{prop}\label{prop 3.1}
Let $\overline\partial_E$ be a holomorphic structure of $E$ and $h$ be a pseudo-Hermitian metric of signature $(p,q)$.
Then there exists a unique connection $\nabla_h$ such that 
\begin{itemize}
\item[(1)] $\nabla_h$ is a connection compatible with $h$,
\item[(2)] $\nabla^{0,1}_h=\overline\partial_E$.
\end{itemize}
\end{prop} 
\begin{de}\label{def 3.2}
Let $\overline\partial_E$ be a holomorphic structure on $E$, and $h$ be a pseudo-Hermitian metric of signature $(p,q)$.
Let $\nabla_h$ be the connection which is compatible with $h$ and $\nabla_h^{0,1}=\overline\partial_E$, and let $F_h:=(\nabla_h)^2$ be the curvature.
We say that $h$ is a pseudo-Hermitian-Einstein metric of signature $(p,q)$ on $(E,\overline\partial_E)$ if there exists a constant $\lambda$ such that 
\be\label{eq 3.1}
\sqrt{-1}\Lambda_{\omega_M} F_h=\lambda \mathrm{Id}_E.
\ee
In this case, we call $\nabla_h$ a pseudo-Hermitian-Yang-Mills connection.
We call the equation \eqref{eq 3.1} the Hermitian-Yang-Mills equation.
\end{de}
\begin{remark}
Let $(M,J)$ be an almost complex manifold and $E\to M$ be a complex vector bundle.
Let $g$ be a Riemannian metric of $M$ compatible with $I$, and $\omega$ be the associated fundamental form.
Let $\nabla$ be a connection of $E$ and $F=(\nabla)^2$ be its curvature.
In \cite{Bry}, the connection $\nabla$ is defined to be
pseudo-Hermitian-Yang-Mills if $F^{0,2}=0$ and
 \[
\sqrt{-1}\Lambda_{\omega}F=\lambda\operatorname{Id}_E
\]
for some constant $\lambda$.\par
In our setting, $F^{0,2}=0$ is automatically satisfied since $\nabla_h^{0,1}=\overline\partial_E$.
\end{remark}
\subsubsection{Equivariant Line Bundles}\label{sec 3.1.2}
We collect some results of $G$-equivariant bundles over homogeneous spaces $G/K$.
The particular Lie groups which we will use later are $G=SU(2), SE(2), SU(1,1)$ and $K=U(1)$.
The result of this section is well-known, and there are no new results.
For a thorough treatment of the results in this section, see \cite{Akh, KN, Se}.\par
We first recall the definition of $G$-equivariant vector bundles.
We also recall the definition of $G$-equivariant holomorphic vector bundle.
Although we do not encounter $G$-equivariant holomorphic bundles in this section, we will use them later.
\begin{de}
Let $M$ be a smooth manifold with a Lie group $G$ acting on it.
Let $\pi:E\to M$ be a smooth vector bundle over $M$.
We say that $E$ is a $G$-equivariant vector bundle if 
\begin{itemize}
\item[(1)] $G$ acts on $E$, 
\item[(2)] $\pi(g\cdot v)=g\cdot \pi(v)$, 
\item[(3)] for every $g\in G$, $g\cdot : \pi^{-1}(x)\to \pi^{-1}(g\cdot x)$ is a linear isomorphism.
\end{itemize}
\par
Now suppose $M$ is a complex manifold and $G$ is acting holomorphically (i.e every $g\in G$ induces a biholomorphic map $g:M\to M$).
Let $E$ be a holomorphic bundle over $M$ and also a $G$-equivariant bundle.
We say that $E$ is a  $G$-equivariant holomorphic bundle if every $g:E\to E$ is a biholomorphic map.
\end{de}
We also recall the definition of $G$-invariant Hermitian forms.
\begin{de}
Let $M$ be a smooth manifold with a Lie group $G$ acting on it, and $E\to M$ is a $G$-equivariant vector bundle.
Let $h$ be a Hermitian form of $E$ (See Definition \ref{def 3.1}).
Then, we say that $h$ is a $G$-invariant Hermitian form if
\[
h_{gx}(g\cdot,g\cdot)=h_x(\cdot,\cdot).
\] 
\end{de}
Let $G$ be a Lie group with a compact subgroup $U(1)\subset G$, and $\mathbb{C}^\ast:=\mathbb{C}-{0}$.
Let $\chi:U(1)\to \mathbb{C}^\ast$ be a continuous (or equivalently, smooth) representation.
Then $k\in U(1)$ acts on $(g,v)\in G\times\mathbb{C}$ from the right as 
\[
 (g,v)\cdot k:=(gk,\chi^{-1}(k)v).
\]
We denote the quotient space by
\[
 G\times_\chi\mathbb{C}:=G\times \mathbb{C}/U(1).
\]
Then $G\times_\chi\mathbb{C}$ is a $G$-equivariant complex line bundle over $G/U(1)$.
Conversely, if $L$ is a  $G$-equivariant line bundle over $G/U(1)$, then there exists a continuous representation 
$\chi:U(1)\to\mathbb{C}^\ast$ such that 
\[
L\simeq G\times_\chi\mathbb{C}
\]
as $G$-equivariant vector bundles.
This follows from the following construction.
Let $e\in G$ be the unit and $L_e$ be the fibre of $L$ at $eU(1)$.
Since $U(1)$ stabilizes $eU(1)$, it acts on $L_e$ and hence defines a representation $\chi_L:U(1)\to GL(L_e)=\mathbb{C}^\ast$.
Then we obtain a line bundle $G\times_{\chi_L}L_e$ over $G/U(1)$.\par
We denote the equivalent class of $(g,v)\in G\times L_e$ in $G\times_{\chi_L}L_e$ as $[g,v]$.
Let $g_0\in G$ and $[g,v]\in G\times_{\chi_L}L_e$.
$G$ acts on $G\times_{\chi_L}L_e$ as 
\[
g_0\cdot [g,v]:=[g_0g,v].
\]
$G\times_{\chi_L}L_e$ becomes a $G$-equivariant line bundle with respect to this action.
We define a map 
\[
\Phi: G\times_{\chi_L}L_e\to L
\]
as 
\[
\Phi([g,v]):=g\cdot v.
\]

$\Phi$ is a smooth $G$-equivariant vector bundle isomorphism.
Hence, every $G$-equivariant line bundle over $G/U(1)$ is determined by a representation $\chi:U(1)\to\mathbb{C}^\ast$.
It is a classical fact that every continuous representation $\chi: U(1) \to \mathbb{C}^*$ is of the form
\[
\chi(e^{i\theta}) = e^{im\theta}
\]
for some integer $m \in \mathbb{Z}$ (See \cite{Ha}, for example).
The integer $m$ is called the weight of $\chi.$
We denote the weight $m$ representation as $\chi_m$.
Hence, every $G$-equivariant line bundle $L$ is determined by a weight $m$, which we call the weight of $L$.
We denote the $G$-equivariant line bundle of weight $m$ over $G/U(1)$ as $L_m$.\par
Let $L_n$ and $L_m$ be the weight $n$ and $m$ $G$-equivariant line bundle over $G/U(1)$.
From the construction of $L_n$ and $L_m$, the tensor product $L_n\otimes L_m$ is a $G$-equivariant line bundle with the weight $n+m$.
Let $L^\vee_n$ be the dual bundle of $L_n$. 
Then $L_n^\vee$ carries the canonical $G$-equivariant structure induced by that of $L_n$.
Moreover, since the dual representation has the inverse weight, $L_n^\vee$ has weight $-n$.
\par
We give a summary of the result by the discussion so far for our convenience.
We again note that the result is nothing new and is well-known.
\begin{lem}\label{lem 3.1}
Let $G$ be a Lie group with compact subgroup $U(1)\subset G$.
For every $m\in\mathbb{Z}$, we denote $\chi_m:U(1)\to \mathbb{C}^\ast$ to be the weight $m$ representation.\par
Let $L$ be a $G$-equivariant line bundle over $G/U(1)$.
Then, there exists a integral $m\in\mathbb{Z}$ such that 
\[
L\simeq G\times_{\chi_m}\mathbb{C}
\]
as $G$-equivariant line bundles.
We call $m$ the weight of $L$.\par
Let $L_n$ and $L_m$ be the $G$-equivariant bundle over $G/U(1)$ with weight $n$ and $m$.
Then the tensor product $L_n\otimes L_m$ is a $G$-equivariant bundle with the weight $n+m$.\par
Let $L^\vee_n$ be the dual bundle of $L_n$.
Then $L_n^\vee$ is a $G$-equivariant bundle with weight $-n$.
\end{lem}
\par
We prepare a result for $G$-equivariant sections for later use.
We first recall the definition of it.
\begin{de}
Let $M$ be a smooth manifold with a Lie group $G$ acting on it, and $E$ be a $G$-equivariant vector bundle over $M$.
We say that a section $s\in A(E)$ is $G$-equivariant if for every $g\in G$ and for every $x\in M$
\[
g\cdot s(x)=s(g\cdot x)
\]
holds.
We denote the space of $G$-equivariant section of $E$ as $A(E)^G$.
\end{de}
\begin{lem}\label{lem 3.2}
Let $G$ be a Lie group with a compact subgroup $U(1)\subset G$, and  $L_m$ be the weight $m$ $G$-equivariant line bundle over $G/U(1)$.
Then 
\begin{itemize}
\item[(1)] If $m\neq 0$, then $A(L_m)^G$ only contains zero section.
\item[(2)] If $m=0$, then $A(L_0)\simeq \mathbb{C}$.
\end{itemize} 
\end{lem}
\begin{proof}
Let $s\in A(L_m)^G$.
We first note that if there exists an $x\in G/U(1)$ such that $s(x)=0$, then $s$ is the zero section.
This follows from the following argument: Let $x\in G/U(1)$ and assume $s(x)=0.$
Since the action of $G$ on $G/U(1)$ is transtive, for every $y\in G/U(1)$, there exists a $g\in G$ such that $y=g\cdot x$.
By the $G$-equivariance of $s$, we have 
\[
s(y)=s(g\cdot x)=g^{-1}\cdot s(x)=0.
\]
Therefore, $s$ is the zero section.\par
We now assume that $s$ is not the zero section.
Let $e$ be the unit of $G$.
By the argument above, $s$ is everywhere non-zero.
In particular $s([e])\neq 0$.
Let $e^{i\theta}\in U(1)$.
By the discussion in this section, $e^{i\theta}$ acts on $s([e])$ as 
\[
e^{i\theta}\cdot s([e])=\chi_m(e^{i\theta})s([e])=e^{im\theta}s([e]).
\]
Since $s$ is $G$-equivariant and $U(1)$ stabilize $[e]$, we have 
\[
s([e])=e^{i\theta}\cdot s([e])=e^{im\theta}s([e])
\]
However, this equality does not hold unless $m=0$.
Hence (1) is proved.\par
We next prove $(2)$.
Since the weight $0$ representation $\chi_0$ is the trivial representation, $L_0$ is the trivial line bundle with the canonical $G$-action over $G/U(1)$.
Therefore $A(L_0)^G$ is a vector space of $G$-equivariant smooth function of $G/U(1)$.
Since the action of $G$ on $G/U(1)$ is transitive, a $G$-equivariant smooth function is just a constant function.
Hence $(2)$ is proved.
\end{proof}

We also prepare a result for $G$-invariant Hermitian forms.

\begin{lem}\label{lem 3.3}
Let $G$ be a Lie group with a compact subgroup $U(1)\subset G$, and  $L_m$ be the weight $m$ $G$-equivariant line bundle over $G/U(1)$.
If there exists a $G$-invariant pseudo-Hermitian metric on $L_m$, then it is unique up to a non-zero real constant.
\end{lem}
\begin{proof}
Suppose $h_1$ and $h_2$ are $G$-invariant pseudo-Hermitian metric on $L_m$.
We denote by $e$ the identity element of $G$.
Since $L_m$ is a line bundle,  there exists an $a\in\mathbb{R}^\times$ such that 
\[
h_{1,eU(1)}=ah_{2,eU(1)}.
\]
Then, by $G$-invariance of $h_1$ and $h_2,$ we have 
\begin{align*}
h_{1,gU(1)}(\cdot,\cdot)&=h_{1,gU(1)}(g\cdot g^{-1}\cdot, g\cdot g^{-1}\cdot)\\
&=h_{1,eU(1)}(g^{-1}\cdot,g^{-1}\cdot)\\
&=ah_{2,eU(1)}(g^{-1}\cdot,g^{-1}\cdot)\\
&=ah_{2,gU(1)}(\cdot,\cdot).
\end{align*}
Therefore 
\[
h_1=ah_2.
\]
\end{proof}
 \subsubsection{Degree of Vector Bundles}\label{sec 3.1.4}
 In this section, we recall the degree of complex vector bundles over compact K\"ahler manifolds based on \cite{Wells}.
 Let $M$ be a compact K\"ahler manifold with a K\"ahler form $\omega_M$, and $E\to M$ be a complex vector bundle.
 Let $D$ be a connection of $E$, and $F_D$ be the curvature of $D$.
 Then 
  \[
 \frac{\sqrt{-1}}{2\pi}\mathrm{Tr}F_D
 \]
 is a closed 2-form on $M$, and the cohomology class of it does not depend on the choice of a connection.
 We denote this class as
 \[
  c_1(E):=\bigg[\frac{\sqrt{-1}}{2\pi}\mathrm{Tr}F_D\bigg]\in H^2(M,\mathbb{C}).
  \]
We define the degree of $E$ as
 \begin{align*}
 \mathrm{deg}_{\omega_M}E:&=\frac{1}{(\mathrm{dim}M-1)!}\int_Mc_1(E)\wedge \omega^{\mathrm{dimM-1}}_M\\
 &=\frac{\sqrt{-1}}{2\pi}\int_M\Lambda_{\omega_M}\mathrm{Tr}F_D\mathrm{Vol}_{\omega_M},
  \end{align*}
 where
 \[
 \mathrm{Vol}_{\omega_M}:=\frac{\omega^{\mathrm{dim}M}}{(\mathrm{dim}M)!}.
  \]
We note that the degree is independent of the choice of the connection $D$, but it depends on the K\"ahler class $[\omega_M]$.  

\subsection{$\tau>0$ case}\label{sec 3.2}
\subsubsection{$SU(2)$-equivariant line bundle over $\mathbb{P}^1$}\label{sec 3.2.1}
In this section, we review the $SU(2)$-action on $\mathbb{P}^1$ and $\mathcal{O}_{\mathbb{P}^1}(n)$.\par
Let $M_2(\mathbb{C})$ be the set of $2\times 2$ complex valued matrixes.
 Recall that 
\begin{equation*}
SU(2)=\bigg\{
\begin{pmatrix}
a&-\overline{b}\\
b&\overline{a}
\end{pmatrix}\in M_2(\mathbb{C})
\,\,
\bigg|
\,\,
 |a|^2+|b|^2=1\bigg\}.
\end{equation*}
$SU(2)$ acts on $\mathbb{P}^1$ as 
\begin{equation*}
\begin{array}{rccc}
& SU(2)\times\mathbb{P}^1  &\longrightarrow& \mathbb{P}^1                     \\
        & \rotatebox{90}{$\in$}&               & \rotatebox{90}{$\in$} \\
        &\bigg( \begin{pmatrix}
a&-\overline{b}\\
b&\overline{a}
\end{pmatrix}   , [z_0:z_1]   \bigg)              & \longmapsto   &  [az_0-\overline{b}z_1:bz_0+\overline{a}z_1].
\end{array}
\end{equation*}
It is well known that this action is transitive and we can regard  $\mathbb{P}^1$ as a homogeneous space $SU(2)/U(1)$.\par
We now regard $\mathbb{P}^1$ as a two-copy of complex planes $\mathbb{C}_z$ and $\mathbb{C}_w$ glued with the relation $w=1/z$.
Then the action of $g= \begin{pmatrix}
a&-\overline{b}\\
b&\overline{a}
\end{pmatrix}$
on $\mathbb{P}^1$ can be realized as 
$z\mapsto\frac{az-\overline{b}}{bz+\overline{a}}$.\par
Let $\mathbb{C}_z\times\mathbb{C}\to\mathbb{C}_z$ and $\mathbb{C}_w\times\mathbb{C}\to\mathbb{C}_w$ be the trivial line bundle on $\mathbb{C}_z$ and $\mathbb{C}_w$.
Let $e_z$ and $e_w$ be the canonical global sections of them. 
The line bundle $\mathcal{O}(-1)$ is obtained by gluing the two trivial line bundles by the relation $e_w=\frac{e_z}{z}$.
$\mathcal{O}(-1)$ is called the taotological bundle.
From the construction, $\mathcal{O}(-1)$ is a holomorphic bundle, and we denote the Dolbeaut operator of it as $\overline\partial_{\mathcal{O}(1)}$.
\par
Since $\mathcal{O}(-1)$ is also a sub-bundle of the rank 2 trivial bundle $\mathbb{P}^1\times \mathbb{C}^2\to\mathbb{P}^1$,
the standard action of $SU(2)$ on $\mathbb{C}^2$ induces an $SU(2)$-action on $\mathcal{O}(-1)$ which is compatible with the $SU(2)$-action on $\mathbb{P}^1$. 
We write down the action explicitly.  
For each 
$g= \begin{pmatrix}
a&-\overline{b}\\
b&\overline{a}
\end{pmatrix}$,
it induces a map $\Psi_g:\mathcal{O}(-1)\to \mathcal{O}(-1)$ such that 
\begin{equation*}
\Psi_g(e_z)=(bz+\overline{a})e_{g\cdot z}.
\end{equation*}
We define a hermitian metric $h^{(-1)}$ on $\mathcal{O}(-1)$ such that 
\begin{equation*}
h^{(-1)}(e_z,e_z):=1+|z|^2, h^{(-1)}(e_w,e_w):=1+|w|^2.
\end{equation*}
This metric is $SU(2)$-invariant: $h^{(-1)}(\Psi_g\cdot,\Psi_g\cdot)=h^{(-1)}(\cdot,\cdot)$ holds for every for every $g\in SU(2)$.\par
We define $\mathcal{O}(1)$ to be the dual bundle of $\mathcal{O}(-1)$. 
For every $n\in\mathbb{Z}_{\geq 0}$, we define $\mathcal{O}(n):=\mathcal{O}(1)^{\otimes n}$ and $\mathcal{O}(-n):=\mathcal{O}(-1)^{\otimes n}$. 
Each $\mathcal{O}(n)\, (n\in\mathbb{Z})$ has a holomorphic bundle structure which we denote as $\overline\partial_{\mathcal{O}(n)},$ and has an $SU(2)$-action which are both induced by $\mathcal{O}(-1)$.
We note that $(\mathcal{O}(n),\overline\partial_{\mathcal{O}(n)})$ is an $SU(2)$-equivariant holomorphic bundle.\par
The metric $h^{(-1)}$ induces a Hermitian metric $h^{(n)}$ on $\mathcal{O}(n)$ and therefore it is $SU(2)$-invariant.
 We note that $h^{(2)}$ is the Fubini-Study metric.\par
 For later use, we recall the $SU(2)$-invariant differential form valued on $\mathcal{O}(-2)$ and $\mathcal{O}(2)$.
  Before we proceed, we prepare some notations. 
  Let $T^{1,0\ast}\mathbb{P}^1$ be the holomorphic cotangent bundle on $\mathbb{P}^1$. 
  Note that $\mathcal{O}(-2)\simeq T^{1,0\ast}\mathbb{P}^1$ as holomorphic bundles.
  Let $T^{0,1\ast}\mathbb{P}$ be the anti-holomorphic cotangent bundle.
  
  \begin{lem}\label{lem 3.4}
  \begin{equation*}
  T^{0,1\ast}\mathbb{P}\simeq \mathcal{O}(2)
    \end{equation*}
    as smooth bundles.
  \end{lem}
  \begin{proof}
  Although this is well known, we give a proof for convenience.\par
 Let $D$ be a connection of $T^{1,0\ast}\mathbb{P}^1$ and $F_D$ be the curvature of $D$. Then
\begin{equation*}
\mathrm{deg}(T^{1,0\ast}\mathbb{P}^1)=\frac{\sqrt{-1}}{2\pi}\int_{\mathbb{P}^1}\mathrm{Tr}\Lambda F_D.
\end{equation*}
This is independent of $D$ (See \cite{Wells}, for example). 
Since $T^{0,1\ast}\mathbb{P}$ is the complex conjugate of $T^{1,0\ast}\mathbb{P}^1$,
the connection $D$ induces a connection $\overline{D}$ on $T^{0,1\ast}\mathbb{P}$. 
The curvatures of $D$ and $\overline{D}$ are related as $F_{\overline{D}}=\overline{F_D}$.
 Hence
 \begin{align*}
 \mathrm{deg}(T^{0,1\ast}\mathbb{P})&=\frac{\sqrt{-1}}{2\pi}\int_{\mathbb{P}^1}\mathrm{Tr}\Lambda F_{\overline{D}}\\
 &=\frac{\sqrt{-1}}{2\pi}\int_{\mathbb{P}^1}\mathrm{Tr}\Lambda\overline{F_D}\\
 &=-\overline{\frac{\sqrt{-1}}{2\pi}\int_{\mathbb{P}^1}\mathrm{Tr}\Lambda F_D}\\
 &=-\overline{\mathrm{deg}(T^{1,0\ast}\mathbb{P}^1)}.
     \end{align*}
Since $\mathrm{deg}(\mathcal{O}(-1))=-1$ and $\mathcal{O}(-2)\simeq T^{1,0\ast}\mathbb{P}^1$, $\mathrm{deg}(T^{1,0\ast}\mathbb{P}^1)=-2$ and hence $\mathrm{deg}(T^{0,1\ast}\mathbb{P})=2$.
Recall that a smooth complex line bundle over $\mathbb{P}^1$ is uniquely determined by its degree. 
Since $\mathrm{deg}(\mathcal{O}(2))=2$, $T^{0,1\ast}\mathbb{P}\simeq \mathcal{O}(2)$ as smooth bundles.
  \end{proof}
We denote the space of smooth $SU(2)$-invariant section of $\mathcal{O}(n)$ as $A(\mathcal{O}(n))^{SU(2)}$, smooth $SU(2)$-invariant $\mathcal{O}(n)$-valued 1-form as $A^1(\mathcal{O}(n))^{SU(2)}$.

\begin{prop}\label{prop 3.2}
\begin{itemize}
\item[(1)] For $n\neq 0$, $A(\mathcal{O}_{\mathbb{P}^1}(n))^{SU(2)}$ only contains zero section. 
\item[(2)] $\mathrm{dim}A^1(\mathcal{O}_{\mathbb{P}^1}(-2))^{SU(2)}=\mathrm{dim}A^1(\mathcal{O}_{\mathbb{P}^1}(2))^{SU(2)}=1$. 
In particular, let $\alpha$ and $\beta$ be the basis of $A^1(\mathcal{O}_{\mathbb{P}^1}(-2))^{SU(2)}$ and $A^1(\mathcal{O}_{\mathbb{P}^1}(2))^{SU(2)}$. 
Then $\alpha\in A^{0,1}(\mathcal{O}_{\mathbb{P}^1}(-2))$ and $\beta\in A^{1,0}(\mathcal{O}_{\mathbb{P}^1}(2)).$
\end{itemize}
\end{prop}
\begin{proof}
This is also well known, but we give a proof for convenience.\par
(1) Suppose we have a non-zero section $a\in A(\mathcal{O}_{\mathbb{P}^1}(n))^{SU(2)}$. 
Since $a$ is $SU(2)$-invariant and $SU(2)$ acts transitive on $\mathbb{P}^1$, $a$ is everywhere non-zero.
However, this is a contradiction because if such $a$ exists, then $a$ trivialize $\mathcal{O}(n)$ but this contradicts to $\mathrm{deg}(\mathcal{O}_{\mathbb{P}^1}(n))=n$.\par
(2) First, we have 
\begin{equation}\label{eq 3.2}
\begin{split}
A^1(\mathcal{O}(-2))&=A^{1,0}(\mathcal{O}(-2))\oplus A^{0,1}(\mathcal{O}(-2))\\
&=A(T^{1,0\ast}\mathbb{P}^1\otimes\mathcal{O}(-2))\oplus A(T^{0,1\ast}\mathbb{P}\otimes \mathcal{O}(-2))\\
&=A(\mathcal{O}(-2)\otimes\mathcal{O}(-2))\oplus A(\mathcal{O}(2)\otimes \mathcal{O}(-2))\\
&=A(\mathcal{O}(-4))\oplus C^{\infty}(\mathbb{P}^1).
\end{split}
\end{equation}
Here $C^{\infty}(\mathbb{P}^1)$ is the space of the complex valued smooth functions on $\mathbb{P}^1$.
Let $C^{\infty}(\mathbb{P}^1)^{SU(2)}$ be the space of $SU(2)$-invariant smooth functions. 
Then
\begin{align*}
A^1(\mathcal{O}(-2))^{SU(2)}&=A(\mathcal{O}(-4))^{SU(2)}\oplus C^{\infty}(\mathbb{P}^1)^{SU(2)}\\
&=C^{\infty}(\mathbb{P}^1)^{SU(2)}.
\end{align*}
The second equation follows from $(1)$.
Since $SU(2)$ acts transitive on $\mathbb{P}^1$, $C^{\infty}(\mathbb{P}^1)^{SU(2)}=\mathbb{C}$ and therefore $\mathrm{dim} A^1(\mathcal{O}(-2))^{SU(2)}=1$. 
By using a similar argument, we can also prove $A^1(\mathcal{O}(-2))^{SU(2)}=C^{\infty}(\mathbb{P}^1)^{SU(2)}$ and hence $\mathrm{dim} A^1(\mathcal{O}(2))^{SU(2)}=1$.
By the equation (\ref{eq 3.2}), we have 
\begin{equation*}
A^1(\mathcal{O}(-2))^{SU(2)}=A(T^{0,1\ast}\mathbb{P}\otimes \mathcal{O}(-2))^{SU(2)}=A^{0,1}(\mathcal{O}(-2))^{SU(2)}.
\end{equation*}
We also have
\begin{equation*}
A^1(\mathcal{O}(2))^{SU(2)}=A(T^{1,0\ast}\mathbb{P}\otimes \mathcal{O}(2))^{SU(2)}=A^{1,0}(\mathcal{O}(2))^{SU(2)}.
\end{equation*}
Hence, we proved (2).\par
We give explicit realizations $\alpha$ and $\beta$. 
Let $\mathbb{C}_z\times\mathbb{C}\to\mathbb{C}_z$ and $\mathbb{C}_w\times\mathbb{C}\to\mathbb{C}_w$ be the trivial line bundle on $\mathbb{C}_z$ and $\mathbb{C}_w$.
Let $e_{n,z}$ and $e_{n,w}$ be the canonical global sections.
$\mathcal{O}(n)$ is constructed by gluing the two trivial line bundles by $e_{n,w}=z^ne_{n,z}$. 
$\alpha$ and $\beta$ are realized as follows:
\begin{align*}
&\alpha|_{\mathbb{C}_z}=\frac{1}{(1+|z|^2)^2}d\overline{z}\otimes e_{-2,z},\,\,\alpha|_{\mathbb{C}_w}=-\frac{1}{(1+|w|^2)^2}d\overline{w}\otimes e_{-2,w}.\\
&\beta|_{\mathbb{C}_z}=dz\otimes e_{2,z},\,\,\beta|_{\mathbb{C}_w}=-dw\otimes e_{2,w}.
\end{align*}
Then $\alpha$ and $\beta$ are indeed bases of $A^1(\mathcal{O}(-2))^{SU(2)}$ and $A^1(\mathcal{O}(2))^{SU(2)}$. This follows from $\alpha$ and $\beta$ are everywhere non-zero and $SU(2)$-invariant.
\end{proof}

\subsubsection{Dimensional Reduction}\label{sec 3.2.2}
Let $X$ be a compact Riemann surface, and $\omega_X$ be a K\"ahler form on $ X$.
We normalize $\omega_X$ so that 
\[
\mathrm{Vol}_{\omega_X}=\int_X\omega_X=1.\]
We now consider the action of $SU(2)$ on $X\times \mathbb{P}^1$ which is trivial on $X$ and standard on $\mathbb{P}^1$.
In the first half of this section, starting from a solution of Manton's exotic $\tau$-vortex equation with $\tau>0$ on $X$, we construct an $SU(2)$-equivariant holomorphic bundle over $X\times \mathbb{P}^1$ and an $SU(2)$-invariant pseudo-Hermitian-Einstein metric of signature $(1,1)$ on it.
In the second half, we establish the converse.
\par
From now on, we assume $\tau>0$, and we denote by 
\begin{align*}
&p:X\times \mathbb{P}^1\to X\\
&q:X\times \mathbb{P}^1\to \mathbb{P}^1
\end{align*}
the natural projections.\par
Let $(L,\overline\partial_L)$ be a holomorphic line bundle over $X$, $\phi \in A(L)$, and $h_L$ be a Hermitian metric on $L$.
We assume the triple $((L,\overline\partial_L),\phi,h_L)$ satisifies Manton's exotic $\tau$-vortex equation with $\tau>0$.
We construct an $SU(2)$-equivariant holomorphic bundle over $X\times \mathbb{P}^1$ and an $SU(2)$-invariant pseudo-Hermitian-Einstein metric of signature $(1,1)$ from $(\overline\partial_L,\phi,h)$.
We first construct an $SU(2)$-equivariant holomophic bundle over $X\times\mathbb{P}^1$.
Recall that we denote the trivial line bundle over $X$ as $\mathbb{C}_X$.
The Dolbeault operator $\overline\partial_{X\times\mathbb{P}^1}$ of $X\times\mathbb{P}^1$ defines a holomorphic bundle structure of $\underline{\mathbb{C}}_{X\times \mathbb{P}^1}$.
Let $c\in\mathbb{C}$ be a constant.
We define
\begin{align}
E:=&p^\ast L\otimes q^\ast \mathcal{O}(-2)\oplus \underline{\mathbb{C}}_{X\times \mathbb{P}^1},\nonumber\\
\overline\partial_{E,c}:=&
\begin{pmatrix}
p^\ast \overline\partial_L\otimes \mathrm{Id}+\mathrm{Id}\otimes q^\ast \overline\partial_{\mathcal{O}(-2)}& c\cdot p^\ast \phi\otimes q^\ast \alpha\\
0&\overline\partial_{X\times\mathbb{P}^1}\label{eq 3.3}
\end{pmatrix}.
\end{align}
Note that $c\cdot  p^\ast \phi\otimes q^\ast \alpha\in A^{0,1}(\mathrm{Hom}(\underline{\mathbb{C}}_{X\times \mathbb{P}}, p^\ast L\otimes q^\ast \mathcal{O}(-2)))$.
The $SU(2)$-action on $\mathcal{O}(-2)$ induces an $SU(2)$-action on $E$, where $SU(2)$ acts trivially on the factors pulled back from $X$.
By the definition of the action of $SU(2)$ on $X\times \mathbb{P}^1$, $E$ is also an $SU(2)$-equivariant bundle over $X\times \mathbb{P}^1$.
We also have
\begin{prop}\label{prop 3.3}
$(E,\overline\partial_{E,c})$ is an $SU(2)$-equivariant holomorphic bundle over $X\times\mathbb{P}^1$.
\end{prop}
\begin{proof}
We first prove $(\overline\partial_{E,c})^2=0$, and hence it defines a holomorphic bundle structure on $E$.
By the definition of $\overline\partial_{E,c}$, we have 
\begin{align*}
(\overline\partial_{E,c})^2=
\begin{pmatrix}
(p^\ast \overline\partial_L\otimes \mathrm{Id}+\mathrm{Id}\otimes q^\ast \overline\partial_{\mathcal{O}(-2)})^2 & (\ast)\\
0&(\overline\partial_{X\times\mathbb{P}^1})^2
\end{pmatrix}
\end{align*}
where
\[
(\ast)=(p^\ast \overline\partial_L\otimes \mathrm{Id}+\mathrm{Id}\otimes q^\ast \overline\partial_{\mathcal{O}(-2)})(c\cdot p^\ast \phi\otimes q^\ast \alpha)+(c\cdot p^\ast \phi\otimes q^\ast \alpha)\cdot \bar\partial_{X\times\mathbb{P}^1}.
\]
The diagonal part of $(\overline\partial_{E,c})^2$ vanish since $\overline\partial_L,\overline\partial_{\mathcal{O}(-2)}$, and $\overline\partial_{X\times\mathbb{P}^1}$ are Dolbeaut operators.
Since $\phi$ is a holomorphic section of $(L,\overline\partial_L)$ and $\overline\partial_{\mathcal{O}(-2)}\alpha=0$, $(\ast)=0$ holds and hence $(\overline\partial_E)^2=0$.\par
$(E,\overline\partial_{E,c})$ is an $SU(2)$-equivariant holomorphic bundle since $(\mathcal{O}(-2),\overline\partial_{\mathcal{O}(-2)})$ is an $SU(2)$-equivariant holomorphic bundle over $\mathbb{P}^1$ and $\alpha$ is an $SU(2)$-equivariant section.
\end{proof}
Hence 
\[
0\longrightarrow p^\ast L\otimes q^\ast \mathcal{O}(-2)\longrightarrow E\longrightarrow\underline{\mathbb{C}}_{X\times \mathbb{P}^1} \longrightarrow0.
\]
is a holomorphic bundle extension.\par
Let $h^{(-2)}$ be the $SU(2)$-invariant Hermitian metric of $\mathcal{O}(-2)$ (See Section \ref{sec 3.2.1}), and $h_X$ be a Hermitian metric of $\underline{\mathbb{C}}_X$.
We note that $h_X$ is a smooth positive function on $X$ and therefore $h_Xh_L$ is also a Hermitian metric on $L$.
We define a pseudo-Hermitian metric $h_E$ on $E$ as 
\begin{align}\label{eq 3.4}
h_E:=
\begin{pmatrix}
p^\ast (h_Xh_L)\otimes q^\ast h^{(-2)}&0\\
0& -p^\ast h_X
\end{pmatrix}.
\end{align}
It follows from the construction that $h_E$ is a pseudo-Hermitian metric of signature $(1,1)$.
Since $h^{(-2)}$ is $SU(2)$-invariant, and $SU(2)$ action on $E$ is trivial on the factors pulled back from $X$, the pseudo-Hermitian metric $h_E$ is also $SU(2)$-invariant.
\par
We denote by $\partial_{h_Xh_L},\partial_{h^{(-2)}}$, and $\partial_{h_X}$ the $(1,0)$-part of the Chern connection associated to $(L,\overline\partial_L,h_Xh_L),$ $(\mathcal{O}(-2),\overline\partial_{h^{(-2)}},h^{(-2)})$, and $(\mathbb{C}_X,\overline\partial_X,h_X)$ respectively.
Let $\nabla_{h_E,c}$ be the unique connection that is compatible with $h_E$ and $\nabla_{h_E,c}^{0,1}=\overline\partial_{E,c}$ (See Proposition \ref{prop 3.1}).
We denote by $\partial_{h_E,c}$ the $(1,0)$-part of $\nabla_{h_E}$.
\begin{lem}\label{lem 3.5}
$\partial_{h_E,c}$ has the form
\begin{equation}\label{eq 3.5}
\partial_{h_E,c}
=\begin{pmatrix}
p^\ast\partial_{h_Xh_L}\otimes\mathrm{Id} +\mathrm{Id}\otimes p^\ast\partial_{h^{(-2)}}&0\\
\overline c \cdot p^\ast (h_L\overline \phi)\otimes q^\ast \beta& p^\ast \partial_{h_X}
\end{pmatrix}.
\end{equation}
\end{lem}
\begin{proof}
We first set 
\begin{align*}
\overline\partial_{E,c,\mathrm{diag}}&:=
\begin{pmatrix}
p^\ast \overline\partial_L\otimes \mathrm{Id}+\mathrm{Id}\otimes q^\ast \overline\partial_{\mathcal{O}(-2)}& 0\\
0&p^\ast\overline\partial_X
\end{pmatrix},
\\
A&:=
\begin{pmatrix}
0& c\cdot p^\ast \phi\otimes q^\ast \alpha\\
0&0
\end{pmatrix}.
\end{align*}
We set 
\begin{align*}
\partial_{h_E,c,\mathrm{diag}}
:&=
\begin{pmatrix}
p^\ast\partial_{h_Xh_L}\otimes \mathrm{Id} +\mathrm{Id}\otimes p^\ast\partial_{h^{(-2)}}&0\\
0& p^\ast \partial_{h_X},
\end{pmatrix}\\
B:&=-h_E^{-1} \overline A^t h_E\\
&=
-\begin{pmatrix}
p^\ast (h_Xh_L)^{-1}\otimes q^\ast h^{(-2)}&0\\
0& -p^\ast h_X^{-1}
\end{pmatrix}
\cdot 
\begin{pmatrix}
0& 0\\
\overline{c\cdot p^\ast \phi\otimes q^\ast \alpha}&0
\end{pmatrix}
\cdot 
\begin{pmatrix}
p^\ast (h_Xh_L)\otimes q^\ast h^{(-2)}&0\\
0& -p^\ast h_X
\end{pmatrix}\\
&=
-\begin{pmatrix}
0&0\\
-\overline c \cdot p^\ast (h_L\overline\phi)\otimes q^\ast \beta&0
\end{pmatrix}\\
&=
\begin{pmatrix}
0&0\\
\overline c \cdot p^\ast (h_L\overline\phi)\otimes q^\ast \beta&0
\end{pmatrix}
\end{align*}
The last equation follows from the definition of $h^{(-2)}$ and $\alpha$ (See Section \ref{sec 3.2.1}).
\par
By the definition of the Chern connection
\[
\partial h_E(\cdot ,\cdot)=h_E(\partial_{h_E,c,\mathrm{diag}}\cdot,\cdot)+h(\cdot, \overline\partial_{E,c,\mathrm{diag}})
\]
holds, and by the definition of $B$
\[
h(A\cdot,\cdot)+h(\cdot, B\cdot)=0.
\]
Since 
\[
\overline\partial_{E,c}=\overline\partial_{E,c,\mathrm{diag}}+A,
\]
the connection
\[
\nabla_{h_E,c}:=\partial_{h_E,c,\mathrm{diag}}+B+\overline\partial_E
\]
is $h_E$.
Moreover
\[
\nabla^{0,1}_{h_E,c}=\overline\partial_{E,c}
\]
follows from the construction.
Therefore
\begin{align*}
\partial_{h_E,c}=\nabla^{1,0}_{h_E,c}&=\partial_{h_E,c,\mathrm{diag}}+B\\
&=\begin{pmatrix}
p^\ast\partial_{h_Xh_L}\otimes \mathrm{Id} +\mathrm{Id}\otimes p^\ast\partial_{h^{(-2)}}&0\\
\overline c \cdot p^\ast (h_L\overline \phi)\otimes q^\ast \beta& p^\ast \partial_{h_X}
\end{pmatrix}.
\end{align*}
\end{proof}
We show that for a suitable $c\in \mathbb{C}$ and a suitable K\"ahler form on $X\times \mathbb{P}^1$, there exists a Hermitian metric $h_X$ of $\underline{\mathbb{C}}_X$ such that $h_E$ is a pseudo-Hermitian-Einstein metric of $(E,\overline\partial_{E,c})$ (Definition \ref{def 3.2}).\par
Let $\omega_{\mathrm{FS}}$ be the Fubini-Study K\"ahler form on $\mathbb{P}^1$.
We define a K\"ahler form of $X\times \mathbb{P}^1$ as 
\be\label{eq 3.6}
\Omega_\tau:=\bigg(\frac{\tau}{8\pi}p^\ast\omega_X\bigg)\oplus q^\ast \omega_{\mathrm{FS}}.
\ee
For later use, we calculate the degree of $E$ with respect to the K\"ahler form $\Omega_\tau$.
\begin{lem}\label{lem 3.6}
Let $\mathrm{deg}_{\Omega_\tau} E$ be the degree of $E$ with respect to $\Omega_\tau$.
Then
\[
\mathrm{deg}_{\Omega_\tau} E=\mathrm{deg}L-\frac{\tau}{4\pi}.
\]
\end{lem}
\begin{proof}
By the definition of $E$, we have 
\begin{align*}
\mathrm{deg}_\tau E&=\mathrm{deg}_\tau (p^\ast L)+\mathrm{deg}_\tau (q^\ast \mathcal{O}(-2))+\mathrm{deg}_\tau (p^\ast\underline{\mathbb{C}}_X)\\
&=\mathrm{deg}_\tau (p^\ast L)+\mathrm{deg}_\tau (q^\ast \mathcal{O}(-2)).
\end{align*}
Let $c_1(L),c_1(\mathcal{O}(-2))$ be the first Chern classes of $L,\mathcal{O}(-2)$.
Then by \cite{Wells}, we have 
\begin{align*}
c_1(p^\ast L)&=p^\ast c_1(L),\\
c_1(q^\ast \mathcal{O}(-2))&=q^\ast c_1(\mathcal{O}(-2)).
\end{align*}
Therefore 
\begin{align*}
\mathrm{deg}_\tau (p^\ast L)&=
\int_{X\times \mathbb{P}^1}c_1(p^\ast L)\wedge \Omega_\tau\\
&=\int_{X\times \mathbb{P}^1}p^\ast c_1( L)\wedge q^\ast\omega_{FS}\\
&=\int_X c_1(L)\int_{\mathbb{P}^1}\omega_{FS}\\
&=\mathrm{deg}L\\
\mathrm{deg}_\tau(q^\ast\mathcal{O}(-2))&=
\int_{X\times \mathbb{P}^1}c_1(q^\ast \mathcal{O}(-2))\wedge\Omega_\tau\\
&=\int_{X\times \mathbb{P}^1}q^\ast c_1(\mathcal{O}(-2))\wedge\frac{\tau}{8\pi}q^\ast\omega_X\\
&=\frac{\tau}{8\pi}\int_X\omega_X\int_{\mathbb{P}^1} q^\ast c_1(\mathcal{O}(-2))\\
&=-\frac{\tau}{4\pi}
\end{align*}
The last equation holds since we normalized $\omega_X$.
Hence, we obtain the desired equality.
\end{proof}
We now set 
\[
c=\sqrt{\frac{1}{\tau}}
\]
and consider $(E,\overline\partial_{E,\sqrt{\frac{1}{\tau}}})$.
For simplicity, we write
\[
\overline\partial_E:=\overline\partial_{E,\sqrt{1/\tau}}.
\]
We also write 
\begin{align*}
\nabla_{h_E}&:=\nabla_{h_E,\sqrt{1/\tau}},\\
\partial_{h_E}&:=\partial_{h_E, \sqrt{1/\tau}}.
\end{align*}
We denote the curvature of $\nabla_{h_E}$ as
\[
F_{h_E}:=(\nabla_{h_E})^2.
\]
We also denote by $F_{h_Xh_L}, F_{h^{(-2)}}$, and $F_{h_X}$ the curvature of the Chern connections of $(L,\overline\partial_L,h_Xh_L),$ $(\mathcal{O}(-2),\overline\partial_{\mathcal{O}(-2)},h^{(-2)})$, and $(\underline{\mathbb{C}}_X,\overline\partial_X, h_X)$ respectively.
\begin{prop}\label{prop 3.4}
There exists a Hermitian metric $h_X$ of $\underline{\mathbb{C}}_X$ such that the pseudo-Hermitian metric $h_E$ defined by \eqref{eq 3.4} is a pseudo-Hermitian-Einstein metric with respect to the K\"ahler form $\Omega_\tau$;
that is, there exists a constant $\lambda\in\mathbb{C}$ such that
\[
\sqrt{-1}\Lambda_{\Omega_\tau}F_{h_E}=\lambda \mathrm{Id}_E.
\]
Equivalently, the connection $\nabla_{h_E}$ is a pseudo-Hermitian-Yang-Mills connection.
\end{prop}
\begin{proof}
We first show that the constant $\lambda$ is uniquely determined when$\nabla_{h_E}$ is pseudo-Hermitian-Yang-Mills.
By Section \ref{sec 3.1.4} we have
\begin{align*}
\mathrm{deg}_{\Omega_\tau} E&=\frac{\sqrt{-1}}{2\pi}\int_{X\times\mathbb{P}^1}\Lambda_{\Omega_\tau}\mathrm{Tr}F_{h_E}\frac{\Omega_\tau^2}{2}\\
&=\frac{1}{2\pi}\int_{X\times \mathbb{P}^1} 2\lambda\cdot\frac{\tau}{8\pi} p^\ast\omega_{X}\wedge q^\ast \omega_{FS}\\
&=\frac{1}{\pi}\frac{\lambda\cdot \tau}{8\pi}.
\end{align*}
Then by Lemma \ref{lem 3.6}
\[
\mathrm{deg}_{\Omega_\tau} E=\mathrm{deg}L-\frac{\tau}{4\pi}.
\]
Therefore
\be\label{eq 3.7}
\frac{\lambda\cdot \tau}{8\pi}=\pi\cdot \bigg(\mathrm{deg}L-\frac{\tau}{4\pi}\bigg)
\ee
\par
Let $h_X$ be a Hermitian metric on $\mathbb{C}_X$ and let $h_E$ be a pseudo-Hermitian metric of signature (1,1) defined as \eqref{eq 3.4}.
Since $(\overline\partial_E)^2=0$, $(\partial_{h_E})^2=0$ and therefore
\[
F_{h_E}=\overline\partial_E\partial_{h_E}+\partial_{h_E}\overline\partial_E.
\]
We set
\[
F_{h_E}=
\begin{pmatrix}
F_{11}&F_{12}\\
F_{21}&F_{22}
\end{pmatrix}.
\]
Then by \eqref{eq 3.3} and \eqref{eq 3.5} we have 
\[
\begin{split}
F_{11}&=p^\ast F_{h_Xh_L}+q^\ast F_{h^{(-2)}}+\frac{1}{\tau}p^\ast|\phi|^2_{h_L}\otimes q^\ast(\alpha\wedge \beta)\\
&=p^\ast F_{h_L}+p^\ast F_{h_X}+q^\ast F_{h^{(-2)}}-\frac{2\pi}{\sqrt{-1}}\cdot \frac{1}{\tau}p^\ast|\phi|^2_{h_L}q^\ast\omega_{FS}.
\end{split}
\]
The last equation follows from the definition of $\alpha$ and $\beta$ (See Section \ref{sec 3.1.1}).
Therefore we have 
\be\label{eq 3.8}
\begin{split}
\sqrt{-1}\Lambda_{\Omega_\tau}F_{11}&=\sqrt{-1}\frac{8\pi}{\tau}p^\ast \Lambda_{\omega_X}(F_{h_L}+F_{h_X})-4\pi-\frac{2\pi}{\tau}p^\ast|\phi|^2_{h_L}.
\end{split}
\ee
We also have 
\[
\begin{split}
F_{22}&=p^\ast F_{h_X}-\frac{1}{\tau}p^\ast|\phi|^2_{h_L}\otimes q^\ast(\beta\wedge \alpha)\\
&=p^\ast F_{h_X}+\frac{2\pi}{\sqrt{-1}}\cdot \frac{1}{\tau}|\phi|^2_{h_L}q^\ast \omega_{\mathrm{FS}}.
\end{split}
\]
Therefore
\be\label{eq 3.9}
\sqrt{-1}\Lambda_{\Omega_\tau}F_{22}=\sqrt{-1}\frac{8\pi}{\tau}p^\ast \Lambda_{\omega_X}F_{h_X}+\frac{2\pi}{\tau}p^\ast|\phi|^2_{h_L}.
\ee

Since $F_{12},F_{21}$ has mixed contribution from $X$ and $\mathbb{P}^1$ we have

\[
\sqrt{-1}\Lambda_{\Omega_\tau}F_{h_E}
=
\begin{pmatrix}
\sqrt{-1}\Lambda_{\Omega_\tau}F_{11}&0\\
0&\sqrt{-1}\Lambda_{\Omega_\tau}F_{22} \\
\end{pmatrix}.
\]
Hence by \eqref{eq 3.8} and \eqref{eq 3.9} the equation
\[
\sqrt{-1}\Lambda_{\Omega_\tau}F_{h_E}=\lambda \mathrm{Id}_E
\]
is equivalent to 
\begin{align}
\sqrt{-1}\frac{8\pi}{\tau}\Lambda_{\omega_X}(F_{h_L}+F_{h_X})-4\pi-\frac{2\pi}{\tau}|\phi|^2_{h_L}&=\lambda,\label{eq 3.10}
\\
\sqrt{-1}\frac{8\pi}{\tau} \Lambda_{\omega_X}F_{h_X}+\frac{2\pi}{\tau}|\phi|^2_{h_L}&=\lambda.\label{eq 3.11}
\end{align}
The equations \eqref{eq 3.10} and \eqref{eq 3.11} are equivalent to 
\begin{align}
\sqrt{-1}\frac{8\pi}{\tau}\Lambda_{\omega_X}F_{h_L}-4\pi-\frac{4\pi}{\tau}|\phi|^2_{h_L}&=0,\label{eq 3.12}
\\
\sqrt{-1}\frac{8\pi}{\tau} \Lambda_{\omega_X}F_{h_L}+\sqrt{-1}\frac{16\pi}{\tau} \Lambda_{\omega_X}F_{h_X}-4\pi&=2\lambda,\label{eq 3.13}
\end{align}
since \eqref{eq 3.12} is obtained by \eqref{eq 3.10}-\eqref{eq 3.11} and  \eqref{eq 3.13} is obtained by \eqref{eq 3.10}+\eqref{eq 3.11}.
\par
\eqref{eq 3.12} holds since $(\overline\partial_L,\phi,h_L)\in \mathrm{Tri}_\tau$.
We show that there exists a Hermitian metric $h_X$ such that  \eqref{eq 3.13}  holds.
Recall that 
\[
\sqrt{-1}\Lambda_{\omega_X}F_{h_X}=\frac{1}{2}\triangle\mathrm{log}h_X.
\]
Here $\triangle$ is the Hodge Laplacian of $X$.
Since $h_X$ is a smooth positive function on $X$, there exists a function $f$ such that $h_X=e^{f}$.
Hence, the existence of the Hermitian metric $h_X$ that satisfies  \eqref{eq 3.13}  is equivalent to the existence of a function $f$ that satisfies 
\be\label{eq 3.14}
\triangle f=\frac{\lambda\cdot \tau}{8\pi}+\frac{\tau}{4}-\frac{1}{2}\cdot\sqrt{-1}\Lambda_{\omega_X}F_{h_L}.
\ee 
By the classical Hodge theory \cite{Wa}, the existence of the solution of \eqref{eq 3.14} is equivalent to
\[
\int_{X}\bigg(\frac{\lambda\cdot \tau}{8\pi}+\frac{\tau}{4}-\frac{1}{2}\cdot\sqrt{-1}\Lambda_{\omega_X}F_{h_L}\bigg)=0.
\]
This equation holds since 
\begin{align*}
\int_{X}\bigg(\frac{\lambda\cdot \tau}{8\pi}+\frac{\tau}{4}-\frac{1}{2}\cdot\sqrt{-1}\Lambda_{\omega_X}F_{h_L}\bigg)&=\frac{\lambda\cdot \tau}{8\pi}+\frac{\tau}{4}-\frac{1}{2}\int_X\sqrt{-1}\Lambda_{\omega_X}F_{h_L}\\
&=\pi\cdot\mathrm{deg}L-\frac{\tau}{4}+\frac{\tau}{4}-\pi\cdot\mathrm{deg}L\\
&=0.
\end{align*}
The second equation follows from \eqref{eq 3.7}.
Let $f$ be a solution of \eqref{eq 3.14}.
Then $h_X:=e^f$ is the desired Hermitian metric.
\end{proof}
As a summary of the above discussion, we have 
\begin{thm}\label{thm 3.1}
Let $X$ be a compact connected Riemann surface and let  $\omega_X$ be a K\"ahler form of $X$ satisfying $\int_X\omega_X=1$.
Let $(L,\overline\partial_L)$ be a holomorphic line bundle over $X$, $\phi$ a holomorphic section  of $(L,\overline\partial_L)$, and $h_L$ a Hermitian metric on $L$.
Assume that the triple $((L,\overline\partial_L),\phi,h_L)$ satisfies Manton's exotic $\tau$-vortex equation with $\tau>0$.
Then 
\begin{align*}
E:=&p^\ast L\otimes q^\ast \mathcal{O}(-2)\oplus \underline{\mathbb{C}}_{X\times\mathbb{P}^1},\\
\overline\partial_{E}:=&
\begin{pmatrix}
p^\ast \overline\partial_L\otimes \mathrm{Id}+\mathrm{Id}\otimes q^\ast \overline\partial_{\mathcal{O}(-2)}& \sqrt{\frac{1}{\tau}}\cdot p^\ast \phi\otimes q^\ast \alpha\\
0&\bar\partial_{X\times\mathbb{P}^1}
\end{pmatrix}.
\end{align*}
is an $SU(2)$-equivariant holomorphic bundle over $X\times \mathbb{P}^1$.
We fix a K\"ahler form on $X\times\mathbb{P}^1$ given by
\[
\Omega_\tau=\bigg(\frac{\tau}{8\pi}p^\ast\omega_X\bigg)\oplus q^\ast \omega_{\mathrm{FS}}.
\]
Then there exists a Hermitian metric $h_X$ on $\underline{\mathbb{C}}_X$ such that the pseudo-Hermitian metric of signature (1,1)
\begin{align*}
h_E=
\begin{pmatrix}
p^\ast (h_Xh_L)\otimes q^\ast h^{(-2)}&0\\
0& -p^\ast h_X
\end{pmatrix}
\end{align*}
is a pseudo-Hermitian-Einstein metric of $(E,\overline\partial_E)$ with respect to the K\"ahler form $\Omega_\tau$.
Equivalently, the unique connection $\nabla_{h_E}$ which is compatible with $h$ and $\nabla^{0,1}_{h_E}=\bar\partial_E$ is a pseudo-Hermitian-Yang-Mills connection.
\end{thm}

We next construct a solution of Manton's exotic $\tau$-vortex ($\tau>0$) equation from an $SU(2)$-equivariant holomorphic bundle and a $SU(2)$-invariant pseudo-Hermitian-Einstein metric of signature (1,1).
We call this procedure dimensional reduction.\par
Let $(L,\bar\partial_L)$ be a holomorphic line bundle over $X$.
Let $E$ be a holomorphic bundle on $X\times\mathbb{P}^1$ that is an extension of $\underline{\mathbb{C}}_{X\times \mathbb{P}^1}$ by $p^\ast L \otimes q^\ast\mathcal{O}(-2)$.
Thus, there exists a short exact sequence of holomorphic vector bundles
\[
0\longrightarrow p^\ast L \otimes q^\ast\mathcal{O}(-2)\longrightarrow E\longrightarrow \underline{\mathbb{C}}_{X\times \mathbb{P}^1}\longrightarrow 0.
\]
Since this short exact sequence is also a short exact sequence of smooth vector bundles, the sequence splits in the smooth category, i.e.
\[
E\simeq p^\ast L \otimes q^\ast\mathcal{O}(-2)\oplus\underline{\mathbb{C}}_{X\times \mathbb{P}^1}
\]
as smooth bundles.
From now on, we identify $E$ and $p^\ast L \otimes q^\ast\mathcal{O}(-2)\oplus \underline{\mathbb{C}}_{X\times \mathbb{P}^1}$.
Let $\overline\partial_E$ be the Dolbeault operator of $E$.
Since $E$ given by an extension, $\overline\partial_E$ can be written in the form of 
\[
\overline\partial_E=
\begin{pmatrix}
p^\ast \overline\partial_L\otimes \mathrm{Id}+\mathrm{Id}\otimes q^\ast \overline\partial_{\mathcal{O}(-2)}& \psi\\
0&\overline\partial_{X\times\mathbb{P}^1}
\end{pmatrix}
\]
where $\psi\in A^{0,1}(\mathrm{Hom}(\underline{\mathbb{C}}_{X\times \mathbb{P}^1}, p^\ast L \otimes q^\ast\mathcal{O}(-2))$.\par
The bundle $E$ is $SU(2)$-equivariant with respect to the $SU(2)$-action induced by that on $\mathcal{O}(-2)$.
Moreover, if $(E,\bar\partial_E)$ is an $SU(2)$-equivariant holomorphic bundle, then $\psi$ has a specific form as follows.
\begin{prop}\label{prop 3.5}
If $(E,\bar\partial_E)$ is an $SU(2)$-equivariant holomorphic bundle, then
\begin{itemize}
\item[(1)] there exists a $\phi\in A(L)$ such that 
\[
\psi=p^\ast\phi\otimes q^\ast\alpha.
\]
\item[(2)] $\phi$ is a holomorphic section of $(L,\bar\partial_L)$.
\end{itemize}
\end{prop}
\begin{proof}
$(1)$ was essentially proved in \cite[3.5. Proposition. (b)]{GP2}.
We give a proof for convenience.
Since we assumed that $(E,\overline\partial_E)$ is an $SU(2)$-equivariant holomorphic bundle, we have 
\begin{align*}
 \psi\in A^{0,1}(\mathrm{Hom}(\underline{\mathbb{C}}_{X\times \mathbb{P}^1}, p^\ast L \otimes q^\ast\mathcal{O}(-2))^{SU(2)}.
 \end{align*}
 Then by
 \[
 T^{0,1\ast}(X\times \mathbb{P}^1)=p^\ast T^{0,1\ast}X\oplus q^\ast T^{0,1\ast}\mathbb{P}^1
 \]
 we have 
 \begin{align*}
 A^{0,1}(\mathrm{Hom}(\underline{\mathbb{C}}_{X\times \mathbb{P}^1}, p^\ast L \otimes q^\ast\mathcal{O}(-2))^{SU(2)}=&A((p^\ast T^{0,1\ast}X\oplus q^\ast T^{0,1\ast}\mathbb{P}^1)\otimes (p^\ast L \otimes q^\ast\mathcal{O}(-2)))^{SU(2)}\\
 =&A(p^\ast T^{0,1\ast}X\otimes p^\ast L \otimes q^\ast\mathcal{O}(-2))^{SU(2)}\\
 &\oplus A(q^\ast T^{0,1\ast}\mathbb{P}^1\otimes p^\ast L \otimes q^\ast\mathcal{O}(-2))^{SU(2)}.
 \end{align*}
Let 
\[
v\in A(p^\ast T^{0,1\ast}X\otimes p^\ast L \otimes q^\ast\mathcal{O}(-2))^{SU(2)}.
\]
For arbitrarily $x\in X$, $v|_{\{x\}\times\mathbb{P}^1}$ is an $SU(2)$-equivariant section of $\mathcal{O}(-2)$.
However, every such section is the zero section and therefore $v=0$.
Hence 
\[
A^{0,1}(\mathrm{Hom}(\underline{\mathbb{C}}_{X\times \mathbb{P}^1}, p^\ast L \otimes q^\ast\mathcal{O}(-2))^{SU(2)}=A(q^\ast T^{0,1\ast}\mathbb{P}^1\otimes p^\ast L \otimes q^\ast\mathcal{O}(-2))^{SU(2)}
\]
and therefore
\[
\psi\in A(q^\ast T^{0,1\ast}\mathbb{P}^1\otimes p^\ast L \otimes q^\ast\mathcal{O}(-2))^{SU(2)}.
\]
For arbitrarily $x\in X$
\[
\psi|_{\{x\}\times\mathbb{P}^1}\in A(T^{0,1\ast}\mathbb{P}^1\otimes \mathcal{O}(-2))^{SU(2)}=A^{0,1}(\mathcal{O}(-2))^{SU(2)}.
\]
Then by Proposition \ref{prop 3.2}, there exist a constant $\phi_x\in\mathbb{C}$ such that 
\[
\psi|_{\{x\}\times\mathbb{P}^1}=\phi_x\otimes q^\ast \alpha.
\]
Since $\psi$ is a smooth section, $\phi_x$ varys smooth with respect to $x$ and defines a section $\phi\in A(L)$.
Hence
\[
\psi=p^\ast\phi\otimes q^\ast\alpha.
\]
We now show $(2).$
Since $(\bar\partial_E)^2=0$, the off diagonal part of it is 
\[
(p^\ast \overline\partial_L\otimes \mathrm{Id}+\mathrm{Id}\otimes q^\ast \overline\partial_{\mathcal{O}(-2)})(p^\ast \phi\otimes q^\ast \alpha)+( p^\ast \phi\otimes q^\ast \alpha) \bar\partial_{X\times\mathbb{P}^1}=0.
\]
This equation is equivalent to 
\[
p^\ast(\overline\partial_L\phi)\otimes q^\ast\alpha=0.
\]
Since $\alpha$ is everywhere non-zero, we have
\[
\overline\partial_L\phi=0.
\]
Therefore $\phi$ is a holomorphic section of $(L,\overline\partial_L)$.
\end{proof}
Since multiplying a holomorphic section by a nonzero constant preserves holomorphicity, $\sqrt{\tau}\phi$ is also holomorphic.
Thus, after replacing $\sqrt{\tau}\phi$ by $\phi$, we may write
\[
\psi
=
\sqrt{\frac{1}{\tau}}\cdot p^\ast\phi\otimes q^\ast\alpha.
\]
We keep the normalization of $\alpha$ fixed.\par

We next study the structure of $SU(2)$-invariant pseudo-Hermitian metric on $E$.
\begin{prop}\label{prop 3.6}
Let $h_E$ be an $SU(2)$-invariant pseudo-Hermitian metric on $E$.
Then there exists a Hermitian metrics $h_L$ of $L$, $h_X$ of $\mathbb{C}_X$ such that
\[
h_E
=
\begin{pmatrix}
C_0\cdot p^\ast h_L\otimes q^\ast h^{(-2)}&0\\
0& C_1\cdot p^\ast h_X
\end{pmatrix},\,\,C_0,C_1\in\{1,-1\}.
\]
\end{prop}
\begin{proof}
We first prove that $p^\ast L\otimes q^\ast \mathcal{O}(-2)$ and $\underline{\mathbb{C}}_{X\times \mathbb{P}^1}$ are orthogonal with respect to $h_E$.
We use the same argument as in \cite{GP2}.
From now on, we view $\mathbb{P}^1$ as a homogeneous manifold $SU(2)/U(1)$ and $\mathcal{O}(-2)$ as the weight $-2$ $SU(2)$-equivariant line bundle.\par
Let $x\in X$ and $g\in SU(2)$.
We denote by $(p^\ast L\otimes q^\ast \mathcal{O}(-2))_{(x,gU(1))}$ and $\underline{\mathbb{C}}_{X\times \mathbb{P}^1,(x,gU(1))}$ the fibers of $p^\ast L\otimes q^\ast \mathcal{O}(-2)$ and $\underline{\mathbb{C}}_{X\times \mathbb{P}^1}$ over $(x,gU(1))\in X\times\mathbb{P}^1$, respectively. 
We also denote by $e$ the identity element of $SU(2)$.\par
Let $p^\ast v\otimes q^\ast u_{(x,eU(1))}$ and $w_{(x,eU(1))}$ be nonzero vectors, hence bases of $(p^\ast L\otimes q^\ast \mathcal{O}(-2))_{(x,eU(1))}$ and $\underline{\mathbb{C}}_{X\times \mathbb{P}^1,(x,eU(1))}$ respectively.
Since $h_E$ is $SU(2)$-invariant and $e^{i\theta}\in U(1)$ stabilizes $eU(1)$, we have 
\begin{align*}
h_{E,(x,eU(1))}(p^\ast v\otimes q^\ast u, w)&=h_{E,(x,eU(1))}(e^{i\theta}\cdot p^\ast v\otimes q^\ast u, e^{i\theta}\cdot w)\\
&=h_{E,(x,eU(1))}(e^{-2i\theta}p^\ast v\otimes q^\ast u_{(x,eU(1))},w_{(x,eU(1))})\\
&=e^{-2i\theta}h_{E,(x,eU(1))}(p^\ast v\otimes q^\ast u_{(x,eU(1))}, w_{(x,eU(1))}).
\end{align*}
Since this holds for all \(\theta\in\mathbb{R}\), we obtain
\[
h_{E,(x,eU(1))}(p^\ast v\otimes q^\ast u_{(x,eU(1))}, w_{(x,eU(1))})=0.
\]
Hence $(p^\ast L\otimes q^\ast \mathcal{O}(-2))_{(x,eU(1))}$ and $\underline{\mathbb{C}}_{X\times \mathbb{P}^1,(x,eU(1))}$ are orthogonal.
Let $p^\ast v\otimes q^\ast u_{(x,gU(1))}$ and $w_{(x,gU(1))}$ be bases of $(p^\ast L\otimes q^\ast \mathcal{O}(-2))_{(x,gU(1))}$ and $\underline{\mathbb{C}}_{X\times \mathbb{P}^1,(x,gU(1))}$ respectively. 
Since $g^{-1}\cdot p^\ast v\otimes q^\ast u_{(x,gU(1))}\in p^\ast L\otimes q^\ast \mathcal{O}(-2)_{(x,eU(1))}, g^{-1}\cdot w_{(x,gU(1))}\in \underline{\mathbb{C}}_{X\times \mathbb{P}^1,(x,eU(1))}$,
we have 
\begin{align*}
h_{E,(x,gU(1))}(p^\ast v\otimes q^\ast u_{(x,gU(1))}, w_{(x,gU(1))})&=h_{E,(x,gU(1))}(g\cdot g^{-1}\cdot p^\ast v\otimes q^\ast u_{(x,gU(1))},g\cdot g^{-1}\cdot  w_{(x,gU(1))})\\
&=h_{E,(x,eU(1))}(g^{-1}\cdot p^\ast v\otimes q^\ast u_{(x,gU(1))}, g^{-1}\cdot  w_{(x,gU(1))})\\
&=0.
\end{align*}
Hence $p^\ast L\otimes q^\ast \mathcal{O}(-2)$ and $\underline{\mathbb{C}}_{X\times \mathbb{P}^1}$ are orthogonal with respect to $h_E$.
Therefore, $h_E$ can be written in the form
\[
h_E
=
\begin{pmatrix}
h_1&0\\
0& h_2
\end{pmatrix}
\]
where $h_1$ and $h_2$ are Hermitian forms on $p^\ast L\otimes q^\ast \mathcal{O}(-2)$ and $\underline{\mathbb{C}}_{X\times \mathbb{P}^1}$.
Since $h_E$ is non-degenerate and $SU(2)$-invariant, $h_1$ and $h_2$ are also non-degenerate and $SU(2)$-invariant.
We first show that there exists a Hermitian metric $h_X$ on $\underline{\mathbb{C}}_X$ such that 
\[
h_2=C\cdot p^\ast h_X, \,\, C\in\{1,-1\}. 
\]
$C$ depends on whether $h_2$ is positive or negative.
Since $h_2$ is $SU(2)$-invariant and the $SU(2)$-action on $\underline{\mathbb{C}}_{X\times\mathbb{P}^1}$ is induced by the action on $\mathbb{P}^1$, for every $x\in X$ there exists a constant $a_x\in \mathbb{R}$ such that  
\[
h_2|_{\{x\}\times\mathbb{P}^1}=a_x.
\]
Since $h_2$ is non-degenerate, $a_x\in \mathbb{R}^\times$ and $a_x$ varies smoothly with respect to $x$.
Moreover, if $h_2$ is positive, then $a_x>0$ and if $h_2$ is negative, then $a_x<0$ for every $x\in X$.
If $h_2$ is positive, we define a smooth positive function on $X$ as
\[
h_X(x):=a_x 
\]
and if $h_2$ is negative, we define a smooth positive function on $X$ as
\[
h_X(x):=-a_x.
\]
We note that $h_X$ defines a Hermitian metric on $\underline{\mathbb{C}}_X$ since it is a positive function.
By construction, we obtain 
\[
h_2=C\cdot p^\ast h_X, \,\, C\in\{1,-1\}.
\]
We finally show that there exists a Hermitian metric $h_L$ of $L$ such that 
\[
h_1=C\cdot p^\ast h_L\otimes q^\ast h^{(-2)},\,\, C\in \{1,-1\}.
\]
Again, $C$ depends on whether $h_1$ is positive or negative. 
Since $h_1$ is $SU(2)$-invariant
\[
h_1|_{\{x\}\times\mathbb{P}^1}
\]
 is an $SU(2)$-invariant pseudo-Hermitian metric on $p^\ast L \otimes q^\ast\mathcal{O}(-2)|_{\{x\}\times\mathbb{P}^1}$.
 Then, since $h^{(-2)}$ is an $SU(2)$-invariant Hermitian metric on $\mathcal{O}(-2)$ by Lemma \ref{lem 3.3}, there exists an $a_x\in\mathbb{R}^\times$
 \[
 h_1|_{\{x\}\times\mathbb{P}^1}=a_x \otimes q^\ast h^{(-2)}.
 \]
 Then, by the same argument as above, $a_x$ is positive for all $x\in X$ if $h_1$ is positive and negative for all $x\in X$ if $h_1$ is negative.
 If $h_1$ is positive, then the family $(a_x)_{x\in X}$ defines a Hermitian metric $h_L$ on $L$; if $h_1$ is negative, then the family $(-a_x)_{x\in X}$ defines a Hermitian metric $h_L$ on $L$.
 By the construction of $h_L$, we have 
 \[
h_1=C\cdot p^\ast h_L\otimes q^\ast h^{(-2)},\,\, C\in \{1,-1\}.
\]
This completes the proof.
\end{proof}
The following is a consequence of Proposition \ref{prop 3.6} and a rescaling of the Hermitian metric on $L$.
\begin{lem}\label{lem 3.7}
Let $h_E$ be an $SU(2)$-invariant pseudo-Hermitian metric of signature $(1,1)$ on $E$.
Then there exist Hermitian metrics $h_L$ on $L$ and $h_X$ on $\underline{\mathbb{C}}_X$ such that 
\[
h_E
=
\begin{pmatrix}
C_0\cdot p^\ast (h_Xh_L)\otimes q^\ast h^{(-2)}&0\\
0& C_1\cdot p^\ast h_X
\end{pmatrix}
\]
where $(C_0,C_1)=(1,-1)$ or $(-1,1)$.
\end{lem} 
We now establish the dimensional reduction.
Throughout this argument, we work with the K\"ahler form $\Omega_\tau$ constructed in \eqref{eq 3.6}.
\par
Suppose that $\overline\partial_E$ defines an $SU(2)$-equivariant holomorphic bundle structure on $E$.
Then by Proposition \ref{prop 3.5}, there exists a holomorphic section $\phi$ of $(L,\bar\partial_L)$ such that 
\[
\overline\partial_E=
\begin{pmatrix}
p^\ast \overline\partial_L\otimes \mathrm{Id}+\mathrm{Id}\otimes q^\ast \overline\partial_{\mathcal{O}(-2)}& \sqrt{\frac{1}{\tau}}\cdot p^\ast\phi\otimes q^\ast\alpha\\
0&\overline\partial_{X\times\mathbb{P}^1}
\end{pmatrix}.
\]
Let $h_E$ be as in Lemma \ref{lem 3.7}, and let $\nabla_{h_E}$ be the connection compatible with $h_E$ and satisfying $\nabla^{0,1}_{h_E}=\overline\partial_E$.
Then, by Lemma \ref{lem 3.5}, $\partial_{h_E}=\nabla^{1,0}_{h_E}$ has the form 
\[
\partial_{h_E}=
\begin{pmatrix}
p^\ast\partial_{h_Xh_L}\otimes\mathrm{Id}+\mathrm{Id}\otimes p^\ast\partial_{h^{(-2)}}&0\\
\sqrt{\frac{1}{\tau}} \cdot p^\ast (h_L\overline \phi)\otimes q^\ast \beta& p^\ast \partial_{h_X}
\end{pmatrix}.
\]
We denote the curvature of $\nabla_{h_E}$ as $F_{h_E}$.
\begin{prop}\label{prop 3.7}
If $h_E$ is a pseudo-Hermitian-Einstein metric of $(E,\overline\partial_E)$ with respect to the K\"ahler form $\Omega_\tau$, that is,
\[
\sqrt{-1}\Lambda_{\Omega_\tau}F_{h_E}=\lambda \mathrm{Id}_E
\]
for some constant $\lambda\in\mathbb{C}$. 
Then, the triple $((L,\overline\partial_L),\phi,h_L)$ satisfies Manton's exotic $\tau$-vortex equation.
\end{prop}
\begin{proof}
We set 
\[
F_{h_E}=
\begin{pmatrix}
F_{11}&F_{12}\\
F_{21}&F_{22}.
\end{pmatrix}
\]
The equaiton
\[
\sqrt{-1}\Lambda_{\Omega_\tau}F_{h_E}=\lambda Id_E
\]
implies 
\begin{align*}
\sqrt{-1}\Lambda_{\Omega_\tau}F_{11}&=\lambda,\\
\sqrt{-1}\Lambda_{\Omega_\tau}F_{22}&=\lambda.
\end{align*}
Then, by \eqref{eq 3.10}, \eqref{eq 3.11}, \eqref{eq 3.12}, and \eqref{eq 3.13}, 
we have 
\begin{align*}
\sqrt{-1}\frac{8\pi}{\tau}\Lambda_{\omega_X}(F_{h_L}+F_{h_X})-4\pi-\frac{2\pi}{\tau}|\phi|^2_{h_L}&=\lambda,
\\
\sqrt{-1}\frac{8\pi}{\tau} \Lambda_{\omega_X}F_{h_X}+\frac{2\pi}{\tau}|\phi|^2_{h_L}&=\lambda.
\end{align*}
Subtracting the second equation from the first, we obtain
\[
\sqrt{-1}\frac{8\pi}{\tau}\Lambda_{\omega_X}F_{h_L}-4\pi-\frac{4\pi}{\tau}|\phi|^2_{h_L}=0.
\]
Hence the triple $((L,\overline\partial_L),\phi,h_L)$ is a solution of Manton's exotic $\tau$-vortex equation.
\end{proof}
Combining the results obtained above, we obtain the following dimensional reduction theorem.
\begin{thm}\label{thm 3.2}
Let $X$ be a compact connected Riemann surface and let  $\omega_X$ be a K\"ahler form of $X$ satisfying $\int_X\omega_X=1$.
Let $(L,\overline\partial_L)$ be a holomorphic line bundle over $X$.
Let $E$ be a holomorphic bundle over $X\times \mathbb{P}^1$ which is an extension of $\underline{\mathbb{C}}_{X\times \mathbb{P}^1}$ by $p^\ast L\otimes q^\ast\mathcal{O}(-2)$.
Let $\overline\partial_E$ be the Dolbeault operator of $E$.
On $X\times\mathbb{P}^1$, we consider the K\"ahler form $\Omega_\tau$ constructed in \eqref{eq 3.6}.
\par
If $(E,\bar\partial_E)$ is an $SU(2)$-equivariant holomorphic bundle, then there exists a holomorphic section $\phi$ of $(L,\bar\partial_L)$ such that 
\[
\overline\partial_E=
\begin{pmatrix}
p^\ast \overline\partial_L\otimes\mathrm{Id}+\mathrm{Id}\otimes q^\ast \overline\partial_{\mathcal{O}(-2)}& \sqrt{\frac{1}{\tau}}\cdot p^\ast\phi\otimes q^\ast\alpha\\
0&\overline\partial_{X\times\mathbb{P}^1}
\end{pmatrix}.
\]
Let $h_E$ be a pseudo-Hermitian metric of signature $(1,1)$ on $E.$
If $h_E$ is $SU(2)$-invariant, then there exist Hermitian metrics $h_L$ on $L$ and $h_X$ on $\mathbb{C}_X$ such that 
\[
h_E
=
\begin{pmatrix}
C_0\cdot p^\ast (h_Xh_L)\otimes q^\ast h^{(-2)}&0\\
0& C_1\cdot p^\ast h_X
\end{pmatrix}
\]
where $(C_0,C_1)=(1,-1)$ or $(-1,1)$.\par
Let $h_E$ and $(E,\overline\partial_E)$ be as above, and $\nabla_{h_E}$ be the connection compatible with $h_E$ and $\nabla^{0,1}_{h_E}=\overline\partial_E$.
If $h_E$ is a pseudo-Hermitian-Einstein metric of $(E,\overline\partial_E)$ with respect to $\Omega_\tau$, or equivalently, if $\nabla_{h_E}$ is a pseudo-Hermitian-Yang-Mills connection, then the triple
 $((L,\overline\partial_L),\phi,h_L)$ satisfies Manton's exotic $\tau$-vortex equation with $\tau>0$.
\end{thm}
\subsection{$\tau<0$ case}\label{sec 3.3}
\subsubsection{$SU(1,1)$-equivariant line bundle over $\Delta$}\label{sec 3.3.1}
Recall that the group $SU(1,1)$ which is the special unitary group of signature $(1,1)$ is a non-compact Lie group defined as
\begin{equation*}
SU(1,1)=\bigg\{
\begin{pmatrix}
a&b\\
\bar b&\bar a
\end{pmatrix}\in M_2(\mathbb{C})
\,\,
\bigg|
\,\,
 |a|^2-|b|^2=1\bigg\}.
\end{equation*}
Let $\Delta:=\{z\in\mathbb{C}\,|\, |z|<1\}.$
$SU(1,1)$ acts on $\Delta$ as 
\begin{equation*}
\begin{array}{rccc}
& SU(1,1)\times\Delta &\longrightarrow& \Delta                    \\
        & \rotatebox{90}{$\in$}&               & \rotatebox{90}{$\in$} \\
        &\bigg( \begin{pmatrix}
a&b\\
\overline b&\overline{a}
\end{pmatrix}   , z  \bigg)              & \longmapsto   &  \frac{az+b}{\overline b z+\bar a}.
\end{array}
\end{equation*}
It is clear from the definition that the action is holomorphic.
Let $z\in\Delta$.
The action is transitive since 
\[
\frac{1}{\sqrt{1-|z|^2}}
\begin{pmatrix}
1&z\\
\bar z&1
\end{pmatrix}\in SU(1,1)
\]
and 
\[
\frac{1}{\sqrt{1-|z|^2}}
\begin{pmatrix}
1&z\\
\bar z&1
\end{pmatrix}\cdot 0=z.
\]
The stabilizer subgroup of $0\in \Delta$ is
\begin{equation*}
\mathrm{Stab}_0=\bigg\{
\begin{pmatrix}
e^{i\theta}&0\\
0&e^{-i\theta}
\end{pmatrix}\
\,\,
\bigg|
\,\,
\theta\in\mathbb{R}\bigg\}\simeq U(1).
\end{equation*}
Therefore we identify $\Delta$ with the homogeneous space $SU(1,1)/U(1)$.
We denote by $\underline{\mathbb{C}}_\Delta$ the trivial line bundle over $\Delta$.
The standard Dolbeault operator $\overline\partial_\Delta$ of $\Delta$ endows $\underline{\mathbb{C}}_\Delta$ with its standard holomorphic structure.
Let $e$ be the canonical holomorphic frame of $\underline{\mathbb{C}}_\Delta$, and denote by $e_z:=e(z)$ its value at $z\in\Delta$.
\par
Let 
\[
g=
\begin{pmatrix}
a&b\\
\bar b&\bar a
\end{pmatrix}
\in SU(1,1).
\]
For every $n\in\mathbb{Z}$ we define an $SU(1,1)$-action on $\underline{\mathbb{C}}_\Delta$ as 
\[
g\cdot e_z:=\frac{1}{(\bar bz+\bar a)^n}e_{g\cdot z}.
\]
We show that this defines an $SU(1,1)$-action.
Let
\[
g_i=
\begin{pmatrix}
a_i&b_i\\
\overline b_i&\overline a_i
\end{pmatrix}
\quad (i=1,2),
\]
and write
\[
g_1g_2=
\begin{pmatrix}
a_{12}&b_{12}\\
\overline b_{12}&\overline a_{12}
\end{pmatrix}.
\]
A direct computation gives
\[
\overline b_{12}z+\overline a_{12}
=
\overline b_1(a_2z+b_2)
+
\overline a_1(\overline b_2z+\overline a_2).
\]
Therefore
\begin{align*}
g_1\cdot(g_2\cdot e_z)
&=
\frac{1}{(\overline b_2z+\overline a_2)^n}
\frac{1}{\left(\overline b_1\frac{a_2z+b_2}{\overline b_2z+\overline a_2}+\overline a_1\right)^n}
e_{g_1g_2\cdot z}\\
&=\frac{1}{\bigg(\overline b_1(a_2z+b_2)+\overline a_1(\overline b_2z+\overline a_2)\bigg)^n}e_{g_1g_2\cdot z}\\
&=\frac{1}{(\overline b_{12}z+\overline a_{12})^n}e_{g_1g_2\cdot z}\\
&=(g_1g_2)\cdot e_z.
\end{align*}
It is also clear that the identity element acts trivially.
Hence this formula defines an $SU(1,1)$-action on $\underline{\mathbb{C}}_\Delta$.
This action is holomorphic. 
Hence $(\underline{\mathbb{C}}_\Delta,\overline\partial_\Delta)$, endowed with this action, is an $SU(1,1)$-equivariant holomorphic line bundle.
\par
Moreover, for
\[
k_\theta=
\begin{pmatrix}
e^{i\theta}&0\\
0&e^{-i\theta}
\end{pmatrix}
\in \mathrm{Stab}_0,
\]
we have
\[
k_\theta\cdot e_0=\frac{1}{(e^{-i\theta})^n}e_0=e^{in\theta}e_0.
\]
Thus, the induced $U(1)$-representation on the fiber over $0$ has weight $n$.
\par
For later use, we denote this $SU(1,1)$-equivariant holomorphic line bundle by $L_n$. 
Then $L_n$ is the weight $n$ $SU(1,1)$-equivariant holomorphic line bundle over $\Delta\simeq SU(1,1)/U(1)$ (See Section \ref{sec 3.1.2}).
 Its underlying holomorphic line bundle is still
\[
(\underline{\mathbb{C}}_\Delta,\overline\partial_\Delta),
\]
and we continue to denote its Dolbeault operator by $\overline\partial_\Delta$.
Although the underlying holomorphic frame is the same canonical frame $e$ of $\underline{\mathbb{C}}_\Delta$, we denote it by $e_n$ when it is regarded as a frame of
$L_n$, in order to indicate the weight. Thus the $SU(1,1)$-action on $L_n$ is written as
\[
g\cdot (e_n)_z=\frac{1}{(\overline b z+\overline a)^n}(e_n)_{g\cdot z}.
\]
\par
Since the $SU(1,1)$-action on $\Delta$ is holomorphic, it induces a natural $SU(1,1)$-action on $T^{1,0}\Delta$ by push-forward as follows.
For $v_z\in T^{1,0}_{z}\Delta$, we define
\[
g\cdot v_z:=(dg)_z(v_z)\in T^{1,0}_{g\cdot z}\Delta.
\]
We describe this induced action explicitly.
Let
\[
g=
\begin{pmatrix}
a&b\\
\overline b&\overline a
\end{pmatrix}\in SU(1,1).
\]
Since
\[
g\cdot z=\frac{az+b}{\overline b z+\overline a},
\]
we have
\[
\frac{d}{dz}(g\cdot z)=\frac{1}{(\overline b z+\overline a)^2}.
\]
Therefore
\[
g\cdot\left(\frac{\partial}{\partial z}\bigg|_z\right)
=\frac{1}{(\overline b z+\overline a)^2}\frac{\partial}{\partial z}\bigg|_{g\cdot z}.
\]
Since $T^{1,0*}\Delta$ is the dual bundle of $T^{1,0}\Delta$, the action on $T^{1,0}\Delta$ naturally induces an $SU(1,1)$-action on $T^{1,0*}\Delta$.
With respect to the frame $dz$, this induced action is given by
\[
g\cdot (dz)_z
=
(\overline b z+\overline a)^2(dz)_{g\cdot z}.
\]
In particular, for
\[
k_\theta=
\begin{pmatrix}
e^{i\theta}&0\\
0&e^{-i\theta}
\end{pmatrix}
\in \mathrm{Stab}_0,
\]
we have
\[
k_\theta\cdot (dz)_0=e^{-2i\theta}(dz)_0.
\]
Thus the induced $U(1)$-representation on $T^{1,0*}_0\Delta$ has weight $-2$.
Therefore
\[
T^{1,0\ast}\Delta\simeq L_{-2}
\]
as $SU(1,1)$-equivariant holomorphic line bundles.\par
Taking the complex conjugate of the above $SU(1,1)$-action on $T^{1,0*}\Delta$, we obtain a smooth $SU(1,1)$-action on $T^{0,1*}\Delta$ given by
\[
g\cdot (d\overline z)_z
=
(b\overline z+a)^2(d\overline z)_{g\cdot z}.
\]
In particular, for
\[
k_\theta=
\begin{pmatrix}
e^{i\theta}&0\\
0&e^{-i\theta}
\end{pmatrix}
\in \mathrm{Stab}_0,
\]
we have
\[
k_\theta\cdot (d\overline z)_0
=
e^{2i\theta}(d\overline z)_0.
\]
Thus the induced $U(1)$-representation on $T^{0,1*}_0\Delta$ has weight $2$.
Therefore
\[
T^{0,1*}\Delta\simeq L_2
\]
as smooth $SU(1,1)$-equivariant complex line bundles.
Then by Lemma \ref{lem 3.2}, we have 
\begin{lem}\label{lem 3.8}
\begin{itemize}
\item[(1)] The space $A^{0,1}(L_{-2})^{SU(1,1)}$ is one-dimensional and is generated by  
\[
\alpha=\frac{1}{(1-|z|^2)^2}d\bar z\otimes e_{-2}.
\]
\item[(2)] The space $A^{1,0}(L_{2})^{SU(1,1)}$ is one-dimensional and is generated by  
\[
\beta=dz\otimes e_{2}.
\]
\end{itemize}
\end{lem}
\bpr
We only prove (1) since (2) follows from a similar argument.
\par
Since 
\[
T^{0,1*}\Delta\simeq L_2
\]
as smooth $SU(1,1)$-equivariant complex line bundles, we have 
\begin{align*}
 A^{0,1}(L_{-2})^{SU(1,1)}&=A(T^{0,1*}\Delta\otimes L_{-2})^{SU(1,1)}\\
 &=A(L_2\otimes L_{-2})^{SU(1,1)}\\
 &=A(L_0)^{SU(1,1)}\\
 &\simeq \mathbb{C}.
\end{align*}
The last follows from Lemma \ref{lem 3.3}.
Hence  $\dim A^{0,1}(L_{-2})^{SU(1,1)}=1.$ 
We show that 
\[
\alpha=\frac{1}{(1-|z|^2)^2}d\bar z\otimes e_{-2}\in A^{0,1}(L_{-2})
\]
is $SU(1,1)$-equivariant and hence the basis of $A^{0,1}(L_{-2})^{SU(1,1)}.$
Equivalently, 
for
\[
g=
\begin{pmatrix}
a&b\\
\overline b&\overline a
\end{pmatrix}\in SU(1,1),
\]
we show that 
\[
\alpha_{g\cdot z}=g\cdot\alpha_z.
\]
First, we have
\begin{align*}
\alpha_{g\cdot z}&=\frac{1}{(1-|g\cdot z|^2)^2}(d\bar z\otimes e_{-2})_{g\cdot z}\\
&=\frac{|\bar b z+\bar\alpha|^4}{(1-|z|^2)^2}(d\bar z\otimes e_{-2})_{g\cdot z}.
\end{align*}
On the other hand, by the actions on $T^{0,1*}\Delta$ and $L_{-2}$,
\begin{align*}
g\cdot \alpha_z&=g\cdot \frac{1}{(1-|z|^2)^2}(d\bar z\otimes e_{-2})_z\\
&=\frac{|\bar b z+\bar\alpha|^4}{(1-|z|^2)^2}(d\bar z\otimes e_{-2})_{g\cdot z}.
\end{align*}
Hence $\alpha\in A^{0,1}(L_{-2})^{SU(1,1)}.$
\epr
We next define an $SU(1,1)$-invariant Hermitian metric on $L_n$.
For every $n\in\mathbb{Z}$, we  define a Hermitian metric $h^{(n)}$ on $L_n$ as
\[
h^{(n)}_z(e_{n,z},e_{n,z}):=\frac{1}{(1-|z|^2)^n}.
\]
Since
\begin{align*}
h^{(n)}_{g\cdot z}(g\cdot e_{n,z},g\cdot e_{n,z})&=\frac{1}{|\bar bz+\bar a|^{2n}}\cdot h^{(n)}_{g\cdot z}( e_{n,g\cdot z},e_{n,g\cdot z})\\
&=\frac{1}{|\bar bz+\bar a|^{2n}}\cdot\frac{1}{(1-|g\cdot z|^2)^{n}}\\
&=\frac{1}{(1-|z|^2)^n}\\
&= h^{(n)}_z(e_{n,z},e_{n,z}),
\end{align*}
$h^{(n)}$ is $SU(1,1)$-invariant.
\par
For later use, we recall the Poincar\'e metric on $\Delta$ and compare its associated K\"ahler form with the curvature of the Chern connection associated with $(\overline\partial_\Delta,h^{(-2)})$ on $L_{-2}$.
The Poincar\'e metric on $\Delta$ is given by
\[
g_{\mathrm{Poi}}=\frac{4\,dz\,d\overline z}{(1-|z|^2)^2}.
\]
The associated K\"ahler form is
\[
\omega_{\mathrm{Poi}}
=
\frac{2\sqrt{-1}\,dz\wedge d\overline z}{(1-|z|^2)^2}.
\]
Let $F_{h^{(-2})}$ be the curvature of the Chern connection associated with $(\overline\partial_\Delta,h^{(-2)})$ on $L_{-2}$.
Then, with respect to the frame $e_{-2}$, $F_{h^{(-2)}}$ has the form 
\begin{equation}\label{eq 3.15}
\begin{split}
F_{h^{(-2)}}&=\bar\partial\partial\log h^{(-2)}(e_{-2},e_{-2})\\
&=\bar\partial\partial\log (1-|z|^2)^2\\
&=-2\bar\partial\frac{\bar z dz}{1-|z|^2}\\
&=-2\frac{d\bar z\wedge dz}{1-|z|^2}-\frac{2|z|^2 d\bar z\wedge dz}{(1-|z|^2)^2}\\
&=\frac{2dz\wedge d\bar z}{(1-|z|^2)^2}\\
&=-\sqrt{-1}\omega_{\mathrm{Poi}}.
\end{split}
\end{equation}
Let $\alpha$ and $\beta$ as in Lemma \ref{lem 3.9}.
Then we have 
\be\label{eq 3.16}
\alpha\wedge\beta=\frac{\sqrt{-1}}{2}\omega_{\mathrm{Poi}}
\ee
\subsubsection{Dimensional Reduction}\label{sec 3.3.2}
We use the same notation as in the previous section.
Let $X$ be a compact Riemann surface, and $\omega_X$ be a K\"ahler form on $ X$.
We normalize $\omega_X$ so that 
\[
\mathrm{Vol}_{\omega_X}=\int_X\omega_X=1.\]
We now consider the action of $SU(1,1)$ on $X\times \Delta$, which is trivial on $X$ and standard on $\Delta$.\par
The argument in this section closely parallels that in Section \ref{sec 3.2.1}.
However, whereas Section~3.2.2 deals with the $SU(2)$-equivariant geometry over $X\times \mathbb{P}^1$, we work here with the $SU(1,1)$-equivariant geometry over $X\times \Delta$.
\par
In the first half of this section, starting from a solution of Manton's exotic $\tau$-vortex equation with $\tau<0$ on $X$, we construct an $SU(1,1)$-equivariant holomorphic bundle over $X\times \Delta$ and an $SU(1,1)$-invariant pseudo-Hermitian-Einstein metric of signature $(1,1)$ on it.
In the second half, we establish the converse.
\par
From now on, we assume $\tau<0$, and we denote by 
\begin{align*}
&p:X\times \Delta\to X\\
&q:X\times \Delta\to \Delta
\end{align*}
the natural projections.\par
Let $(L,\overline\partial_L)$ be a holomorphic line bundle over $X$, $\phi \in A(L)$, and $h_L$ be a Hermitian metric on $L$.
We assume that the triple 
$((L,\overline\partial_L),\phi,h_L)$ satisfies Manton's exotic $\tau$-vortex equation with $\tau<0$.
We construct an $SU(1,1)$-equivariant holomorphic bundle over $X\times \Delta$ and an $SU(1,1)$-invariant pseudo-Hermitian-Einstein metric of signature $(1,1)$ from $((L,\overline\partial_L),\phi,h_L)$.
We first construct an $SU(1,1)$-equivariant holomophic bundle over $X\times\Delta$.
Recall that we denote the trivial line bundle over $X$ as $\underline{\mathbb{C}}_X$.
The Dolbeault operator $\overline\partial_{X\times\Delta}$ of $X\times\Delta$ defines a holomorphic bundle structure of $\underline{\mathbb{C}}_{X\times \Delta}$.
Let $c\in\mathbb{C}$ be a constant.
We define
\begin{align}
E:=&p^\ast L\otimes q^\ast L_{-2}\oplus \underline{\mathbb{C}}_{X\times \Delta},\nonumber\\
\overline\partial_{E,c}:=&
\begin{pmatrix}
p^\ast \overline\partial_L\otimes\mathrm{Id}+\mathrm{Id}\otimes q^\ast \overline\partial_{\Delta}& c\cdot p^\ast \phi\otimes q^\ast \alpha\\
0&\overline\partial_{X\times\Delta}
\end{pmatrix}.\label{eq 3.17}
\end{align}
See Lemma \ref{lem 3.8} for $\alpha$.
We also note that $c\cdot  p^\ast \phi\otimes q^\ast \alpha\in A^{0,1}(\mathrm{Hom}(\underline{\mathbb{C}}_{X\times \Delta}, p^\ast L\otimes q^\ast L_{-2})$.
The $SU(1,1)$-action on $L_{-2}$ induces an $SU(1,1)$-action on $E$, where $SU(1,1)$ acts trivially on the factors pulled back from $X$.
By the definition of the action of $SU(1,1)$ on $X\times \Delta$, $E$ is also an $SU(1,1)$-equivariant bundle over $X\times \Delta$.
By the same argument as in Proposition \ref{prop 3.3}, we obtain the following proposition.
\begin{prop}
$(E,\overline\partial_{E,c})$ is an $SU(1,1)$-equivariant holomorphic bundle over $X\times\Delta$.
\end{prop}
Hence 
\[
0\longrightarrow p^\ast L\otimes q^\ast L_{-2}\longrightarrow E\longrightarrow \underline{\mathbb{C}}_{X\times \Delta} \longrightarrow0.
\]
is a holomorphic bundle extension.\par
Let $h^{(-2)}$ be the $SU(1,1)$-invariant Hermitian metric of $L_{-2}$ (See Section \ref{sec 3.3.1}), and $h_X$ be a Hermitian metric of $\underline{\mathbb{C}}_X$.
We note that $h_X$ is a smooth positive function on $X$ and therefore $h_Xh_L$ is also a Hermitian metric on $L$.
We define a pseudo-Hermitian metric $h_E$ on $E$ as 
\begin{align}\label{eq 3.18}
h_E:=
\begin{pmatrix}
p^\ast (h_Xh_L)\otimes q^\ast h^{(-2)}&0\\
0& -p^\ast h_X
\end{pmatrix}.
\end{align}
It follows from the construction that $h_E$ is a pseudo-Hermitian metric of signature $(1,1)$.
Since $h^{(-2)}$ is $SU(1,1)$-invariant, and $SU(1,1)$ action on $E$ is trivial on the factors pulled back from $X$, the pseudo-Hermitian metric $h_E$ is also $SU(1,1)$-invariant.
\par
We denote by $\partial_{h_Xh_L},\partial_{h^{(-2)}}$, and $\partial_{h_X}$ the $(1,0)$-part of the Chern connection associated to $(L,\overline\partial_L,h_Xh_L),$ $(L_{-2},\overline\partial_\Delta,h^{(-2)})$, and $(\mathbb{C}_X,\overline\partial_X,h_X)$ respectively.
Let $\nabla_{h_E,c}$ be the unique connection that is compatible with $h_E$ and $\nabla_{h_E,c}^{0,1}=\overline\partial_{E,c}$ (See Proposition \ref{prop 3.1}), and $\partial_{h_E,c}$ the $(1,0)$-part of $\nabla_{h_E}$.
By the same argument as in Lemma \ref{lem 3.5}, we obtain the following Lemma.
\begin{lem}\label{lem 3.9}
$\partial_{h_E,c}$ has the form
\begin{equation}\label{eq 3.19}
\partial_{h_E,c}
=\begin{pmatrix}
p^\ast\partial_{h_Xh_L}\otimes \mathrm{Id} +\mathrm{Id}\otimes p^\ast\partial_{h^{(-2)}}&0\\
\overline c \cdot p^\ast (h_L\overline \phi)\otimes q^\ast \beta& p^\ast \partial_{h_X}
\end{pmatrix}.
\end{equation}
See Lemma \ref{lem 3.9} for $\beta$.
\end{lem}
We show that for a suitable $c\in \mathbb{C}$ and a suitable K\"ahler form on $X\times \Delta$, there exists a Hermitian metric $h_X$ of $\underline{\mathbb{C}}_X$ such that $h_E$ is a pseudo-Hermitian-Einstein metric of $(E,\overline\partial_{E,c})$ (Definition \ref{def 3.2}).\par
Let $\omega_{\mathrm{Poi}}$ be the K\"ahler form on $\Delta$, associated to the Poinc\'are metric (See Section \ref{sec 3.3.1}).
We define a K\"ahler form of $X\times \Delta$ as 
\be\label{eq 3.20}
\Omega_\tau:=\bigg(-\frac{\tau}{2}p^\ast\omega_X\bigg)\oplus q^\ast \omega_{\mathrm{Poi}}.
\ee
Note that $\tau<0$.\\
We now set 
\[
c=\sqrt{-\frac{1}{\tau}}
\]
and consider $(E,\overline\partial_{E,\sqrt{-\frac{1}{\tau}}})$.
For simplicity, we write
\[
\overline\partial_E:=\overline\partial_{E,\sqrt{-1/\tau}}.
\]
We also write 
\begin{align*}
\nabla_{h_E}&:=\nabla_{h_E,\sqrt{-1/\tau}},\\
\partial_{h_E}&:=\partial_{h_E, \sqrt{-1/\tau}}.
\end{align*}
We denote the curvature of $\nabla_{h_E}$ as
\[
F_{h_E}:=(\nabla_{h_E})^2.
\]
We also denote by $F_{h_Xh_L}, F_{h^{(-2)}}$, and $F_{h_X}$ the curvature of the Chern connections of $(L,\overline\partial_L,$$h_Xh_L),$ $(L_{(-2)},\overline\partial_{\Delta},h^{(-2)})$, and $(\underline{\mathbb{C}}_X,\overline\partial_X, h_X)$ respectively.
The next Proposition is an analogy of Proposition \ref{prop 3.4}. 
However, since we are now considering a vector bundle over the non-compact space $X\times\Delta$, we also need to specify the constant appearing in \eqref{eq 3.1}.
\begin{prop}\label{prop 3.9}
There exists a Hermitian metric $h_X$ of $\underline{\mathbb{C}}_X$ such that the pseudo-Hermitian metric $h_E$ defined by \eqref{eq 3.18} is a pseudo-Hermitian-Einstein metric with respect to the K\"ahler form $\Omega_\tau$ and the constant $\frac{1}{2}(1-\frac{4\pi}{\tau}\mathrm{deg}L)$;
that is
\be
\sqrt{-1}\Lambda_{\Omega_\tau}F_{h_E}=\frac{1}{2}\bigg(1-\frac{4\pi}{\tau}\mathrm{deg}L\bigg) \mathrm{Id}_E.
\ee
Equivalently, the connection $\nabla_{h_E}$ is a pseudo-Hermitian-Yang-Mills connection.
\end{prop}
\begin{proof}
For brevity, we set
\[
\lambda:=\frac{1}{2}\bigg(1-\frac{4\pi}{\tau}\mathrm{deg}L\bigg).
 \]
 \par
Let $h_X$ be a Hermitian metric on $\mathbb{C}_X$ and let $h_E$ be a pseudo-Hermitian metric of signature (1,1) defined as \eqref{eq 3.18}.
Since $(\overline\partial_E)^2=(\partial_{h_E})^2=0$, we have
\[
F_{h_E}=\overline\partial_E\partial_{h_E}+\partial_{h_E}\overline\partial_E.
\]
We set
\[
F_{h_E}=
\begin{pmatrix}
F_{11}&F_{12}\\
F_{21}&F_{22}
\end{pmatrix}.
\]
Then by \eqref{eq 3.17} and \eqref{eq 3.19} we have 
\[
\begin{split}
F_{11}&=p^\ast F_{h_Xh_L}+q^\ast F_{h^{(-2)}}-\frac{1}{\tau}p^\ast|\phi|^2_{h_L}\otimes q^\ast(\alpha\wedge \beta)\\
&=p^\ast F_{h_L}+p^\ast F_{h_X}-\sqrt{-1}q^\ast \omega_{\mathrm{Poi}}-\frac{\sqrt{-1}}{2\tau}p^\ast|\phi|^2_{h_L}q^\ast\omega_{\mathrm{Poi}}.
\end{split}
\]
The last equation follows since $h_X$ is a positive function on $X$, \eqref{eq 3.15}, and \eqref{eq 3.16}.
Therefore, we have 
\be\label{eq 3.22}
\begin{split}
\sqrt{-1}\Lambda_{\Omega_\tau}F_{11}&=-\frac{2\sqrt{-1}}{\tau}p^\ast \Lambda_{\omega_X}(F_{h_L}+F_{h_X})+1+\frac{1}{2\tau}p^\ast|\phi|^2_{h_L}.
\end{split}
\ee
We also have 
\[
\begin{split}
F_{22}&=p^\ast F_{h_X}-\frac{1}{\tau}p^\ast|\phi|^2_{h_L}\otimes q^\ast(\beta\wedge \alpha)\\
&=p^\ast F_{h_X}+\frac{\sqrt{-1}}{2\tau}|\phi|^2_{h_L}q^\ast \omega_{\mathrm{Poi}}.
\end{split}
\]
Therefore
\be\label{eq 3.23}
\sqrt{-1}\Lambda_{\Omega_\tau}F_{22}=-\frac{2\sqrt{-1}}{\tau}p^\ast \Lambda_{\omega_X}F_{h_X}-\frac{1}{2\tau}p^\ast|\phi|^2_{h_L}.
\ee
Since $F_{12},F_{21}$ has mixed contribution from $X$ and $\Delta$ we have
\[
\sqrt{-1}\Lambda_{\Omega_\tau}F_{h_E}
=
\begin{pmatrix}
\sqrt{-1}\Lambda_{\Omega_\tau}F_{11}&0\\
0&\sqrt{-1}\Lambda_{\Omega_\tau}F_{22} \\
\end{pmatrix}.
\]
Hence by \eqref{eq 3.22} and \eqref{eq 3.23} the equation
\[
\sqrt{-1}\Lambda_{\Omega_\tau}F_{h_E}=\lambda \mathrm{Id}_E
\]
is equivalent to 
\begin{align}
-\frac{2\sqrt{-1}}{\tau} \Lambda_{\omega_X}(F_{h_L}+F_{h_X})+1+\frac{1}{2\tau}|\phi|^2_{h_L}&=\lambda,\label{eq 3.24}
\\
-\frac{2\sqrt{-1}}{\tau} \Lambda_{\omega_X}F_{h_X}-\frac{1}{2\tau}|\phi|^2_{h_L}&=\lambda.\label{eq 3.25}
\end{align}
The equations \eqref{eq 3.24} and \eqref{eq 3.25} are equivalent to 

\begin{align}
-\frac{2\sqrt{-1}}{\tau} \Lambda_{\omega_X}F_{h_L}+1+\frac{1}{\tau}|\phi|^2_{h_L}&=0,\label{eq 3.26}
\\
-\frac{2\sqrt{-1}}{\tau} \Lambda_{\omega_X}F_{h_L}-\frac{4\sqrt{-1}}{\tau} \Lambda_{\omega_X}F_{h_X}+1&=2\lambda,\label{eq 3.27}
\end{align}
since \eqref{eq 3.26} is obtained by $\eqref{eq 3.24}-\eqref{eq 3.25}$ and  \eqref{eq 3.27} is obtained by $\eqref{eq 3.24}+\eqref{eq 3.25}$.
\par
\eqref{eq 3.26} holds since the triple $((L,\overline\partial_L),\phi,h_L)$ satisfies Manton's exotic $\tau$-vortex equation.
We show that there exists a Hermitian metric $h_X$ such that \eqref{eq 3.27} holds.
By the same argument as in Proposition \ref{prop 3.4}, the existence of such a metric is equivalent to 
\[
\int_{X}\bigg(-1+2\lambda+\frac{2\sqrt{-1}}{\tau} \Lambda_{\omega_X}F_{h_L}\bigg)=0.
\]
Then by the definition of $\lambda$, we have
\begin{align*}
\int_{X}\bigg(-1+2\lambda+\frac{2\sqrt{-1}}{\tau} \Lambda_{\omega_X}F_{h_L}\bigg)&=-1+2\lambda+\frac{4\pi}{\tau}\mathrm{deg}L\\
&=0.
\end{align*}
Hence, the claim is proved.
\end{proof}
We give the summary of the results so far.
\begin{thm}\label{thm 3.3}
Let $X$ be a compact connected Riemann surface and let  $\omega_X$ be a K\"ahler form of $X$ satisfying $\int_X\omega_X=1$.
Let $(L,\overline\partial_L)$ be a holomorphic line bundle over $X$, $\phi$ a holomorphic section  of $(L,\overline\partial_L)$, and $h_L$ a Hermitian metric on $L$.
Suppose the triple $((L,\overline\partial_L),\phi,h_L)$ satisfies Manton's exotic $\tau$-vortex equation with $\tau<0$.
Then 
\begin{align*}
E:=&p^\ast L\otimes q^\ast L_{-2}\oplus \underline{\mathbb{C}}_{X\times\Delta},\\
\overline\partial_{E}:=&
\begin{pmatrix}
p^\ast \overline\partial_L\otimes \mathrm{Id}+\mathrm{Id}\otimes q^\ast \overline\partial_{\Delta}& \sqrt{-\frac{1}{\tau}}\cdot p^\ast \phi\otimes q^\ast \alpha\\
0&\bar\partial_{X\times\mathbb{P}^1}
\end{pmatrix}.
\end{align*}
is an $SU(1,1)$-equivariant holomorphic bundle over $X\times \Delta$.
We fix a K\"ahler form on $X\times\Delta$ given by
\[
\Omega_\tau=\bigg(-\frac{\tau}{2}p^\ast\omega_X\bigg)\oplus q^\ast \omega_{\mathrm{Poi}}.
\]
Then there exists a Hermitian metric $h_X$ on $\underline{\mathbb{C}}_X$ such that the pseudo-Hermitian metric of signature (1,1)
\begin{align*}
h_E=
\begin{pmatrix}
p^\ast (h_Xh_L)\otimes q^\ast h^{(-2)}&0\\
0& -p^\ast h_X
\end{pmatrix}
\end{align*}
is a pseudo-Hermitian-Einstein metric of $(E,\overline\partial_E)$ with respect to the K\"ahler form $\Omega_\tau$ and the constant $\frac{1}{2}(1-\frac{4\pi}{\tau}\mathrm{deg}L)$.
Equivalently, the unique connection $\nabla_{h_E}$ which is compatible with $h_E$ and $\nabla^{0,1}_{h_E}=\bar\partial_E$ is a pseudo-Hermitian-Yang-Mills connection.
\end{thm}
We next prove the dimensional reduction: we construct a solution of Manton's exotic $\tau$-vortex ($\tau<0$) equation from an $SU(1,1)$-equivariant holomorphic bundle and an $SU(1,1)$-invariant pseudo-Hermitian-Einstein metric of signature $(1,1).$\par
Let $(L,\overline\partial_L)$ be a holomorphic line bundle over $X$.
Let $E$ be a holomorphic bundle on $X\times\Delta$ that is an extension of $\underline{\mathbb{C}}_{X\times \Delta}$ by $p^\ast L \otimes q^\ast L_{-2}$.
Thus, there exists a short exact sequence of holomorphic vector bundles
\[
0\longrightarrow p^\ast L \otimes q^\ast L_{-2}\longrightarrow E\longrightarrow \underline{\mathbb{C}}_{X\times \Delta}\longrightarrow 0.
\]
We also have
\[
E\simeq p^\ast L \otimes q^\ast L_{-2}\oplus\underline{\mathbb{C}}_{X\times \Delta}
\]
as smooth bundles.
From now on, we identify $E$ and $p^\ast L \otimes q^\ast L_{-2}\oplus \underline{\mathbb{C}}_{X\times \Delta}$.
Let $\bar\partial_E$ be the Dolbeault operator of $E$.
Since $E$ given by an extension, $\bar\partial_E$ can be written in the form of 
\[
\overline\partial_E=
\begin{pmatrix}
p^\ast \overline\partial_L\otimes \mathrm{Id}+\mathrm{Id}\otimes q^\ast \overline\partial_{\Delta}& \psi\\
0&\overline\partial_{X\times\Delta}
\end{pmatrix}
\]
where $\psi\in A^{0,1}(\mathrm{Hom}(\underline{\mathbb{C}}_{X\times \Delta}, p^\ast L \otimes q^\ast L_{-2})$.
\par
The bundle $E$ is $SU(1,1)$-equivariant with respect to the $SU(1,1)$-action induced by that on $L_{-2}$.
The following is an analogue of Proposition \ref{prop 3.5} and can be proved by the same argument.
\begin{prop}\label{prop 3.10}
If $(E,\overline\partial_E)$ is an $SU(1,1)$-equivariant holomorphic bundle, then
\begin{itemize}
\item[(1)] there exists a $\phi\in A(L)$ such that 
\[
\psi=p^\ast\phi\otimes q^\ast\alpha.
\]
\item[(2)] $\phi$ is a holomorphic section of $(L,\overline\partial_L)$.
\end{itemize}
\end{prop}
Since multiplying a holomorphic section by a nonzero constant preserves holomorphicity, $\sqrt{-\tau}\phi$ is also holomorphic.
Thus, after replacing $\sqrt{-\tau}\phi$ by $\phi$, we may write
\[
\psi
=
\sqrt{-\frac{1}{\tau}}\cdot p^\ast\phi\otimes q^\ast\alpha.
\]
We keep the normalization of $\alpha$ fixed.\par
As in Proposition \ref{prop 3.6}, any $SU(1,1)$-invariant pseudo-Hermitian metric on $E$ is of the following form. 
The proof is the same as that of Proposition \ref{prop 3.6}.
\begin{prop}\label{prop 3.11}
Let $h_E$ be an $SU(1,1)$-invariant pseudo-Hermitian metric on $E$.
Then there exists a Hermitian metrics $h_L$ of $L$, $h_X$ of $\underline{\mathbb{C}}_X$ such that
\[
h_E
=
\begin{pmatrix}
C_0\cdot p^\ast h_L\otimes q^\ast h^{(-2)}&0\\
0& C_1\cdot p^\ast h_X
\end{pmatrix},\,\,C_0,C_1\in\{1,-1\}.
\]
\end{prop}
The following is a consequence of Proposition \ref{prop 3.11} and a rescaling of the Hermitian metric on $L$.
\begin{lem}\label{lem 3.10}
Let $h_E$ be an $SU(1,1)$-invariant pseudo-Hermitian metric of signature $(1,1)$ on $E$.
Then there exist Hermitian metrics $h_L$ on $L$ and $h_X$ on $\underline{\mathbb{C}}_X$ such that 
\[
h_E
=
\begin{pmatrix}
C_0\cdot p^\ast (h_Xh_L)\otimes q^\ast h^{(-2)}&0\\
0& C_1\cdot p^\ast h_X
\end{pmatrix}
\]
where $(C_0,C_1)=(1,-1)$ or $(-1,1)$.
\end{lem} 
We now establish the dimensional reduction.
Throughout this argument, we work with the K\"ahler form $\Omega_\tau$ constructed in \eqref{eq 3.20}.
\par
Suppose that $\overline\partial_E$ defines an $SU(1,1)$-equivariant holomorphic bundle structure on $E$.
Then by Proposition \ref{prop 3.10}, there exists a holomorphic section $\phi$ of $(L,\bar\partial_L)$ such that 
\[
\overline\partial_E=
\begin{pmatrix}
p^\ast \overline\partial_L\otimes \mathrm{Id}+\mathrm{Id}\otimes q^\ast \overline\partial_{\Delta}& \sqrt{-\frac{1}{\tau}}\cdot p^\ast\phi\otimes q^\ast\alpha\\
0&\overline\partial_{X\times\Delta}
\end{pmatrix}.
\]
Let $h_E$ be as in Lemma \ref{lem 3.10}, and let $\nabla_{h_E}$ be the connection compatible with $h_E$ and satisfying $\nabla^{0,1}_{h_E}=\bar\partial_E$.
Then, by Lemma \ref{lem 3.9}, $\partial_{h_E}=\nabla^{1,0}_{h_E}$ has the form 
\[
\partial_{h_E}=
\begin{pmatrix}
p^\ast\partial_{h_Xh_L}\otimes \mathrm{Id} +\mathrm{Id}\otimes p^\ast\partial_{h^{(-2)}}&0\\
\sqrt{-\frac{1}{\tau}} \cdot p^\ast (h_L\overline \phi)\otimes q^\ast \beta& p^\ast \partial_{h_X}
\end{pmatrix}.
\]
We denote the curvature of $\nabla_{h_E}$ as $F_{h_E}$.
The following proposition is an analogue of Proposition \ref{prop 3.7}.
 In the present case, however, the base space is non-compact, and hence the constant in the Hermitian-Yang-Mills equation is not determined by the topological data of the bundle.
Nevertheless, we show that this constant is forced to take a specific value.

\begin{prop}\label{prop 3.12}
If $h_E$ is a pseudo-Hermitian-Einstein metric of $(E,\overline\partial_E)$ with respect to the K\"ahler form $\Omega_\tau$, that is,
\[
\sqrt{-1}\Lambda_{\Omega_\tau}F_{h_E}=\lambda \mathrm{Id}_E
\]
for some constant $\lambda\in\mathbb{C}$, then 
\begin{itemize}
\item[(1)] $(\overline\partial_L,\phi,h_L)$ satisfies Manton's exotic $\tau$-vortex equation with $\tau<0$.
\item[(2)]
\[
\lambda=\frac{1}{2}\bigg(1-\frac{4\pi}{\tau}\mathrm{deg}L\bigg).
\]
\end{itemize}
\end{prop}
\begin{proof}
We set 
\[
F_{h_E}=
\begin{pmatrix}
F_{11}&F_{12}\\
F_{21}&F_{22}.
\end{pmatrix}
\]
The equaiton
\[
\sqrt{-1}\Lambda_{\Omega_\tau}F_{h_E}=\lambda Id_E
\]
implies 
\begin{align*}
\sqrt{-1}\Lambda_{\Omega_\tau}F_{11}&=\lambda,\\
\sqrt{-1}\Lambda_{\Omega_\tau}F_{22}&=\lambda.
\end{align*}
Then, by \eqref{eq 3.24}, \eqref{eq 3.25}, \eqref{eq 3.26}, and \label{eq 3.27}, 
we have 
\begin{align*}
-\frac{2\sqrt{-1}}{\tau} \Lambda_{\omega_X}(F_{h_L}+F_{h_X})+1+\frac{1}{2\tau}|\phi|^2_{h_L}&=\lambda,
\\
-\frac{2\sqrt{-1}}{\tau} \Lambda_{\omega_X}F_{h_X}-\frac{1}{2\tau}|\phi|^2_{h_L}&=\lambda.
\end{align*}
Subtracting the second equation from the first, we obtain
\[
-\frac{2\sqrt{-1}}{\tau} \Lambda_{\omega_X}F_{h_L}+1+\frac{1}{\tau}|\phi|^2_{h_L}=0.
\]
Hence $(\overline\partial_L,\phi,h_L)\in\mathrm{Tri}_\tau$ is a solution of Manton's exotic $\tau$-vortex equation.\par
In addition, adding the two equations, we obtain 
\be\label{eq 3.28}
-\frac{2\sqrt{-1}}{\tau} \Lambda_{\omega_X}F_{h_L}-\frac{4\sqrt{-1}}{\tau} \Lambda_{\omega_X}F_{h_X}+1=2\lambda.
\ee
Recall that 
\[
\sqrt{-1} \Lambda_{\omega_X}F_{h_X}=\frac{1}{2}\triangle h_X.
\]
Hence the equation \eqref{eq 3.28} is equivariant to 
\[
-\frac{2\sqrt{-1}}{\tau} \Lambda_{\omega_X}F_{h_L}-\frac{2}{\tau}\triangle h_X+1=2\lambda.
\]
This equation implies 
\begin{align*}
0&=\int_X\bigg(2\lambda-1+\frac{2\sqrt{-1}}{\tau} \Lambda_{\omega_X}F_{h_L}\bigg)\\
&=2\lambda-1+\frac{4\pi}{\tau}\mathrm{deg}L
\end{align*}
which derive (2).
\end{proof}

Combining the results obtained above, we obtain the following dimensional reduction theorem.
\begin{thm}
Let $X$ be a compact connected Riemann surface and let  $\omega_X$ be a K\"ahler form of $X$ satisfying $\int_X\omega_X=1$.
Let $(L,\overline\partial_L)$ be a holomorphic line bundle over $X$.
Let $E$ be a holomorphic bundle over $X\times \Delta$ which is an extension of $\underline{\mathbb{C}}_{X\times \Delta}$ by $p^\ast L\otimes q^\ast L_{-2}$.
Let $\overline\partial_E$ be the Dolbeault operator of $E$.
On $X\times\Delta$, we consider the K\"ahler form $\Omega_\tau$ constructed in \eqref{eq 3.20}.
\par
If $(E,\overline\partial_E)$ is an $SU(1,1)$-equivariant holomorphic bundle, then there exists a holomorphic section $\phi$ of $(L,\bar\partial_L)$ such that 
\[
\overline\partial_E=
\begin{pmatrix}
p^\ast \overline\partial_L\otimes \mathrm{Id}+\mathrm{Id}\otimes q^\ast \overline\partial_{\Delta}& \sqrt{-\frac{1}{\tau}}\cdot p^\ast\phi\otimes q^\ast\alpha\\
0&\overline\partial_{X\times\Delta}
\end{pmatrix}.
\]
Let $h_E$ be a pseudo-Hermitian metric of signature $(1,1)$ on $E.$
If $h_E$ is $SU(1,1)$-invariant, then there exist Hermitian metrics $h_L$ on $L$ and $h_X$ on $\underline{\mathbb{C}}_X$ such that 
\[
h_E
=
\begin{pmatrix}
C_0\cdot p^\ast (h_Xh_L)\otimes q^\ast h^{(-2)}&0\\
0& C_1\cdot p^\ast h_X
\end{pmatrix}
\]
where $(C_0,C_1)=(1,-1)$ or $(-1,1)$.\par
Let $h_E$ and $(E,\overline\partial_E)$ be as above, and $\nabla_{h_E}$ be the connection compatible with $h_E$ and $\nabla^{0,1}_{h_E}=\bar\partial_E$.
If $h_E$ is a pseudo-Hermitian-Einstein metric of $(E,\overline\partial_E)$ with respect to $\Omega_\tau$, or equivalently, if $\nabla_{h_E}$ is a pseudo-Hermitian-Yang-Mills connection, then
\begin{itemize}
\item The triple $((L,\overline\partial_L),\phi,h_L)$ satisfies Manton's exotic $\tau$-vortex equation with $\tau<0$.
\item The constant $\lambda$ appearing in \eqref{eq 3.1} satisfies
\[
\lambda=\frac{1}{2}\bigg(1-\frac{4\pi}{\tau}\mathrm{deg}L\bigg).
\]
\end{itemize}
\end{thm}

\subsection{$\tau=0$ case}\label{sec 3.4}
\subsubsection{$SE(2)$-equivariant line bundle over $\mathbb{C}$}\label{sec 3.4.1}
The special Euclidean group $SE(2)$ is defined as the semi-direct product 
\[
SE(2):=\mathbb{C}\rtimes U(1).
\]
Let $(a,e^{i\theta})\in SE(2)$ and $z\in\mathbb{C}$.
$SE(2)$ acts on $\mathbb{C}$ as 
\[
(a,e^{i\theta})\cdot z:=a+e^{i\theta} z.
\]
The action is holomorphic.
Moreover, it is transitive since 
\[
(z,1)\cdot 0=z.
\]
The stabilizer subgroup at $0\in \mathbb{C}$ is 
\[
\mathrm{Stab}_0=\{(0,e^{i\theta})\,\,|\,\, \theta\in\mathbb{R}\}\simeq U(1).
\]

Hence, we identify $\mathbb{C}$ with the homogeneous space $SE(2)/U(1)$.
We denote by $\underline{\mathbb{C}}_{\mathbb{C}}$ the trivial line bundle over $\underline{\mathbb{C}}$.
The standard Dolbeault operator $\overline\partial_{\mathbb{C}}$ of $\mathbb{C}$ endows $\mathbb{C}_{\mathbb{C}}$ with its standard holomorphic structure.
Let $e$ be the canonical holomorphic frame of $\mathbb{C}_\mathbb{C}$, and denote by $e_z:=e(z)$ its value at $z\in\mathbb{C}$.
\par
Let 
\[
g=(a,e^{i\theta})\in SE(2).
\]
For every $n\in\mathbb{Z}$, we define an $SE(2)$-action on $\mathbb{C}_{\mathbb{C}}$ by
\[
g\cdot e_z:=e^{in\theta}e_{g\cdot z}.
\]
We show that this defines an $SE(2)$-action.
Let
\[
g_j=
(a_j,e^{i\theta_j})
\quad (j=1,2).
\]
Then
\[
g_1g_2=(a_1+e^{i\theta_1}a_2,e^{i(\theta_1+\theta_2)}).
\]
Therefore
\begin{align*}
g_1\cdot(g_2\cdot e_z)
&=e^{in\theta_2}g_1\cdot e_{g_2\cdot z}\\
&=e^{in\theta_2+in\theta_1}e_{g_1g_2\cdot z}\\
&=(g_1g_2)\cdot e_z.
\end{align*}
It is also clear that the identity element acts trivially.
Hence this formula defines an $SE(2)$-action on $\underline{\mathbb{C}}_{\mathbb{C}}$.
This action is holomorphic. 
Hence $(\underline{\mathbb{C}}_{\mathbb{C}},\overline\partial_{\mathbb{C}})$, endowed with this action, is an $SE(2)$-equivariant holomorphic line bundle.
\par
For later use, we denote this $SE(2)$-equivariant holomorphic line bundle by
$L_n$.
Then, by construction, $L_n$ is the weight $n$ $SE(2)$-equivariant holomorphic line bundle
over $\mathbb{C}\simeq SE(2)/U(1)$ (See Section \ref{sec 3.2.1}).
Its underlying holomorphic line bundle is still
\[
(\underline{\mathbb{C}}_{\mathbb{C}},\overline\partial_{\mathbb{C}}),
\]
and we continue to denote its Dolbeault operator by
$\overline\partial_{\mathbb{C}}$.
Although the underlying holomorphic frame is the same canonical frame $e$
of $\underline{\mathbb{C}}_{\mathbb{C}}$, we denote it by $e_n$ when it is
regarded as a frame of $L_n$, in order to indicate the weight.
Thus the $SE(2)$-action on $L_n$ is written as
\[
(a,e^{i\theta})\cdot (e_n)_z
=
e^{in\theta}(e_n)_{a+e^{i\theta}z}.
\]
Since the $SE(2)$-action on $\mathbb{C}$ is holomorphic, it induces a natural
$SE(2)$-action on $T^{1,0}\mathbb{C}$ by push-forward.
For $v_z\in T^{1,0}_z\mathbb{C}$, we define
\[
g\cdot v_z:=(dg)_z(v_z)\in T^{1,0}_{g\cdot z}\mathbb{C}.
\]
Let
\[
g=(a,e^{i\theta})\in SE(2).
\]
Since
\[
g\cdot z=a+e^{i\theta}z,
\]
we have
\[
\frac{d}{dz}(g\cdot z)=e^{i\theta}.
\]
Therefore
\[
g\cdot\left(\frac{\partial}{\partial z}\bigg|_z\right)
=
e^{i\theta}\frac{\partial}{\partial z}\bigg|_{g\cdot z}.
\]
Since $T^{1,0*}\mathbb{C}$ is the dual bundle of $T^{1,0}\mathbb{C}$,
the action on $T^{1,0}\mathbb{C}$ naturally induces an $SE(2)$-action on
$T^{1,0*}\mathbb{C}$.
With respect to the frame $dz$, this induced action is given by
\[
g\cdot (dz)_z
=
e^{-i\theta}(dz)_{g\cdot z}.
\]
It follows from this formula that $T^{1,0*}\mathbb{C}$ is the weight $-1$ $SE(2)$-equivariant holomorphic line bundle over $\mathbb{C}$.
Therefore
\[
T^{1,0*}\mathbb{C}\simeq L_{-1}
\]
as $SE(2)$-equivariant holomorphic line bundles.
The complex conjugate of the above action defines a smooth $SE(2)$-action on
$T^{0,1*}\mathbb{C}$, which is given by
\[
g\cdot (d\overline z)_z
=
e^{i\theta}(d\overline z)_{g\cdot z}.
\]
Thus $T^{0,1*}\mathbb{C}$ is, by construction, the weight $1$
smooth $SE(2)$-equivariant complex line bundle over $\mathbb{C}$.
Therefore
\[
T^{0,1*}\mathbb{C}\simeq L_1
\]
as smooth $SE(2)$-equivariant complex line bundles.
Then, by a similar argument with Lemma \ref{lem 3.8}, the following statement holds.
\begin{lem}\label{lem 3.11}
\begin{itemize}
\item[(1)]
The space $A^{0,1}(L_{-1})^{SE(2)}$ is one-dimensional and is generated by
\[
\alpha=d\overline z\otimes e_{-1}.
\]
\item[(2)]  The space $A^{1,0}(L_{1})^{SE(2)}$ is one-dimensional and is generated by
\[
\beta=dz\otimes e_1.
\]
\end{itemize}
\end{lem}
We next define an $SE(2)$-invariant Hermitian metric on $L_n$.
For every $n\in\mathbb{Z}$, we  define a Hermitian metric $h^{(n)}$ on $L_n$ as
\[
h^{(n)}_z(e_{n,z},e_{n,z}):=1.
\]
Let $g=(a,e^{i\theta})\in SE(2)$.
Since 
\begin{align*}
h^{(n)}_{g\cdot z}(g\cdot e_{n,z},g\cdot e_{n,z})&=e^{in\theta}\overline {e^{in\theta}}\,h^{(n)}_{g\cdot z}( e_{n,g\cdot z},e_{n,g\cdot z})\\
&=1\\
&= h^{(n)}_z(e_{n,z},e_{n,z}),
\end{align*}
$h^{(n)}$ is $SE(2)$-invariant.
We denote by $\partial_{\mathbb{C}}$ the usual $\partial$ on $\mathbb{C}$.
 It is clear that the (1,0)-part of the Chern connection associated with $(\bar\partial_{\mathbb{C}},h^{(n)})$ is independent of $n$ and is given by $\partial_{\mathbb{C}}.$
\par
We next fix our notation for the canonical Riemannian metric and the associated K\"ahler form on $\mathbb{C}.$
Writing $z=x+\sqrt{-1}y,$ we denote by
\[
g_{\mathrm{can}}:=dx\otimes dx+dy\otimes dy
\]
the canonical Euclidean metric on $\mathbb{C}$. 
Its associated K\"ahler form is
\[
\omega_{\mathrm{can}}:=\frac{\sqrt{-1}}{2}dz\wedge d\bar z.
\]
Then, by the definition of $\alpha$ and $\beta$, we have 
\be\label{eq 3.29}
\alpha\wedge\beta=2\sqrt{-1}\omega_{\mathrm{can}}.
\ee
\subsubsection{Dimensional Reduction}
We use the same notation as in the previous section.
Let $X$ be a compact Riemann surface, and $\omega_X$ be a K\"ahler form on $ X$.
We normalize $\omega_X$ so that 
\[
\mathrm{Vol}_{\omega_X}=\int_X\omega_X=1.\]
We now consider the action of $SE(2)$ on $X\times \mathbb{C}$, which is trivial on $X$ and standard on $\mathbb{C}$.
\par
The argument in this section closely parallels that in Section \ref{sec 3.2.2} and \ref{sec 3.3.2}.
However, whereas Section \ref{sec 3.2.2} (resp. Section \ref{sec 3.3.2}) deals with the $SU(2)$ (resp. $SU(1,1)$)-equivariant geometry over $X\times \mathbb{P}^1$ (resp. $X\times\Delta$), we work here with the $SE(2)$-equivariant geometry over $X\times \mathbb{C}$.
\par
In the first half of this section, starting from a solution of Manton's exotic $\tau$-vortex equation with $\tau=0$ on $X$, we construct an $SE(2)$-equivariant holomorphic bundle over $X\times \mathbb{C}$ and an $SE(2)$-invariant Hermitian-Einstein form of signature $(1,1)$ on it.
In the second half, we establish the converse.
\par
We denote by 
\begin{align*}
&p:X\times \mathbb{C}\to X\\
&q:X\times \mathbb{C}\to \mathbb{C}
\end{align*}
the natural projections.\par
Let $(L,\overline\partial_L)$ be a holomorphic line bundle over $X$, $\phi \in A(L)$, and $h_L$ be a Hermitian metric on $L$.
We assume 
$(\overline\partial_L,\phi,h_L)\in \mathrm{Tri}_0$ (See Section \ref{sec 2.1} for the defintion of $\mathrm{Tri}_\tau$).
We construct an $SE(2)$-equivariant holomorphic bundle over $X\times \mathbb{C}$ and an $SE(2)$-invariant pseudo-Hermitian-Einstein metric of signature $(1,1)$ from $(\overline\partial_L,\phi,h_L)$.
We first construct an $SE(2)$-equivariant holomophic bundle over $X\times\mathbb{C}$.
Recall that we denote the trivial line bundle over $X$ as $\underline{\mathbb{C}}_X$.
The Dolbeault operator $\overline\partial_{X\times\mathbb{C}}$ of $X\times\mathbb{C}$ defines a holomorphic bundle structure of $\underline{\mathbb{C}}_{X\times \mathbb{C}}$.
Let $c\in\mathbb{C}$ be a constant.
We define
\begin{align}
E:=&p^\ast L\otimes q^\ast L_{-1}\oplus \underline{\mathbb{C}}_{X\times \mathbb{C}},\nonumber\\
\overline\partial_{E,c}:=&
\begin{pmatrix}
p^\ast \overline\partial_L\otimes \mathrm{Id}+\mathrm{Id}\otimes q^\ast \overline\partial_{\mathbb{C}}& c\cdot p^\ast \phi\otimes q^\ast \alpha\\
0&\overline\partial_{X\times\mathbb{C}}
\end{pmatrix}.\label{eq 3.30}
\end{align}
See Lemma \ref{lem 3.11} for $\alpha$.
We also note that $c\cdot  p^\ast \phi\otimes q^\ast \alpha\in A^{0,1}(\mathrm{Hom}(\underline{\mathbb{C}}_{X\times\mathbb{C}}, p^\ast L\otimes q^\ast L_{-1})$.
The $SE(2)$-action on $L_{-1}$ induces an $SE(2)$-action on $E$, where $SE(2)$ acts trivially on the factors pulled back from $X$.
By the definition of the action of $SE(2)$ on $X\times\mathbb{C}$, $E$ is also an $SE(2)$-equivariant bundle over $X\times \mathbb{C}$.
By the same argument as in Proposition \ref{prop 3.3}, we obtain the following proposition.
\begin{prop}
$(E,\overline\partial_{E,c})$ is an $SE(2)$-equivariant holomorphic bundle over $X\times\mathbb{C}$.
\end{prop}
Hence 
\[
0\longrightarrow p^\ast L\otimes q^\ast L_{-1}\longrightarrow E\longrightarrow \underline{\mathbb{C}}_{X\times\mathbb{C}} \longrightarrow0.
\]
is a holomorphic bundle extension.\par
Let $h^{(-1)}$ be the $SE(2)$-invariant Hermitian metric of $L_{-1}$ (See Section \ref{sec 3.4.1}), and $h_X$ be a Hermitian metric of $\underline{\mathbb{C}}_X$.
We note that $h_X$ is a smooth positive function on $X$ and therefore $h_Xh_L$ is also a Hermitian metric on $L$.
We define a pseudo-Hermitian metric $h_E$ on $E$ as 
\begin{align}\label{eq 3.31}
h_E:=
\begin{pmatrix}
p^\ast (h_Xh_L)\otimes q^\ast h^{(-1)}&0\\
0& -p^\ast h_X
\end{pmatrix}.
\end{align}
It follows from the construction that $h_E$ is a pseudo-Hermitian metric of signature $(1,1)$.
Since $h^{(-1)}$ is $SE(2)$-invariant, and $SE(2)$ action on $E$ is trivial on the factors pulled back from $X$, the pseudo-Hermitian metric $h_E$ is also $SE(2)$-invariant.
\par
We denote by $\partial_{h_Xh_L}$ and $\partial_{h_X}$ the $(1,0)$-part of the Chern connection associated to $(L,\overline\partial_L,h_Xh_L)$ and $(\underline{\mathbb{C}}_X,\overline\partial_X,h_X)$ respectively.
Recall from Section \ref{sec 3.4.1} that the $(1,0)$-part of the Chern connection of $(L_{-1},\overline\partial_{\mathbb C},h^{(-1)})$ is $\partial_{\mathbb C}$.
Let $\nabla_{h_E,c}$ be the unique connection that is compatible with $h_E$ and $\nabla_{h_E,c}^{0,1}=\overline\partial_{E,c}$ (See Proposition \ref{prop 3.1}), and $\partial_{h_E,c}$ the $(1,0)$-part of $\nabla_{h_E}$.
By the same argument as in Lemma \ref{lem 3.5}, we obtain the following Lemma.
\begin{lem}\label{lem 3.12}
$\partial_{h_E,c}$ has the form
\begin{equation}\label{eq 3.32}
\partial_{h_E,c}
=\begin{pmatrix}
p^\ast\partial_{h_Xh_L}\otimes \mathrm{Id} +\mathrm{Id}\otimes p^\ast\partial_{\mathbb C}&0\\
\overline c \cdot p^\ast (h_L\overline \phi)\otimes q^\ast \beta& p^\ast \partial_{h_X}
\end{pmatrix}.
\end{equation}
See Lemma \ref{lem 3.11} for $\beta$.
\end{lem}
We show that for a suitable $c\in \mathbb{C}$ and a suitable K\"ahler form on $X\times  \mathbb{C}$, there exists a Hermitian metric $h_X$ of $\underline{\mathbb{C}}_X$ such that $h_E$ is a pseudo-Hermitian-Einstein metric of $(E,\overline\partial_{E,c})$ (See Definition \ref{def 3.2}).\par
Let $\omega_{\mathrm{can}}$ be the K\"ahler form on $\mathbb{C}$, associated to the canonical Euclidean metric (See Section \ref{sec 3.4.1}).
We define a K\"ahler form of $X\times \mathbb{C}$ as 
\be\label{eq 3.33}
\Omega:=p^\ast\omega_X\oplus q^\ast \omega_{\mathrm{can}}.
\ee
We now set 
\[
c=\frac{1}{2\sqrt{2}}
\]
and consider $(E,\overline\partial_{E,\frac{1}{2\sqrt{2}}})$.
For simplicity, we write
\[
\overline\partial_E:=\overline\partial_{E,\frac{1}{2\sqrt{2}}}.
\]
We also write 

\begin{align*}
\nabla_{h_E}&:=\nabla_{h_E,\frac{1}{2\sqrt{2}}},\\
\partial_{h_E}&:=\partial_{h_E, \frac{1}{2\sqrt{2}}}.
\end{align*}

We denote the curvature of $\nabla_{h_E}$ as
\[
F_{h_E}:=(\nabla_{h_E})^2.
\]
We also denote by $F_{h_Xh_L}$ and $F_{h_X}$ the curvature of the Chern connections of $(L,\overline\partial_L,h_Xh_L)$ and $(\underline{\mathbb{C}}_X,\overline\partial_X, h_X)$ respectively.
The Chern connection of $(L_{-1},\bar\partial_\mathbb{C},h^{-1})$ is flat.\par
The next Proposition is an analogy of Proposition \ref{prop 3.4} and Proposition \ref{prop 3.9}. 
As in Proposition \ref{prop 3.9}, since we are now considering a vector bundle over the non-compact space $X\times\mathbb{C}$, we also need to specify the constant appearing in \eqref{eq 3.1}.
\begin{prop}\label{prop 3.14}
There exists a Hermitian metric $h_X$ of $\underline{\mathbb{C}}_X$ such that the pseudo-Hermitian metric $h_E$ defined by \eqref{eq 3.31} is a pseudo-Hermitian-Einstein metric with respect to the K\"ahler form $\Omega$ and the constant $\pi\mathrm{deg}L$;
that is
\be
\sqrt{-1}\Lambda_{\Omega}F_{h_E}=\pi\mathrm{deg}L\cdot \mathrm{Id}_E.
\ee
Equivalently, the connection $\nabla_{h_E}$ is a pseudo-Hermitian-Yang-Mills connection.
\end{prop}
\begin{proof}
For brevity, we set
\[
\lambda:=\pi\mathrm{deg}L.
 \]
 \par
Let $h_X$ be a Hermitian metric on $\mathbb{C}_X$ and let $h_E$ be a pseudo-Hermitian metric of signature (1,1) defined as \eqref{eq 3.31}.
Since $(\overline\partial_E)^2=(\partial_{h_E})^2=0$ we have
\[
F_{h_E}=\overline\partial_E\partial_{h_E}+\partial_{h_E}\overline\partial_E.
\]
We set
\[
F_{h_E}=
\begin{pmatrix}
F_{11}&F_{12}\\
F_{21}&F_{22}
\end{pmatrix}.
\]
Then by \eqref{eq 3.30} and \eqref{eq 3.32} we have 
\[
\begin{split}
F_{11}&=p^\ast F_{h_Xh_L}+\frac{1}{8}p^\ast|\phi|^2_{h_L}\otimes q^\ast(\alpha\wedge \beta)\\
&=p^\ast F_{h_L}+p^\ast F_{h_X}+\frac{\sqrt{-1}}{4}p^\ast|\phi|^2_{h_L}q^\ast\omega_{\mathrm{can}}.
\end{split}
\]
The last equation follows since $h_X$ is a positive function on $X$, the Chern connection of $(L_{-1},\overline\partial_\mathbb{C},h^{-1})$ being flat, and \eqref{eq 3.29}.
Therefore we have 
\be\label{eq 3.35}
\begin{split}
\sqrt{-1}\Lambda_{\Omega}F_{11}&=\sqrt{-1}p^\ast \Lambda_{\omega_X}(F_{h_L}+F_{h_X})-\frac{1}{4}p^\ast|\phi|^2_{h_L}.
\end{split}
\ee
We also have 
\[
\begin{split}
F_{22}&=p^\ast F_{h_X}+\frac{1}{8}p^\ast|\phi|^2_{h_L}\otimes q^\ast(\beta\wedge \alpha)\\
&=p^\ast F_{h_X}-\frac{\sqrt{-1}}{4}p^\ast|\phi|^2_{h_L}q^\ast \omega_{\mathrm{can}}.
\end{split}
\]
Ttherefore
\be\label{eq 3.36}
\sqrt{-1}\Lambda_{\Omega}F_{22}=\sqrt{-1}p^\ast \Lambda_{\omega_X}F_{h_X}+\frac{1}{4}p^\ast|\phi|^2_{h_L}.
\ee
Since $F_{12},F_{21}$ has mixed contribution from $X$ and $\mathbb{C}$ we have

\[
\sqrt{-1}\Lambda_{\Omega}F_{h_E}
=
\begin{pmatrix}
\sqrt{-1}\Lambda_{\Omega}F_{11}&0\\
0&\sqrt{-1}\Lambda_{\Omega}F_{22} \\
\end{pmatrix}.
\]
Hence by \eqref{eq 3.35} and \eqref{eq 3.36} the equation
\[
\sqrt{-1}\Lambda_{\Omega}F_{h_E}=\lambda Id_E
\]
is equivalent to 
\begin{align}
\sqrt{-1} \Lambda_{\omega_X}(F_{h_L}+F_{h_X})-\frac{1}{4}|\phi|^2_{h_L}&=\lambda,\label{eq 3.37}
\\
\sqrt{-1}\Lambda_{\omega_X}F_{h_X}+\frac{1}{4}|\phi|^2_{h_L}&=\lambda.\label{eq 3.38}
\end{align}
The equations \eqref{eq 3.37} and \eqref{eq 3.38} are equivalent to 

\begin{align}
\sqrt{-1} \Lambda_{\omega_X}F_{h_L}-\frac{1}{2}|\phi|^2_{h_L}&=0,\label{eq 3.39}
\\
\sqrt{-1} \Lambda_{\omega_X}F_{h_L}+2\sqrt{-1} \Lambda_{\omega_X}F_{h_X}&=2\lambda,\label{eq 3.40}
\end{align}

since \eqref{eq 3.39} is obtained by \eqref{eq 3.37}-\eqref{eq 3.38} and \eqref{eq 3.40} is obtained by \eqref{eq 3.37}+\eqref{eq 3.38}.
\par
\eqref{eq 3.39} holds since $((L,\overline\partial_L),\phi,h_L)$ satisfies Manton's exotic vortex equation with $\tau=0$.
We show that there exists a Hermitian metric $h_X$ such that \eqref{eq 3.40} holds.
By the same argument as in Proposition \ref{prop 3.4}, the existence of such a metric is equivalent to 
\[
\int_{X}\bigg(2\lambda-\sqrt{-1} \Lambda_{\omega_X}F_{h_L}\bigg)=0.
\]
Then, by the definition of $\lambda$, we have 
\begin{align*}
\int_{X}\bigg(2\lambda-\sqrt{-1} \Lambda_{\omega_X}F_{h_L}\bigg)&=2\lambda-2\pi\mathrm{deg}L\\
&=0.
\end{align*}
Hence, the claim is proved.
\end{proof}

We give the summary of the results so far.

\begin{thm}\label{thm 3.5}
Let $X$ be a compact connected Riemann surface and let  $\omega_X$ be a K\"ahler form of $X$ satisfying $\int_X\omega_X=1$.
Let $(L,\overline\partial_L)$ be a holomorphic line bundle over $X$, $\phi$ a holomorphic section  of $(L,\overline\partial_L)$, and $h_L$ a Hermitian metric on $L$.
Suppose the triple $((L,\overline\partial_L),\phi,h_L)$ satisfies Manton's exotic $\tau$-vortex equation with $\tau=0$.
Then 
\begin{align*}
E:=&p^\ast L\otimes q^\ast L_{-1}\oplus \underline{\mathbb{C}}_{X\times\mathbb{C}},\\
\overline\partial_{E}:=&
\begin{pmatrix}
p^\ast \overline\partial_L\otimes \mathrm{Id}+\mathrm{Id}\otimes q^\ast \overline\partial_{\mathbb{C}}& \frac{1}{2\sqrt{2}}\cdot p^\ast \phi\otimes q^\ast \alpha\\
0&\overline\partial_{X\times\mathbb{C}}
\end{pmatrix}.
\end{align*}
is an $SE(2)$-equivariant holomorphic bundle over $X\times\mathbb{C}$.
We fix a K\"ahler form on $X\times\mathbb{C}$ given by
\[
\Omega=p^\ast\omega_X\oplus q^\ast \omega_{\mathrm{can}}.
\]
Then there exists a Hermitian metric $h_X$ on $\underline{\mathbb{C}}_X$ such that the pseudo-Hermitian metric of signature (1,1)
\begin{align*}
h_E=
\begin{pmatrix}
p^\ast (h_Xh_L)\otimes q^\ast h^{(-1)}&0\\
0& -p^\ast h_X
\end{pmatrix}
\end{align*}
is a pseudo-Hermitian-Einstein metric of $(E,\bar\partial_E)$ with respect to the K\"ahler form $\Omega$ and the constant $\pi\mathrm{deg}L$.
Equivalently, the unique connection $\nabla_{h_E}$ which is compatible with $h$ and $\nabla^{0,1}_{h_E}=\bar\partial_E$ is a pseudo-Hermitian-Yang-Mills connection.
\end{thm}
We next prove the dimensional reduction: we construct a solution of Manton's exotic $\tau$-vortex equation with $\tau=0$ from an $SE(2)$-equivariant holomorphic bundle and an $SE(2)$-invariant pseudo-Hermitian-Einstein metric of signature $(1,1).$\par
Let $(L,\overline\partial_L)$ be a holomorphic line bundle over $X$.
Let $E$ be a holomorphic bundle on $X\times\Delta$ that is an extension of $\underline{\mathbb{C}}_{X\times \mathbb{C}}$ by $p^\ast L \otimes q^\ast L_{-1}$.
Thus, there exists a short exact sequence of holomorphic vector bundles
\[
0\longrightarrow p^\ast L \otimes q^\ast L_{-1}\longrightarrow E\longrightarrow \underline{\mathbb{C}}_{X\times \mathbb{C}}\longrightarrow 0.
\]
We also have
\[
E\simeq p^\ast L \otimes q^\ast L_{-1}\oplus\underline{\mathbb{C}}_{X\times\mathbb{C}}
\]
as smooth bundles.
From now on, we identify $E$ and $p^\ast L \otimes q^\ast L_{-1}\oplus \underline{\mathbb{C}}_{X\times \mathbb{C}}$.
Let $\overline\partial_E$ be the Dolbeault operator of $E$.
Since $E$ given by an extension, $\overline\partial_E$ can be written in the form of 
\[
\bar\partial_E=
\begin{pmatrix}
p^\ast \overline\partial_L\otimes \mathrm{Id}+\mathrm{Id}\otimes q^\ast \overline\partial_{\mathbb{C}}& \psi\\
0&\overline\partial_{X\times\mathbb{C}}
\end{pmatrix}
\]
where $\psi\in A^{0,1}(\mathrm{Hom}(\underline{\mathbb{C}}_{X\times \mathbb{C}}, p^\ast L \otimes q^\ast L_{-1})$.
\par
The bundle $E$ is $SE(2)$-equivariant with respect to the $SE(2)$-action induced by that on $L_{-1}$.
The following is an analogue of Proposition \ref{prop 3.5} and Proposition \ref{prop 3.10} and can be proved by the same argument.
\begin{prop}\label{prop 3.15}
If $(E,\bar\partial_E)$ is an $SE(2)$-equivariant holomorphic bundle, then
\begin{itemize}
\item[(1)] there exists a $\phi\in A(L)$ such that 
\[
\psi=p^\ast\phi\otimes q^\ast\alpha.
\]
\item[(2)] $\phi$ is a holomorphic section of $(L,\bar\partial_L)$.
\end{itemize}
\end{prop}
Since multiplying a holomorphic section by a nonzero constant preserves holomorphicity, $2\sqrt{2}\phi$ is also holomorphic.
Thus, after replacing $2\sqrt{2}\phi$ by $\phi$, we may write
\[
\psi
=
\frac{1}{2\sqrt{2}}\cdot p^\ast\phi\otimes q^\ast\alpha.
\]
We keep the normalization of $\alpha$ fixed.\par
As in Proposition \ref{prop 3.6} and Proposition \ref{prop 3.11}, any $SE(2)$-invariant pseudo-Hermitian metric on $E$ is of the following form. 
The proof is the same as that of Proposition \ref{prop 3.6}.

\begin{prop}\label{prop 3.16}
Let $h_E$ be an $SE(2)$-invariant pseudo-Hermitian metric on $E$.
Then there exists a Hermitian metrics $h_L$ of $L$, $h_X$ of $\underline{\mathbb{C}}_X$ such that
\[
h_E
=
\begin{pmatrix}
C_0\cdot p^\ast h_L\otimes q^\ast h^{(-1)}&0\\
0& C_1\cdot p^\ast h_X
\end{pmatrix},\,\,C_0,C_1\in\{1,-1\}.
\]
\end{prop}

The following is a consequence of Proposition \ref{prop 3.16} and a rescaling of the Hermitian metric on $L$.
\begin{lem}\label{lem 3.13}
Let $h_E$ be an $SE(2)$-invariant pseudo-Hermitian metric of signature $(1,1)$ on $E$.
Then there exist Hermitian metrics $h_L$ on $L$ and $h_X$ on $\underline{\mathbb{C}}_X$ such that 
\[
h_E
=
\begin{pmatrix}
C_0\cdot p^\ast (h_Xh_L)\otimes q^\ast h^{(-1)}&0\\
0& C_1\cdot p^\ast h_X
\end{pmatrix}
\]
where $(C_0,C_1)=(1,-1)$ or $(-1,1)$.
\end{lem} 

We now establish the dimensional reduction.
Throughout this argument, we work with the K\"ahler form $\Omega$ constructed in \eqref{eq 3.33}.
\par
Suppose that $\overline\partial_E$ defines an $SE(2)$-equivariant holomorphic bundle structure on $E$.
Then by Proposition \ref{prop 3.15}, there exists a holomorphic section $\phi$ of $(L,\bar\partial_L)$ such that 
\[
\overline\partial_E=
\begin{pmatrix}
p^\ast \overline\partial_L\otimes\mathrm{Id}+\mathrm{Id}\otimes q^\ast \overline\partial_{\Delta}& \frac{1}{2\sqrt{2}}\cdot p^\ast\phi\otimes q^\ast\alpha\\
0&\overline\partial_{X\times\mathbb{C}}
\end{pmatrix}.
\]
Let $h_E$ be as in Lemma \ref{lem 3.13}, and let $\nabla_{h_E}$ be the connection compatible with $h_E$ and satisfying $\nabla^{0,1}_{h_E}=\bar\partial_E$.
Then, by Lemma \ref{lem 3.12}, $\partial_{h_E}=\nabla^{1,0}_{h_E}$ has the form 
\[
\partial_{h_E}=
\begin{pmatrix}
p^\ast\partial_{h_Xh_L}\otimes\mathrm{Id} +\mathrm{Id}\otimes p^\ast\partial_{h^{(-2)}}&0\\
\frac{1}{2\sqrt{2}}\cdot p^\ast (h_L\overline \phi)\otimes q^\ast \beta& p^\ast \partial_{h_X}
\end{pmatrix}.
\]
We denote the curvature of $\nabla_{h_E}$ as $F_{h_E}$.
The following proposition is an analogue of Proposition \ref{prop 3.7} and Proposition \ref{prop 3.12}.
As in Proposition \ref{prop 3.12}, the constant in the Hermitian-Yang-Mills equation is not determined by the topological data of the bundle.
Nevertheless, we show that this constant is forced to take a specific value.

\begin{prop}
If $h_E$ is a pseudo-Hermitian-Einstein metric of $(E,\bar\partial_E)$ with respect to the K\"ahler form $\Omega$, that is,
\[
\sqrt{-1}\Lambda_{\Omega}F_{h_E}=\lambda \mathrm{Id}_E
\]
for some constant $\lambda\in\mathbb{C}$, then 
\begin{itemize}
\item[(1)] The triple $((L,\overline\partial_L),\phi,h_L)$ satisfies Manton's exotic $\tau$-vortex equation with $\tau=0$.
\item[(2)]
\[
\lambda=\pi\mathrm{deg}L.
\]
\end{itemize}
\end{prop}
\begin{proof}
We set 
\[
F_{h_E}=
\begin{pmatrix}
F_{11}&F_{12}\\
F_{21}&F_{22}.
\end{pmatrix}
\]
The equaiton
\[
\sqrt{-1}\Lambda_{\Omega}F_{h_E}=\lambda\mathrm{Id}_E
\]
implies 
\begin{align*}
\sqrt{-1}\Lambda_{\Omega}F_{11}&=\lambda,\\
\sqrt{-1}\Lambda_{\Omega}F_{22}&=\lambda.
\end{align*}
Then, by \eqref{eq 3.37}, \eqref{eq 3.38}, \eqref{eq 3.39}, and \eqref{eq 3.40}, 
we have 
\begin{align*}
\sqrt{-1} \Lambda_{\omega_X}(F_{h_L}+F_{h_X})-\frac{1}{4}|\phi|^2_{h_L}&=\lambda,
\\
\sqrt{-1}\Lambda_{\omega_X}F_{h_X}+\frac{1}{4}|\phi|^2_{h_L}&=\lambda.
\end{align*}
Subtracting the second equation from the first, we obtain
\[
\sqrt{-1} \Lambda_{\omega_X}F_{h_L}-\frac{1}{2}|\phi|^2_{h_L}=0.
\]
Hence the triple $((L,\overline\partial_L),\phi,h_L)$ is a solution of Manton's exotic $\tau$-vortex equation with $\tau=0$.\par
In addition, adding the two equations, we obtain 
\be\label{eq 3.41}
\sqrt{-1} \Lambda_{\omega_X}F_{h_L}+2\sqrt{-1} \Lambda_{\omega_X}F_{h_X}=2\lambda.
\ee
Recall that 
\[
\sqrt{-1} \Lambda_{\omega_X}F_{h_X}=\frac{1}{2}\triangle h_X.
\]
Hence the equation \eqref{eq 3.41} is equivariant to 
\[
\sqrt{-1} \Lambda_{\omega_X}F_{h_L}+\triangle h_X=2\lambda.
\]
This equation implies 
\begin{align*}
0&=\int_X\bigg(2\lambda-\sqrt{-1} \Lambda_{\omega_X}F_{h_L}\bigg)\\
&=2\lambda-2\pi\mathrm{deg}L
\end{align*}
which derive (2).
\end{proof}
Combining the results obtained above, we obtain the following dimensional reduction theorem.
\begin{thm}
Let $X$ be a compact connected Riemann surface with a K\"ahler form $\omega_X$ of $\int_X \omega_X=1$.
Let $(L,\overline\partial_L)$ be a holomorphic line bundle over $X$.
Let $E$ be a holomorphic bundle over $X\times \mathbb{C}$ which is an extension of $\underline{\mathbb{C}}_{X\times \mathbb{C}}$ by $p^\ast L\otimes q^\ast L_{-1}$.
Let $\overline\partial_E$ be the Dolbeault operator of $E$.
On $X\times\mathbb{C}$, we consider the K\"ahler form $\Omega$ constructed in \eqref{eq 3.33}.
\par
If $(E,\overline\partial_E)$ is an $SE(2)$-equivariant holomorphic bundle, then there exists a holomorphic section $\phi$ of $(L,\bar\partial_L)$ such that 
\[
\overline\partial_E=
\begin{pmatrix}
p^\ast \overline\partial_L\otimes\mathrm{Id}+\mathrm{Id}\otimes q^\ast \overline\partial_{\Delta}& \frac{1}{2\sqrt{2}}\cdot p^\ast\phi\otimes q^\ast\alpha\\
0&\bar\partial_{X\times\mathbb{C}}
\end{pmatrix}.
\]
Let $h_E$ be a pseudo-Hermitian metric of signature $(1,1)$ on $E.$
If $h_E$ is $SE(2)$-invariant, then there exist Hermitian metrics $h_L$ on $L$ and $h_X$ on $\underline{\mathbb{C}}_X$ such that 
\[
h_E
=
\begin{pmatrix}
C_0\cdot p^\ast (h_Xh_L)\otimes q^\ast h^{(-1)}&0\\
0& C_1\cdot p^\ast h_X
\end{pmatrix}
\]
where $(C_0,C_1)=(1,-1)$ or $(-1,1)$.\par
Let $h_E$ and $(E,\bar\partial_E)$ be as above, and $\nabla_{h_E}$ be the connection compatible with $h_E$ and $\nabla^{0,1}_{h_E}=\overline\partial_E$.
If $h_E$ is a pseudo-Hermitian-Einstein metric of $(E,\overline\partial_E)$ with respect to $\Omega$, or equivalently, if $\nabla_{h_E}$ is a pseudo-Hermitian-Yang-Mills connection, then
\begin{itemize}
\item The triple $((L,\overline\partial_L),\phi,h_L)$ satisfies Manton's exotic $\tau$-vortex equation with $\tau=0$.
\item The constant $\lambda$ appearing in \eqref{eq 3.1} satisfies
\[
\lambda=\pi\mathrm{deg}L.
\]
\end{itemize}
\end{thm}


\begin{thebibliography}{99}
\bibitem{Akh} D. N. Akhiezer, Lie group actions in complex analysis. Aspects of Mathematics, E27. Friedr. Vieweg $\&$ Sohn, Braunschweig, 1995.
\bibitem{AmOl1} J. Ambj{\o}rn, P. Olesen, Anti-screening of large magnetic fields by vector bosons, Phys. Lett. \textbf{B214}, 565 (1988).
\bibitem{AmOl2} J. Ambj{\o}rn, P. Olesen,  On electroweak magnetism, Nucl. Phys. \textbf{B} 315, 606 (1989).
\bibitem{Br1}  S. B. Bradlow, Vortices in holomorphic line bundles over closed K\"ahler manifolds. Commun. Math. Phys. \textbf{135} (1990).
\bibitem{BrG} S. B. Bradlow and O. Garc\'ia-Prada, Stable triples, equivariant bundles and dimensional reduction, Math. Ann. \textbf{304} (1996).
\bibitem{Bry} R. L. Bryant, On the geometry of almost complex 6-manifolds, Asian J. Math. \textbf{10} (2006).
\bibitem{CD} F. Contatto, M. Dunajski, Manton's five vortex equations from self-duality, J. Phys. A: Math. Theor. \textbf{50} 375201
\bibitem{DK}  S. K. Donaldson, P. B. Kronheimer, The geometry of four-manifolds, Oxford Mathematical Monographs. The Clarendon Press, Oxford University Press, New York, (1990), 440 pp.
\bibitem{GP1} O. Garc\'ia-Prada, Invariant connections and vortices, Commun. Math. Phys. \textbf{156} (1993).
\bibitem{GP2} O. Garc\'ia-Prada, Dimensional reduction of stable bundles, vortices and stable pairs, Int. J. Math. \textbf{5} (1994).
\bibitem{GP3} O. Garc\'ia-Prada, A direct existence proof for the vortex equations over a compact Riemann surface. Bulletin of the London Mathematical Society \textbf{26}.1 (1994): 88-96.
\bibitem{GL} V. L. Ginsburg, L. D. Landau, Zh. Eksp. Theor. Fiz. \textbf{20}, 1064 (1950)
\bibitem{Ha} B. C. Hall, Lie Groups, Lie Algebras, and Representations: An Elementary Introduction, Springer, 2015.
\bibitem{JaPi1} R. Jackiw, S.-Y. Pi, Soliton solutions to the gauged nonlinear Schr\"odinger equation on the plane, Phys. Rev. Lett. \textbf{64}, 2969 (1990).
\bibitem{JaPi2} R. Jackiw, S.-Y. Pi, Classical and quantum nonrelativistic Chern-Simons theory, Phys. Rev. \textbf{D42}, 3500 (1990).
\bibitem{JT} A. Jaffe, C. H. Taubes, Vortices and Monopoles, Progress in Physics 2, Boston: Birkhauser 1980 
\bibitem{KN} S. Kobayashi K. Nomizu, “Foundations of Differential Geometry. Volume I,” John Wiley and Sons (1996)
\bibitem{KW} J. L. Kazdan, F. Warner, Curvature functions for compact 2-manifolds, Ann. of Math. (2) 99 (1974).
\bibitem{Ko}  S. Kobayashi, Differential geometry of complex vector bundles, Publications of the MSJ, 15. Kan\^o Memorial Lectures, 5. Princeton University Press, Princeton, NJ;
Iwanami Shoten, Tokyo, 1987
\bibitem{LT} M. L\"ubke, A. Teleman, The Kobayashi-Hitchin correspondence, World Scientific, Singapore (1995).
\bibitem{Ma} N. S. Manton, Five Vortex Equations J. Phys. A \textbf{50}, 125403 (2017).
\bibitem{Nico} L. I. Nicolaescu, Notes on Seiberg-Witten theory. Vol. 28. Providence, RI: American Math
\bibitem{Nico 1} L. I. Nicolaescu, Notes on the Atiyah-Singer Index Theorem. Notes for a Topics in Topology Course, University of Notre Dame (2013)
\bibitem{Par} T. H. Parker,  Gauge theories on four dimensional Riemannian manifolds. Commun.Math. Phys. \textbf{85} (1982).
\bibitem{Po1} A. D. Popov, Hermitian Yang-Mills equations and pseudo-holomorphic bundles on nearly K\"ahler and nearly Calabi-Yau twistor 6-manifolds, Nuclear Phys. B \textbf{828} (2010).
\bibitem{Po2} A. D. Popov, Integrable vortex-type equations on the two-sphere, Phys. Rev. D \textbf{86} (2012).
\bibitem{R} C. Ross, Cartan connections and integrable vortex equations. Journal of Geometry and Physics \textbf{179} (2022).
\bibitem{Se} G. Segal, Equivariant K-theory, Inst. Hautes \'Etudes Sci. Publ. Math, (34), 1968.
\bibitem{Tdl}  F. Torres de Lizaur, Vortex Equations and Non-Compact Classical Yang-Mills Theories, 2013.
\bibitem{Wa} F. W. Warner, Foundations of differentiable manifolds and Lie groups. Vol. 94. Springer Science and Business Media, 1983.
\bibitem{Wells} R. O. Wells, Jr., Differential analysis on complex manifolds, volume 65 of Graduate Texts in Mathematics. Springer-Verlag, New York-Berlin, second edition, 1980.
\end{thebibliography}
\end{document}